\documentclass[aop,preprint]{imsart}

\RequirePackage[OT1]{fontenc}
\RequirePackage{amsthm,amsmath}
\RequirePackage[numbers]{natbib}
\RequirePackage[colorlinks,citecolor=blue,urlcolor=blue]{hyperref}

\startlocaldefs

\usepackage{amsmath}
\usepackage{amsfonts}
\usepackage{amssymb}
\usepackage{mathrsfs}
\usepackage{amsthm}
\usepackage{latexsym}
\usepackage{mathrsfs}
\usepackage{amsxtra}
\usepackage{amscd}
\usepackage{mathtools}
\usepackage{verbatim}
\usepackage{enumerate}
\usepackage{color}

\numberwithin{equation}{section}
\theoremstyle{plain}
\newtheorem{Theorem}{Theorem}[section]
\newtheorem{Definition}[Theorem]{Definition}
\newtheorem{Proposition}[Theorem]{Proposition}

\newtheorem{Assumption}[Theorem]{Assumption}

\newtheorem{Lemma}[Theorem]{Lemma}
\newtheorem{Corollary}[Theorem]{Corollary}
\newtheorem{Remark}[Theorem]{Remark}

\renewcommand{\epsilon}{\varepsilon}

\newcommand{\tr}{\operatorname{Tr}}

\newcommand{\defeq}{\coloneqq }

\newcommand{\cpol}{C_{\mathrm{Pol}}}

\def \E{\mathbb{E}}

\def \P{\mathbb{P}}

\def \R{\mathbb{R}}

\def \N{\mathbb{N}}

\def\Ac{{\cal A}}

\def\Cc{{\cal C}}

\def\Fc{{\cal F}}

\def\Jc{{\cal J}}

\def\Lc{{\cal L}}

\def\Tc{{\cal T}}
\def\Uc{{\cal U}}

\def\Acu{\overline{\Ac}}
\def\Acb{\underline{\Ac}}

\def \eps{\varepsilon}

\def \0{\mathbf{0}}
\def \H{\textsc{h}}

\def\x{\times}

\def\1{{\bf 1}}

\endlocaldefs

\makeatletter
\let\@fnsymbol\@arabic
\makeatother

\begin{document}

\begin{frontmatter}
\title{Path-dependent equations and viscosity solutions in infinite dimension\thanksref{T1}}
\runtitle{Path-dependent equations in infinite dimension}
\thankstext{T1}{This research has been partially supported by the Italian PRIN project ``Problemi differenziali di evoluzione: approcci deterministici e stocastici e loro interazioni'' and by the GNAMPA project ``Equazioni stocastiche con memoria e applicazioni'' (2014).}

\begin{aug}
\author{\fnms{Andrea} \snm{Cosso}\ead[label=e1]{andrea.cosso@polimi.it}},
\author{\fnms{Salvatore} \snm{Federico}\thanksref{t1}\ead[label=e2]{salvatore.federico@unisi.it}},
\author{\fnms{Fausto} \snm{Gozzi}\ead[label=e3]{fgozzi@luiss.it}},\\
\author{\fnms{Mauro} \snm{Rosestolato}\thanksref{t2}\ead[label=e4]{mauro.rosestolato@polytechnique.edu}},
\and
\author{\fnms{Nizar} \snm{Touzi}\thanksref{t3}\ead[label=e5]{nizar.touzi@polytechnique.edu}
}

\thankstext{t1}{Corresponding author.}
\thankstext{t2}{Supported by  Deutscher Akademischer Austauschdienst.}
\thankstext{t3}{Supported by the ERC Advanced Grant 321111, and by the Chairs {\it Financial Risk} and {\it Finance and Sustainable Development}.}
\runauthor{A.~COSSO, S.~FEDERICO, F.~GOZZI, M.~ROSESTOLATO, AND N.~TOUZI}

\affiliation{Politecnico di Milano, Universit\`a di
  Siena, LUISS University,\\ \'Ecole
  Polytechnique, and \'Ecole Polytechnique}

\address{
  Dipartimento di Matematica\\
  Politecnico di Milano\\
  via Bonardi 9\\
  20133 Milano\\
  Italy\\
  \printead{e1}\\
}

\address{Dipartimento di Economica Politica e Statistica\\
  Universit\`a di Siena\\
  P.zza S.\ Francesco 7--8
  \\53100 Siena\\
  Italy\\
  \printead{e2}}

\address{Dipartimento di Economia e Finanza\\
  Libera Universit\`a degli Studi Sociali ``Guido Carli'' \\
  Viale Romania 32\\
  00197 Roma\\
  Italy\\
  \printead{e3}}

\address{CMAP\\
  Ecole Polytechnique\\
  91128 Palaiseau\\
  France\\
  \printead{e4}\\
  \phantom{E-mail:\ }\printead*{e5}}
\end{aug}

\begin{abstract}
 Path-dependent PDEs (PPDEs) are natural objects to study when one deals with non Markovian models. Recently, after the introduction  of the so-called pathwise (or functional or Dupire) calculus (see \cite{Dupire}), in the case of finite-dimensional underlying space
various papers have been devoted to studying the well-posedness of such kind of equations, both from the point of view of regular
solutions (see e.g.\ \cite{Dupire,CF16}) and viscosity solutions
(see e.g.\ \cite{EkrenKellerTouziZhang}).
In this paper, motivated by the study of models driven by path-dependent stochastic PDEs, we give a first well-posedness result for viscosity solutions of PPDEs when the underlying space is a
separable 
 Hilbert space.
 We also observe that, 
in contrast with
the finite-dimensional case, our well-posedness result, even in the Markovian case, applies to equations which
cannot be treated, up to now, with the known theory of viscosity solutions.
\end{abstract}

\begin{keyword}[class=MSC]
 \kwd{35D40}
 \kwd{35R15}
 \kwd{60H15}
 \kwd{60H30}.
\end{keyword}

\begin{keyword}
\kwd{Viscosity solutions}
\kwd{path-dependent stochastic differential equations}
\kwd{path-dependent partial differential equations}
\kwd{partial differential equations in infinite dimension}.
\end{keyword}

\end{frontmatter}

\section{Introduction}
Given $T>0$ and a real separable Hilbert space $H$,  let $C([0,T];H)$ be the Banach space of continuous functions from $[0,T]$ to $H$, endowed with the supremum norm $|\mathbf x|_\infty \defeq  \sup_{t\in[0,T]}|\mathbf x_t|$, for all $\mathbf x\in C([0,T];H)$. Let $\Lambda\coloneqq [0,T]\times C([0,T];H)$ and consider the following pseudometric on $\Lambda$:
\[
\textbf{d}_\infty\big((t,\mathbf x),(t',\mathbf x')\big)  \coloneqq   |t-t'|+|\mathbf x_{.\wedge t}-\mathbf x'_{.\wedge t'}|_\infty, \qquad (t,\mathbf x),\ (t',\mathbf x')\in\Lambda.
\]
The pseudo-metric $\mathbf{d}_\infty$ allows to account for the \emph{non-anticipativity} condition: each  function $v\colon (\Lambda,\mathbf d_\infty) \rightarrow E$,  where $E$ is a Banach space, which is  measurable with respect to the Borel $\sigma$-algebra induced by $\textbf{d}_\infty$, is such that  $v(t,\mathbf x)=v(t,\mathbf x_{\cdot\wedge t})$ for all $(t,\mathbf x)\in\Lambda$. Let $A$ be the generator of a strongly continuous semigroup on $H$, and  let $b\colon\Lambda\rightarrow H$, $\sigma\colon\Lambda\rightarrow L(K;H)$, where $K$ is another
real separable Hilbert space (the noise space) 
and $L(K;H)$ is the vector space of linear and continuous functions $K\rightarrow H$. In this paper, we study the well-posedness of the following infinite-dimensional path-dependent partial differential equation (PPDE):
{\small{\begin{equation}
\label{PPDE-intro}
- \partial_t u - \langle A\mathbf x_t,\partial_{\mathbf x}u\rangle - \langle b(t,\mathbf x), \partial_{\mathbf x}u\rangle - \frac{1}{2}\text{Tr}\big[\sigma(t,\mathbf x)\sigma^*(t,\mathbf x)\partial_{\mathbf x\mathbf x}^2 u\big] - F(t,\mathbf x, u)  =  0,
\end{equation}}}
for all $t\in[0,T)$ and $\mathbf x\in C([0,T];H)$,
where $F\colon\Lambda\times\R \rightarrow \R$ and $\partial_t u$, $\partial_{\mathbf x}u$, $\partial_{\mathbf{xx}}^2u$ denote formally  the so-called pathwise (or functional or Dupire, see \cite{CF10b, CF13, Dupire}) derivatives. The unknown is a non-anticipative functional $u\colon\Lambda\rightarrow\R$. We are deliberately restricting the nonlinearity $F$ to depend only on $u$, and not on
$\partial_{\mathbf x}u$, in order to focus on our main well-posedness objective.
The treatment of nonlinearities involving the derivatives (e.g., the case of Hamilton-Jacobi-Bellman equations) needs different methods and involves non-trivial technical difficulties. For this reason, we leave a complete treatment of it for future research.
Nevertheless, in order to illustrate how the definition of viscosity solution here adopted can be
extended to the nonlinear case, in
 Section~\ref{S:Extension} 
we provide
a setting for the treatment of HJB equations
 and show existence
of viscosity solutions.
Then
we  specify  the steps 
needed to obtain uniqueness, 
assuming those results which are non-trivial and whose study is not among the aims of this paper.

{We emphasize that,} in addition to the infinite-dimensional feature of the equation \eqref{PPDE-intro},  coefficients $b,\sigma,$ and $F$ are path-dependent.  Such a path-dependency may be addressed with a standard PDE approach by introducing a ``second level'' of infinite-dimensionality, {that is} embedding the state space $H$ in a larger infinite-dimensional space, like $L^2(-T,0;H)$ and converting equation \eqref{PPDE-intro} into a PDE on this larger space (see e.g., in the context of \emph{delay equations} and when the original space $H$ is finite-dimensional, \cite{Chojnowska78}, \cite[Ch.\,10]{DPZ02}, or \cite[Sec.\ 2.6.8]{FGS14}).
The latter methodology turns out to be problematic when the data, as in our case, are required to have continuity properties with respect to the supremum norm, as the PDE should be considered 
in
the space of continuous functions,
a non-reflexive Banach space.
 Indeed most of the results on well posedness of infinite dimensional PDEs are proved when the underlying space is a Hilbert space (this is, in particular, the case for the viscosity solutions theory, see e.g.\ \cite[Ch.\,3]{FGS14}).
However, we should mention that some attempts have been made along this direction, we refer to \cite{DGR1, DGR2, Ffinsto, FZ, FMT, MasieroEJP07,Roses2016b}.

When the space $H$ is finite-dimensional, PPDEs with a structure more general than  \eqref{PPDE-intro} have been investigated by means of a new concept of viscosity solution recently introduced in  \cite{EkrenKellerTouziZhang}, and further developed in \cite{EkrenTouziZhang1,EkrenTouziZhang2,RTZ14}. This new notion enlarges the class  of test functions, by defining the smoothness only ``with respect to the dynamics'' of the underlying stochastic system and requiring the usual ``tangency condition'' --- required locally pointwise in the standard viscosity definition --- only in mean. These two weakenings, on { the} one hand, keep safe the existence of solutions; on the other hand, simplify a lot the proof of uniqueness, as this does not require  anymore the passage through the Crandall-Ishii Lemma.

The main objective of this paper is to extend to our  infinite-dimensional
path-dependent context such new notion of viscosity solution.
Before illustrating our results, we recall that, for equation like (\ref{PPDE-intro}), when all coefficients are Markovian, results on existence and uniqueness of classical solutions (that can be found e.g.\ in \cite[Ch.\ 7]{DPZ02})
are much weaker than in the finite-dimensional case, due to the lack of local compactness and to the absence of a reference measure like the Lebesgue { measure}.
This makes quite relevant the notion of viscosity solution, introduced in the infinite-dimensional case by \cite{Lio1,Lio2,Lio3}, see also \cite{S94} and, for a survey, \cite[Ch.\,3]{FGS14}. The infinite-dimensional extension of the usual notion of viscosity solution to these PDEs is not trivial, as the comparison results are established only under 
non-standard
 continuity assumptions on the coefficients (needed to generate maxima and minima) and under a nuclearity condition on the diffusion coefficient $\sigma$. The latter  purely technical condition is a methodological bound of this notion of viscosity solutions{: it is  needed to} adapt the Crandall-Ishii Lemma to the infinite-dimensional context.

The core results of the present paper (contained in the main Section \ref{sec:ppde}) are the following.

First, similar to \cite{RTZ14}, we show that the infinite-dimensional definition has an equivalent version with semijets (Proposition \ref{prop:jets}).
Then, under natural assumptions on the operator $A$ and the coefficients $b,\sigma,F$, we prove a sub/supermartingale characterization of sub/supersolutions  (Theorem \ref{lemma:pp1}), which extends the corresponding result in \cite{RTZ14}.
This key theorem is the starting point of several important results, which
are listed here.
\begin{enumerate}[(a)]
\setlength\itemsep{0.05em}
\item\label{2016-10-17:02}
PPDE \eqref{PPDE-intro} satisfies a comparison principle in the class of continuous functions with polynomial growth on $\Lambda$ (Corollary \ref{C:CompPrinciple}).
In particular, since the Crandall-Ishii Lemma is not needed to establish comparison, we emphasize that, with respect to the standard viscosity solution theory in infinite dimension, the  aforementioned conditions (non-standard continuity on the coefficients and  nuclearity on $\sigma$) are completely by-passed in our framework, therefore not needed.
\item 
For equations of type \eqref{PPDE-intro},
 our notion of viscosity solution is equivalent to the notion of \emph{mild solution} (i.e., solution of suitable integral equations; see Subsection \ref{sub:mild-visc}).
\item\label{2016-09-12:00} Given a terminal condition
$u(T,\mathbf x)=\xi(\mathbf x)$, with $\xi$ belonging to the space of continuous functions with polynomial growth, existence  and uniqueness of viscosity solution holds
(Theorem \ref{teo:main}). { Such existence and uniqueness result is proved using the equivalence with mild solutions, i.e.\ employing fixed point arguments. It must be noted that uniqueness also follows from the comparison principle (point
 (\ref{2016-10-17:02}) above).}
  \item PPDE \eqref{PPDE-intro} satisfies the  stability property of viscosity solutions (Proposition \ref{propp:2014-10-boh}). 
\end{enumerate}
\vskip-5pt
An important consequence 
{ of  (\ref{2016-10-17:02}) }
 is that the passage from finite to infinite dimension 
highlights the relevance of the new notion
of viscosity solution \emph{even} in the Markovian (no path-dependent) case. Indeed, while in the finite-dimensional case the theory based on the usual definition of viscosity solutions is so well-developed to cover basically a huge class of PDEs, in  the infinite-dimensional case the known theory of viscosity solutions collides with the structural constraints described above{; the latter} can be by-passed with the new notion allowing to cover types of equations which could not be treated with the current theory of viscosity solution in infinite dimension.

As mentioned above,  Section \ref{S:Extension} is devoted to investigating extensions of the results to   semilinear or even fully nonlinear PPDEs.
More precisely, 
 we first introduce a path-dependent stochastic optimal control problem in infinite dimension (to this regard, we mention, in finite dimension, the recent paper \cite{TZ15}, dealing within the framework of Dupire or functional It\^o calculus, an account of which can be found, e.g., in \cite{CF10b, CF13, Dupire}); then, we write the corresponding path-dependent Hamilton-Jacobi-Bellman equation and provide a definition of viscosity solution coherent with that given for PPDE \eqref{PPDE-intro}.
{It must be noted that here we are not able to prove the analogous of the key Theorem \ref{lemma:pp1}, hence we have to change the methods to attack the well posedness problem.
What we do is the following.
We state, without proof, the dynamic programming principle for the value function $v$ of the stochastic control problem and we prove that $v$ is a viscosity solution of the corresponding path-dependent Hamilton-Jacobi-Bellman equation. We also prove a partial comparison principle for this equation, namely a comparison principle when 
either the sub- or the super-solution is smooth. Finally, concerning the comparison principle, we focus on the semilinear case (i.e., the diffusion coefficient is not controlled and the drift satisfies a structure condition). In this case, when $H=\R^n$, a proof of the comparison principle has been given in \cite{RTZ14} and it is based on the notion of punctual differentiability introduced in \cite{CaffCabre}. In Subsection \ref{SubS:Punctual} we describe (without reporting a real proof, which would go beyond the scopes of the present paper) the steps that are needed in order to prove the comparison principle adapting to our infinite-dimensional framework the proof of \cite{RTZ14}.

\bigskip

The paper is organized as follows.
In Section \ref{sec:notation} we present the notation used throughout the paper.
Section \ref{sec:sde} is devoted to recalling results on  existence, uniqueness, and stability of mild solutions of path-dependent SDEs in Hilbert spaces.
In 
 Section \ref{sec:ppde}, we introduce the notion of viscosity solution for path-dependent PDEs in Hilbert spaces, in terms
(which we prove to be equivalent) of both test functions and semijets (Subsection \ref{sec:2014-06-17:aa});
we prove the key martingale characterization of viscosity sub/supersolutions (Subsection \ref{sec:stability});
we prove the comparison principle (Subsection \ref{sec:comparison}); we prove the  equivalence with mild solutions (Subsection \ref{sub:mild-visc});
finally, we provide an existence and uniqueness result and a stability result for the path-dependent PDE (Subsection \ref{sec:exun}).
In Section \ref{sec:markov}, we consider the Markovian case, i.e., when all data depend only on the present, and we compare the notion of viscosity solution studied in Section \ref{sec:ppde} to the usual notions of viscosity solutions adopted in the literature for partial differential equations in Hilbert spaces.
In Section \ref{S:Extension} we study other PPDEs,
 of Hamilton-Jacobi-Bellman (HJB) type, which can be semilinear or even fully nonlinear, and
which are associated to a stochastic control problem in infinite dimensions with path-dependence. We begin formulating the stochastic control problem and writing the corresponding path-dependent HJB equation. We give the definition of viscosity solution of that equation and we prove that the value function is a viscosity solution of it (Subsection \ref{SubS:Existence}). We prove the partial comparison principle (Subsection \ref{SubS:PartialComp}). In the semilinear case, we describe a possible way of proving the comparison principle (Subsection \ref{SubS:Punctual}).

\section{Notation}
\label{sec:notation}

Consider a real separable Hilbert space $H$. Denote by $\langle\cdot,\cdot\rangle$ and $|\cdot|$ the scalar product and norm on $H$, respectively.
Let $T>0$ and consider the Banach space
$$\mathbb W\coloneqq C([0,T];H)$$ of continuous functions from $[0,T]$ to  $H$, whose generic element is denoted by $\mathbf{x}$ and whose norm is denoted by $|\cdot|_\infty$, i.e., $|\mathbf x|_\infty \defeq  \sup_{t\in[0,T]}|\mathbf x_t|$. Introduce the space
$$
\Lambda \defeq  [0,T]\times \mathbb W
$$
and the map $\textbf{d}_\infty\colon\Lambda\times\Lambda\rightarrow\R^+$ defined by (\footnote{We use the same symbol, $|\cdot|$, to denote both the \emph{norm} on $H$ and the \emph{absolute value} of a real number. No confusion should arise, as the meaning will be clear from the context.})
\[
\textbf{d}_\infty \big((t,\mathbf{x}),(t',\mathbf{x}')\big)  \defeq   |t-t'| + |\mathbf{x}_{\cdot \wedge t} - \mathbf{x}'_{\cdot \wedge t'}|_\infty.
\]
Then $\textbf{d}_\infty$ is a pseudometric on $\Lambda$. In particular, $(\Lambda,\textbf{d}_\infty)$ is a topological space with the topology induced by the pseudometric $\textbf{d}_\infty$.
$\Lambda$ becomes a measurable space when endowed with the Borel $\sigma$-algebra induced by $\textbf d_\infty$. Throughout the paper, the topology and $\sigma$-algebra on $\Lambda$ are those induced by $\textbf{d}_\infty$.

\begin{Definition}
Let $E$ be a Banach space. 
A \textbf{non-anticipative function} on $\Lambda$ taking values in $E$
is a  function $v\colon\Lambda\rightarrow E$ such that
$$
v(t,\mathbf x) = v(t,\mathbf x_{\cdot\wedge t})
\qquad \forall (t,\mathbf x)\in\Lambda.
$$
\end{Definition}

\begin{Definition}
\label{D:Consistent}
Let $(E,|\cdot|_E)$ be a Banach space.
\begin{enumerate}[(i)]
\setlength\itemsep{0.05em}
\item \vskip-5pt
 $C(\Lambda;E)$ is the space of continuous functions $v\colon\Lambda\rightarrow E$.
\item $C_p(\Lambda;E)$, $p\geq 0$,
is the space of continuous functions $v\colon\Lambda\rightarrow E$ 
such that 
$$
|v|_{C_p(\Lambda;E)}  \coloneqq  \sup_{(t,\mathbf{x})\in \Lambda }\frac{|v(t,\mathbf{x})|_E}{1+|\mathbf{x}|_\infty^p}<\infty.
$$
$C_p(\Lambda;E)$ is a Banach space when endowed with the norm $|\cdot|_{C_p(\Lambda;E)}$.
\item
$\cpol(\Lambda;E)$ is the set of $E$-valued continuous functions with polynomial growth on $\Lambda$:
$$
\cpol(\Lambda;E) \coloneqq
\bigcup_{p\geq 0}C_p(\Lambda;E).
$$
\item $UC(\Lambda;E)$ is the space of uniformly continuous  functions $v\colon\Lambda\rightarrow E$.
\end{enumerate}
When $E=\R$, we drop $\R$ and  simply write $C(\Lambda)$, $C_p(\Lambda)$, $\cpol(\Lambda)$, and $UC(\Lambda)$.
\end{Definition}

\begin{Remark}
\label{R:non-anticipative}
${}$
\begin{enumerate}[(i)]
\setlength\itemsep{0.05em}
\item \vskip-5pt For all $p\geq 1$,
it holds
 $UC(\Lambda;E)\subset C_1(\Lambda;E)\subset C_p(\Lambda;E)\subset \cpol(\Lambda;E)\subset C(\Lambda;E)$.
\item\label{2016-06-21:01}  A measurable map $v\colon\Lambda\rightarrow E$ is automatically non-anticipative. For this reason, we will drop the term non-anticipative when $v$ is measurable.
\end{enumerate}
\end{Remark}

Now let  $(\Omega,\mathcal{F},\mathbb{F}=(\mathcal{F}_t)_{t\geq 0}, \mathbb{P})$ be  a filtered probability space satisfying the usual conditions. We shall make use of the following
 classes of stochastic processes on this space.

\begin{Definition}
\label{D:Spaces}
Let $(E,|\cdot|_E)$ be a Banach space.
\begin{enumerate}[(i)]
\setlength\itemsep{0.05em}
\item \vskip-5pt
$L_{\cal{P}}^0(E)\defeq L_{\cal{P}}^0(\Omega\times[0,T];E)$  is the space 
 of $E$-valued predictable processes $X$, endowed with the topology induced by the convergence in measure.
\item $L_{\cal{P}}^p(E)\defeq L_{\cal{P}}^p(\Omega\times[0,T];E)$, $p\geq 1$, is the Banach space of $E$-valued predictable processes $X$ such that
$$
|X|^p_{L^p_{\cal{P}}(E)}  \defeq   \mathbb{E}\left[\int_0^T|X_t|_E^pdt\right] < \infty.
$$
\item $\mathcal{H}_\mathcal{P}^0(E)$ is the subspace of elements  $X\in L_{\cal{P}}^0(E)$ admitting a continuous version.
Given an element of $\mathcal{H}_\mathcal{P}^0(E)$ we shall always refer to its uniquely determined (up to a $\P$-null set) continuous version.
\item $\mathcal{H}_\mathcal{P}^p(E)$, $p\geq 1$, is the subspace of elements  $X\in L_{\cal{P}}^p(E)$ admitting a continuous version and such that
$$
|X|_{\mathcal{H}_\mathcal{P}^p(E)}^p  \defeq   \mathbb{E}\bigg[\sup_{t\in[0,T]}|X_t|^p_E \bigg] < \infty.
$$
$\mathcal{H}_\mathcal{P}^p(E)$, when endowed with the norm $|\cdot|_{\mathcal{H}_\mathcal{P}^p(E)}$ defined above, is a  Banach space.
\end{enumerate}
When $E=\R$, we drop $\R$ and  simply write $L_{\cal{P}}^0, \ L_{\cal{P}}^p, \ \mathcal{H}_\mathcal{P}^0$, and $\mathcal{H}_\mathcal{P}^p$.
\end{Definition}

\begin{Remark}\label{rem:setting}
In the present paper, as it is usually done in the literature on infinite-dimensional second order PDEs (see, e.g., \cite{DPZ92, S94}),  we distinguish between the probability space $(\Omega,\Fc,\P)$, whose generic element is $\omega$, and the path space $\mathbb W$, whose generic element is $\mathbf{x}$.
Instead, in \cite{EkrenKellerTouziZhang},  the authors identify these two spaces (up to the translation of the initial point), taking as probability space  the canonical space $\{\mathbf x\in \mathbb W\colon \mathbf x_0=0\}$ and calling $\omega$ its generic element. For equations \eqref{PPDE-intro} (treated up to Section \ref{sec:markov} included), { our setting} can be rephrased in the setting of \cite{EkrenKellerTouziZhang} by taking as probability space $(\mathbb W, \mathcal{B}(\mathbb W), \mathbb{P}^X)$, where  $\mathcal{B}(\mathbb W)$ is the $\sigma$-algebra of Borel subsets of $\mathbb W$ and $\P^X$ is the law of the process $X$ that we shall define in the next section as mild solution of a path-dependent SDE.
{ For the equations treated in   
 Section \ref{S:Extension}, our setting can be rephrased in the setting of \cite{EkrenKellerTouziZhang} by  considering the family of probability measure $\P^{X,\mathbf{a}}$, the laws of the controlled process $X$ when $\mathbf{a}$ ranges over the set of control processes.}
\end{Remark}

\section{Preliminaries on path-dependent SDEs in Hilbert spaces}
\label{sec:sde}

In this section we introduce a path-dependent SDE in Hilbert space whose mild solution will provide our reference process for the definition of viscosity solution.
 As general references for stochastic integration and SDEs in infinite-dimensional spaces, we refer to the monographies \cite{DPZ92, GM10}.

Let $K$ be a real separable Hilbert space with inner product $\langle\cdot,\cdot\rangle_K$ and  let $W=(W_t)_{t\geq 0}$ be a $K$-valued cylindrical Wiener process on the filtered probability space $(\Omega,\mathcal{F},\mathbb{F}=(\mathcal{F}_t)_{t\geq 0}, \mathbb{P})$. We consider, for  $t\in[0,T]$ and $Z\in\mathcal{H}_\mathcal{P}^0(H)$, the following \emph{path-dependent} SDE:
\begin{equation}
\label{SHDE}
\begin{cases}
dX_s= AX_s ds + b(s,X)ds + \sigma(s,X)dW_s, \qquad &s\in[t,T], \\
X_{\cdot\wedge t} = Z_{\cdot\wedge t}.
\end{cases}
\end{equation}
The precise notion of solution is given below. First, we introduce some notations and then impose Assumption \ref{A:SHDE} on $A$, $b$, $\sigma$. We recall that
 $L(K;H)$ denotes the Banach space of bounded linear operators from $K$ to  $H$, endowed with the operator norm. We  denote by $L_2(K;H)$ the Hilbert space of Hilbert-Schmidt operators from $K$ to $H$, whose scalar product and norm are, respectively,
\[
\langle P,Q\rangle_{L_2(K;H)}  \coloneqq   \sum_{k=1}^\infty \langle Pe_k,Qe_k\rangle, \qquad |P|_{L_2(K;H)}   \coloneqq   \bigg(\sum_{k=1}^\infty |Pe_k|^2\bigg)^{1/2},
\]
for all $P,Q\in L_2(K;H)$, where $\{e_k\}_k$ is a complete orthonormal basis of $K$.

\begin{Assumption}
\label{A:SHDE}
\quad
\begin{enumerate}[(i)]
\setlength\itemsep{0.05em}
\item 
\label{2016-07-13B:00}
\vskip-5pt The operator $A\colon\mathcal{D}(A)\subset H\to H$ is the generator of a strongly continuous semigroup $\{e^{tA},\ t\geq 0\}$ in the Hilbert space $H$.
\item\label{2016-10-03:00} $b\colon\Lambda \rightarrow H$ is measurable and such that, for some constant $M>0$,
\[
|b(t,\mathbf{x}) - b(t,\mathbf{x}')|  \leq  M|\mathbf{x}-\mathbf{x}'|_{\infty}, \qquad |b(t,\mathbf x)|  \leq  M(1 + |\mathbf x|_\infty),
\]
for all $\mathbf x,\mathbf x'\in\mathbb W$, $t\in[0,T]$.
\item\label{2016-06-21:00} $\sigma\colon\Lambda\rightarrow L(K;H)$  is such that $\sigma(\cdot,\cdot)v\colon\Lambda \rightarrow H$ is measurable for each $v\in K$ and  $e^{sA}\sigma(t,\mathbf x)\in L_2(K;H)$ for every $s>0$ and every  $(t,\mathbf{x})\in\Lambda.$  Moreover, there exist $\hat M>0$ and $\gamma\in[0,1/2)$ such that, for all $\mathbf{x},\mathbf{x}'\in \mathbb W$, $t\in[0,T]$, $s\in(0,T]$,
\begin{align}\label{ICS}
|e^{sA}\sigma(t,\mathbf x)|_{L_2(K;H)}  &\leq  \hat M s^{-\gamma}(1 + |\mathbf{x}|_{\infty}), \\[5pt]
|e^{sA}\sigma(t,\mathbf x) - e^{sA}\sigma(t,\mathbf x')|_{L_2(K;H)}  &\leq  \hat M s^{-\gamma}|\mathbf{x}- \mathbf{x}'|_{\infty}.\label{ICS2}
\end{align}
\end{enumerate}
\end{Assumption}
\begin{Remark}
 Regarding  Assumption \ref{A:SHDE}(\ref{2016-06-21:00}), we observe that one could do the more demanding assumption of sublinear growth and Lipschitz continuity of $\sigma(t,\cdot)$ as function valued in the space $L_2(K;H)$ (see \cite{GM10}).
The assumption we give, which is the minimal one used in literature to give sense to the stochastic integral and to ensure the continuity of the stochastic convolution, is taken from \cite[Hypothesis\ 7.2]{DPZ92} and \cite{fuhrmantess02}.
 Regarding Assumption \ref{A:SHDE}(\ref{2016-10-03:00}),  we observe that it could be relaxed giving assumptions on the composition of the map $b$ with the semigroup, as done for $\sigma$ in part (\ref{2016-06-21:00}) of the same Assumption. Here, we follow \cite{DPZ92,fuhrmantess02} and we do not perform it.
\end{Remark}

Before giving the precise notion of solution to \eqref{SHDE} we make some observations.
\begin{enumerate}[(O1)]
\setlength\itemsep{0.05em}
\item 
\label{2016-08-09:00}
\vskip-5pt For $p=0$ and $p\geq 1$, we have the isometric embedding (\footnote{In the case $p=0$, the spaces $\mathcal{H}_\mathcal{P}^0$ and $L^0_\mathcal{P}$ are endowed with the metrics associated to the convergence in measure (see \cite[Ch.\ 1, Sec.\ 5]{M95}).})
$$
\mathcal{H}^p_{\mathcal{P}}(H) \hookrightarrow  L^p(\Omega,\mathcal{F},\P;\mathbb W).
$$
Hence a process in $\mathcal{H}^p_{\mathcal{P}}(H)$, $p=0$ or $p\geq 1$, can be seen (and we shall adopt this point of view in many points throughout the paper) as a $\mathbb W$-valued random variable.

\item\label{2016-08-09:01} If $X\in\mathcal{H}^p_{\mathcal{P}}(H)$, $p=0$ or $p\geq 1$, then $X_{\cdot\wedge t} \in L^p(\Omega, \mathcal{F}_t,\P;\mathbb W)$.
\item\label{2016-08-09:02} 
The topology on $\Lambda$ induced by the pseudometric $\mathbf{d}_\infty$ is weaker than the topology on $\Lambda\subset \mathbb{R}\times \mathbb{W}$ induced by  the norm $|\cdot|+|\cdot|_\infty$.
\item\label{2016-08-09:03} Given $v\in C(\Lambda;H)$ and $X\in \mathcal{H}^0_{\mathcal{P}}(H)$, due to
 (O\ref{2016-08-09:00})--(O\ref{2016-08-09:02})) above, the composition $v(\cdot,X)$  belongs to $ \mathcal{H}^0_\mathcal{P}(H)$.
\item\label{2016-08-09:04} Given $v\in C_{q}(\Lambda;H)$ and $X\in \mathcal{H}^p_{\mathcal{P}}(H)$, with {$q>0$, $1 \leq p<\infty$, $p\geq q$},  due to 
 (O\ref{2016-08-09:00})--(O\ref{2016-08-09:02})) above,  the composition $v(\cdot,X)$ is a process in the class $\mathcal{H}^{p/q}_{\mathcal{P}}(H)$.
In particular, if $v\in \cpol(\Lambda;H)$, then
$$
X\in \bigcap_{p\geq 1}\mathcal{H}^p_\mathcal{P}(H)\
\Rightarrow\
v(\cdot,X)\in
\bigcap_{p\geq 1}\mathcal{H}^p_\mathcal{P}(H).
$$
\end{enumerate}

\begin{Definition}\label{def:2014-11-17:aa}
Let  $t\in[0,T]$, $Z\in \mathcal{H}_\mathcal{P}^0(H)$. We call \textbf{mild solution} of \eqref{SHDE} a process $X\in\mathcal{H}^0_\mathcal{P}(H)$ such that $X_{\cdot\wedge t}=Z_{\cdot\wedge t}$ and
\begin{equation}
\label{mild}
X_s  =  e^{(s-t)A}Z_t + \int_t^s e^{(s-r)A}b(r,X)dr + \int_t^s e^{(s-r)A} \sigma(r,X) dW_r, \quad \forall\,s\in [t,T].
\end{equation}
\end{Definition}

Notice that  condition \eqref{ICS} implies
$$
\int_t^s |e^{(s-r)A} \sigma(r,\mathbf{x})|^2_{L_2(K;H)} dr \leq  C_0(1+|\mathbf{x}|^2_\infty), \quad \forall\, t\in[0,T],\ s\in[t,T],\ \mathbf{x}\in\mathbb W,
$$
which ensures that the stochastic integral in Definition \ref{def:2014-11-17:aa} makes sense for every process  $X\in \mathcal{H}^0_\mathcal{P}(H)$.

\smallskip
We are going to state an existence and uniqueness result. To this end, we define
$$ p^*\coloneqq  \frac{2}{1-2\gamma}.
$$

\begin{Theorem}
\label{Exist}
Let Assumption \ref{A:SHDE} hold. Then, for every $p>p^*$, $t\in[0,T]$ and   $Z\in \mathcal{H}_\mathcal{P}^p(H)$, there exists a unique mild solution $X^{t,Z}$ to \eqref{SHDE}. Moreover,    $X^{t,Z}\in \mathcal{H}_\mathcal{P}^p(H)$ and
\begin{equation}
\label{E:EstimateSDE}
|X^{t,Z}|_{\mathcal{H}^p_\mathcal{P}(H)}  \leq   K_0(1 + |Z|_{\mathcal{H}^p_\mathcal{P}(H)}), \qquad \forall \, (t,Z)\in[0,T]\times \mathcal{H}^p_\mathcal{P}(H).
\end{equation}
Finally,
 the map
\begin{equation}\label{eq:2014-10-29:aa}
 [0,T]\times\mathcal{H}_\mathcal{P}^p(H)\to\mathcal{H}_\mathcal{P}^p(H), \qquad (t,Z)\mapsto X^{t,Z}
\end{equation}
is
Lipschitz continuous with respect to $Z$,  uniformly in $t\in[0,T]$, and  jointly continuous.
\end{Theorem}
\begin{Remark}
Since for $p^*<p<q$ we have $\mathcal{H}_\mathcal{P}^p(H)\supset \mathcal{H}_\mathcal{P}^q(H)$, if $Z\in \mathcal{H}_\mathcal{P}^q(H)$, then the associated mild solution $X^{t,Z}$ is also a solution in $\mathcal{H}_\mathcal{P}^p(H)$ and, by uniqueness, it is \emph{the} solution in that space. Hence, the solution does not depend on the specific $p>p^*$ chosen.
\end{Remark}
\begin{proof}[Proof of Theorem \ref{Exist}.]
The theorem is a particular
case of \cite[Th.\ 3.6]{Roses2016a},
for the existence/uniqueness part and for the Lipschitz
continuity with respect to $Z$,
and of 
\cite[Th.\ 3.14]{Roses2016a}, for the joint continuity in $t,Z$.
For a sketch of proof of the existence/uniqueness part, the reader can also refer to \cite[Prop.\ 3.2]{fuhrmantess02}.
\end{proof}

We notice that uniqueness of mild solutions and the semigroup property of $\{e^{sA}\}_{s\geq 0}$ yield the flow property for the solution with initial data $(t,\mathbf{x})\in\Lambda$:
\begin{equation}\label{flow}
X^{t,\mathbf{x}}= X^{s,X^{t,\mathbf x}}\ \mbox{in}\ \mathcal{H}^p_\mathcal{P}(H), \  \forall\, (t,\mathbf{x})\in \Lambda, \ \forall\,s\in[t,T].
\end{equation}
In the sequel, we shall use  the following generalized dominated convergence result.
\begin{Lemma}\label{GD}
\emph{Let $(\Sigma, \mu)$ be a measure space. Assume that  $f_n, g_n, f, g\in L^1(\Sigma, \mu;\R)$, $f_n\to f$ and $g_n\to g$  $\mu$-a.e., $|f_n|\leq g_n$ and $\int_{\Sigma} g_nd\mu\to \int_{\Sigma} g d\mu$. Then $\int_{\Sigma} f_nd\mu\to \int_{\Sigma} fd\mu$.}
\end{Lemma}
\begin{Corollary}\label{corr:2014-10-29:aa}

  Let $p'\geq 1$, $\Psi\in L^{\infty}(\Omega, \mathcal{F},\P;C_{p'}(\Lambda))$ and $p>p^*$, $p\geq p'$.
 Then the map
\begin{equation}\label{eq:2014-11-26:ab}
[0,T]\times [0,T]\times \mathcal{H}_\mathcal{P}^p(H)\to \mathbb{R}, \  (s,t,Z)\mapsto \mathbb{E}\left[\Psi(\cdot)(s,X^{t,Z})\right]
\end{equation}
is well-defined and continuous.
\end{Corollary}
\begin{proof}
In view of Theorem \ref{Exist},  the map  \eqref{eq:2014-11-26:ab} is well-defined.
Concerning continuity, again in view of Theorem \ref{Exist},
 it suffices to show that the map
$$
[0,T]\times\mathcal{H}^p_\mathcal{P}(H)\to \mathbb{R}, \
 (s,Y)\mapsto \mathbb{E}[\Psi(\cdot)(s,Y)]
$$
is continuous.
Let $\{Y^{(n)}\}_{n}$ be a sequence converging to $Y$ in
$ \mathcal{H}_\mathcal{P}^p(H)$, and $s_n\to s$ in $[0,T]$.
Let $\{Y^{(n_k)}\}_k$ be  a subsequence such that
$
| Y-Y^{(n_k)}|_\infty\to 0
$\
$\mathbb{P}$-a.s.. Then, using the continuity of $\Psi(\omega)(\cdot,\cdot)$  we get, by applying Lemma \ref{GD}, the convergence
$
\mathbb{E}[\Psi(\cdot)(s_{n_k},Y^{(n_k)})]\to \mathbb{E}[\Psi(\cdot)(s,Y)].
$
Since the original converging sequence $\{(s_n,Y^{(n)})\}_n$ was arbitrary, we get the claim.
\end{proof}

The following stability result for SDE \eqref{SHDE} 
 will be used to prove the stability of viscosity solutions in the next section.

\begin{Proposition}\label{prop:stabSDE}
Let Assumption \ref{A:SHDE} hold and assume that it holds also, for each $n\in\N$, for analogous objects $A_n$, $b_n$ and $\sigma_n$, such that the estimates of parts 
(\ref{2016-10-03:00})--(\ref{2016-06-21:00})
 in Assumption~\ref{A:SHDE} hold with the constants $M,\hat M, \gamma$.
Assume that the following convergences hold for every $(t,\mathbf{x})\in \Lambda$ and every $s\in[0,T]$:
\begin{enumerate}[(i)]
\setlength\itemsep{0.05em}
\item \vskip-5pt
  $e^{sA_n}\mathbf{x}_s\to e^{sA}\mathbf{x}_s$ in $H$;
\item  $e^{s A_{n}} b_n(t,\mathbf{x})\to e^{s A} b(t,\mathbf{x})$ in $H$;
 \item $e^{s A_{n}} \sigma_n(t,\mathbf{x})\to e^{s A} \sigma(t,\mathbf{x})$ in $L_2(K;H)$.
\end{enumerate}
Let $t\in [0,T]$, $Z\in \mathcal{H}_\mathcal{P}^p(H)$, 
for some $p> p^*$, 
and let $X^{(n),t,Z}$ be the mild solution to \eqref{SHDE}, where $A,b,\sigma$ are replaced by $A_n, b_n,\sigma_n$.
Then  $X^{(n), t,Z}\stackrel{n\rightarrow \infty}{\longrightarrow} X^{t,Z}$ in  $ \mathcal{H}_\mathcal{P}^p(H)$ and, for fixed $t$, there exists $K_0$ such that
\begin{equation}
\label{E:EstimateSDE-n}
|X^{(n),t,Z}|_{\mathcal{H}^p_\mathcal{P}(H)} \leq   K_0(1 + |Z|_{\mathcal{H}^p_\mathcal{P}(H)}), \qquad \forall\, Z\in \mathcal{H}^p_\mathcal{P}(H),\ \forall\, n\in \mathbb{N}.
\end{equation}
\end{Proposition}
\begin{proof}
See \cite[Th.\ 3.14]{Roses2016a}.
\end{proof}

\section{Path-dependent PDEs and viscosity solutions in Hilbert spaces}
\label{sec:ppde}

In the present section, we introduce a path-dependent PDE in the space $H$ and study it through the concept of viscosity solutions in the spirit of the definition given in \cite{EkrenKellerTouziZhang, EkrenTouziZhang1,RTZ14}. As in \cite{RTZ14}, we also provide an equivalent definition in terms of jets. The key result is a martingale characterization for viscosity sub/supersolution, from which { the stability result and the comparison principle follow}. We finally prove the existence of a viscosity solution through a fixed point argument.

\smallskip

Assumption~\ref{A:SHDE} on the coefficients $A,b,\sigma$ will be standing for the remaining part of this section.

\subsection{Definition: test functions and semijets}\label{sec:2014-06-17:aa}

We begin introducing
 the set
$C^{1,2}_{X}(\Lambda)$ of smooth functions, which will be used to define test functions. We note that the definition of the latter set shall depend on the process $X^{t,\mathbf{x}}$, solution to \eqref{SHDE}, that is on the coefficients $A,b,\sigma$. The subscript $X$ in the notation $C^{1,2}_{X}(\Lambda)$ stays there to recall that.

\begin{Definition}\label{def:smooth_functions}
 We say that $u\in C^{1,2}_{X}(\Lambda)$ if
 $u\in \cpol(\Lambda)$ and there exist $\alpha\in \cpol(\Lambda)$, $\beta\in \cpol(\Lambda;K)$
 such that, 
for all $(t,\mathbf{x})\in\Lambda$,
$\mathbb{P}$-a.s. 
\begin{equation}
\label{eq:fct_Ito}
du(s,X^{t,\mathbf x})  =  \alpha(s,X^{t,\mathbf x})ds + \langle \beta(s,X^{t,\mathbf x}),dW_s\rangle_{K}, \qquad \forall\, 
 s\in[t,T].
\end{equation}

\end{Definition}
Note that Theorem~\ref{Exist} guarantees integrability in \eqref{eq:fct_Ito}.
Note also that $\alpha$ and $\beta$ in Definition~\ref{def:smooth_functions} are uniquely determined, as it can be easily shown by identifying the finite variation part and the Brownian part in \eqref{eq:fct_Ito}. Given $u\in C^{1,2}_{X}(\Lambda)$, we define
\begin{equation}\label{def:L}
\Lc u \coloneqq   \alpha.
\end{equation}

Before to proceed, we argue to motivate the notation above and the meaning of $\mathcal{L}$ as a generalization of a second order  differential operator.
The class of test functions used to define viscosity solutions for path-dependent PDEs has evolved from \cite{EkrenKellerTouziZhang} and \cite{EkrenTouziZhang1} to the recent work \cite{RTZ14}. In Definition \ref{def:smooth_functions}, which is inspired by \cite{RTZ14}, there is no more reference to the so-called pathwise (or functional, or Dupire) derivatives (for which we refer to \cite{Dupire} and also to \cite{CF10a,CF10b,CF13,CR14}), which are instead adopted in \cite{EkrenKellerTouziZhang} and \cite{EkrenTouziZhang1} (actually in \cite{EkrenTouziZhang1} only the pathwise time derivative is used). This allows to go directly to the definition of viscosity solution, without pausing on the definition of pathwise derivatives, and, more generally, on recalling tools from functional It\^o calculus. However, the class of test functions used in \cite{EkrenKellerTouziZhang} or \cite{EkrenTouziZhang1} has the advantage to be defined in a similar way to $C^{1,2}$, the standard class of functions 
continuously Fr\'echet differentiable once in time and twice in space. In this case the object $\Lc u$ of \eqref{def:L}, which in the present paper is only abstract,  can be expressed in terms of the pathwise derivatives, as in the non path-dependent case, where $\Lc$ corresponds to a parabolic operator and can be written by means of time and spatial derivatives.

For this reason, in order to better understand Definition \ref{def:smooth_functions} and the notation $\Lc u$, we now define a subset of test functions $\mathscr{C}_X^{1,2}(\Lambda)\subset C^{1,2}_{X}(\Lambda)$ which admit the pathwise derivatives we are going to define. Here we follow \cite{EkrenTouziZhang1}, generalizing it to the present infinite-dimensional setting.
\begin{Definition}
\label{D:TimeDer}
Given $u\in \cpol(\Lambda)$, we  define the \textbf{pathwise time derivative} of $u$ at $(t,\mathbf x)\in\Lambda$ as follows:
\[
\begin{cases}
\partial_t u(s,\mathbf{x})\defeq \lim_{h\rightarrow0^+} \dfrac{ u(s+h,\mathbf{x}_{\cdot \wedge s}) - u(s,\mathbf{x})}{h},   &s\in[0,T),\\\\
\partial_t u(T,\mathbf{x})\defeq \lim_{s\rightarrow T^-}\partial_t u(s,\mathbf{x}), &s=T,
\end{cases}
\]
when these limits exist.
\end{Definition}
In the following definition $A^*$ is the adjoint operator of $A$, defined on $\mathcal{D}(A^*)\subset H$.

\begin{Definition}
\label{D:SpaceDer}
Denote by $S(H)$ the Banach space of bounded and self-adjoint operators in the Hilbert space $H$ endowed with the operator norm, and let  $\mathcal{D}(A^*)$ be endowed  with the graph norm, which makes it a Hilbert space. We say that $u\in \cpol(\Lambda)$ belongs to $\mathscr C^{1,2}_{X}(\Lambda)$ if:
\begin{enumerate}[(i)]
\setlength\itemsep{0.05em}
\item\label{2016-10-04:00} \vskip-5pt
there exists $\partial_t u$ in $\Lambda$ in the sense of Definition \ref{D:TimeDer} and it  belongs to $\cpol(\Lambda)$;
\item\label{2016-10-04:01}
there exist two maps $\partial_{\mathbf x}u\in \cpol(\Lambda;\mathcal{D}(A^*))$ and $\partial_{\mathbf{xx}}^2u\in \cpol(\Lambda;S(H))$ such that $\textup{Tr}\left[\sigma \sigma^*\partial_{\mathbf{xx}}^2u\right]$  is finite over $\Lambda$ and
the following \textbf{functional It\^o formula} holds for all $(t,\mathbf{x})\in \Lambda$:
\begin{equation}\label{functionalIto}
du(s,X^{t,\mathbf x})  =  \mathcal{L}_0 u(s,X^{t,\mathbf{x}}) ds 
+ \langle \sigma^*(s,X^{t,\mathbf x})\partial_{\mathbf x}u(s,X^{t,\mathbf x}),dW_s\rangle,
\end{equation}
for  $s\in [t,T]$,
where, for $(s,\mathbf{y})\in \Lambda$,
\begin{multline}\label{defL}
 \mathcal{L}_0u(s,\mathbf{y}) \coloneqq  \partial_t u(s,\mathbf y) + \langle \mathbf{y}_t,A^*\partial_{\mathbf x}u(s,\mathbf{y})\rangle \\+ \langle b(s,\mathbf{y}),\partial_{\mathbf x}u(s,\mathbf{y})\rangle
+ \frac{1}{2}\textup{Tr}\big[\sigma(s,\mathbf y)\sigma^*(s,\mathbf y)\partial_{\mathbf{xx}}^2u(s,\mathbf{y})\big].
\end{multline}
\end{enumerate}
We call $\partial_{\mathbf x}u$ and $\partial_{\mathbf{xx}}^2u$ \textbf{pathwise first order spatial derivative} and \textbf{pathwise second order spatial derivative} of $u$ with respect to $X$, respectively.
\end{Definition}
\smallskip

Notice that, given $u\in\mathscr C^{1,2}_{X}(\Lambda)$ and $(t,\mathbf{x})\in\Lambda$, the objects  $\partial_{\mathbf x}u$ and $\partial_{\mathbf{xx}}^2u$ are not necessarily uniquely determined, while
$
 \mathcal{L}_0u$ defined as in \eqref{defL} and $\sigma^*\partial_{\mathbf x}u$ are uniquely determined (this can be shown by identifying the
 part
with  finite variation and the Brownian part in the functional It\^o formula \eqref{functionalIto}). Moreover, if $u\in\mathscr C^{1,2}_{X}(\Lambda)$, then  \eqref{eq:fct_Ito} is satisfied with
\begin{equation*}
  \begin{dcases}
          \alpha(t,\mathbf{x}) =& \partial_t u(t,\mathbf x) + \langle \mathbf{x}_t,A^*\partial_{\mathbf x}u(t,\mathbf{x})\rangle\\
& + \langle b(t,\mathbf{x}),\partial_{\mathbf x}u(t,\mathbf{x})\rangle
+ \frac{1}{2}\textup{Tr}\big[\sigma(t,\mathbf x)\sigma^*(t,\mathbf x)\partial_{\mathbf{xx}}^2u(t,\mathbf{x})\big],\\
\beta(t,\mathbf x) =& \sigma^*(t,\mathbf x)\partial_{\mathbf x}u(t,\mathbf x).
\end{dcases}
\end{equation*}
In particular, $\mathscr C^{1,2}_{X}(\Lambda)\subset C^{1,2}_{X}(\Lambda)$
and  the operator $\mathcal{L}$ acts on the elements of $\mathscr C^{1,2}_{X}(\Lambda)$ as a differential operator. Indeed, in this case $\mathcal{L}u=\mathcal{L}_0u$ with $\mathcal{L}_0u$ defined in \eqref{defL}.

\begin{Remark}\label{remppp}
One of the key ingredients of the notion of viscosity solution we are going to define is the concept of test function introduced in Definition \ref{def:smooth_functions}. Notice that, the larger the class of test functions, the easier should be the proof of the comparison principle and the harder the proof of the existence.
 In order to make easier the proof of uniqueness, we weaken the concept of test functions as much as possible, 
but  keeping ``safe'' the existence part. The space $C^{1,2}_{X}(\Lambda)$ is the result of this trade-off. It is a quite large class of test functions: for example,
if $f\in \cpol(\Lambda)$, then $\varphi(t,\mathbf x)\coloneqq \int_0^t f(s,\mathbf x)ds$ belongs to $C^{1,2}_{X}(\Lambda)$, whereas, even if $H=\R$ and $f$ is Markovian (i.e., $f(s,\mathbf{x})=f(s,\mathbf{x}_s)$), it does not belong, in general, to the usual class  $C^{1,2}([0,T]\times\R;\R)$ of smooth  functions.
\end{Remark}

We are concerned with the  study of the following \emph{path-dependent} PDE (from now on, PPDE):
\begin{equation}
\mathcal{L} u(t,\mathbf{x})
+ {F}(t,\mathbf{x},u(t,\mathbf{x})) = 0, \qquad \forall \,(t,\mathbf x)\in\Lambda, \ t<T, \label{eq:PPDE}
 \end{equation}
 with terminal condition
 \begin{equation}
u(T,\mathbf{x}) =  \xi(\mathbf{x}), \quad \mathbf{x}\in \mathbb W, \label{terminal}
\end{equation}
where $F\colon\Lambda\times \mathbb{R}\to \mathbb{R}$ and $\xi\colon \mathbb W\to \mathbb{R}$.
 From what we have said above, if $u\in\mathscr C^{1,2}_{X}(\Lambda)$, then  \eqref{eq:PPDE} can be written in the form \eqref{PPDE-intro}, expressing $\mathcal Lu(t,\mathbf x)$ in terms of the pathwise derivatives of $u$. Motivated by that, even if in general $\mathcal{L}$ is not a differential operator, we still keep the terminology PPDE to refer to  \eqref{eq:PPDE}.

Now we introduce the concept of viscosity solution for PPDE (\ref{eq:PPDE}), following \cite{EkrenKellerTouziZhang,EkrenTouziZhang1,RTZ14}. To this end, we denote
$$
\mathcal{T}  \defeq  \big\{\tau\colon\Omega\rightarrow [0,T] \colon\  \tau \ \mbox{is an} \ \mathbb{F}\mbox{-stopping time}\big\}.
$$
Given $u\in \cpol(\Lambda)$, we define the following two classes of test functions:
\begin{multline*}
\underline{\Ac} u(t,\mathbf{x}) \defeq  \Big\{ \varphi \in C^{1,2}_{X}(\Lambda) \colon \text{there exists }\H\in\mathcal{T},\ \H>t, \mbox{ such that } \\
(\varphi-u)(t,\mathbf{x}) = \min_{\tau \in \Tc,\,\tau\geq t}\E\big[ (\varphi-u)(\tau\wedge\H,X^{t,\mathbf{x}}) \big] \Big\},
\end{multline*}
\begin{multline*}
\overline{\Ac} u(t,\mathbf{x}) \defeq  \Big\{ \varphi \in C^{1,2}_{X}(\Lambda) \colon \text{there exists }\H\in\mathcal{T},\ \H>t, \mbox{ such that } \\
(\varphi-u)(t,\mathbf{x}) = \max_{\tau \in \Tc,\,\tau\geq t}\E\big[ (\varphi-u)(\tau\wedge\H,X^{t,\mathbf{x}}) \big] \Big\}.
\end{multline*}

 \begin{Remark}\label{rem:H}
Throughout this section, the fact that the localizing stopping time $\H$ in the definition of test functions above is \emph{stochastic} does not play a role in the proofs: actually, the definition could be given with deterministic localizing times $\H$ and the proofs would work as well. This comment applies also to the definition of test functions given in Section \ref{S:Extension}. However, we  keep the definition with stochastic stopping times $\H$, as this enlarges the set of test functions and so, in principle, makes easier uniqueness\ (\,\footnote{The existence part  --- not only for the equations treated  in the present section, but also  for those treated in Subsection \ref{SubS:Existence} --- is still kept safe by this enlargement of the set of test functions defined with localizing stochastic stopping times.}). This might be needed or useful to treat other types of equations  and/or to prove stronger comparison results than those provided here 
(see \cite{RTZ14}).
\end{Remark}

\begin{Definition} \label{def:visc_sol_PPDE}
Let  $u\in \cpol(\Lambda)$.
\begin{enumerate}[(i)]
\setlength\itemsep{0.05em}
\item \vskip-5pt
 We say that $u$ is a \textbf{viscosity subsolution} (resp.\ \textbf{supersolution}) of PPDE~\eqref{eq:PPDE} if
\begin{equation*}
  - \Lc \varphi(t,\mathbf{x})
- {F}(t,\mathbf{x}, u(t,\mathbf{x}))
\leq 0, \qquad  (\mbox{resp.}\  \ge 0)
\end{equation*}
for any $(t,\mathbf{x}) \in \Lambda$, $t<T$, and any $\varphi \in \Acb u(t, \mathbf{x})$ (resp.\ $\varphi \in \Acu u(t, \mathbf{x})$).
\item
 We say that $u$ is a \textbf{viscosity solution} of PPDE~\eqref{eq:PPDE} if it is both a viscosity subsolution and a viscosity supersolution.
\end{enumerate}
\end{Definition}

Following \cite{RTZ14}, we now provide an equivalent definition of viscosity solution in terms of semijets.
Given $u\in \cpol(\Lambda)$, define the \emph{subjet} and \emph{superjet} of $u$ at $(t,\mathbf x)\in\Lambda$ as
\begin{align*}
\underline{\Jc} u(t,\mathbf{x})  &\defeq  \big\{ \alpha\in \R \colon  \exists \, \varphi\in\underline{\Ac} u(t,\mathbf{x})\mbox{ such that }\varphi(s,\mathbf{y})=\alpha s,\ \forall\, (s,\mathbf{y})\in\Lambda \big \}, \\[5pt]
\overline{\Jc} u(t,\mathbf{x})  &\defeq  \big\{ \alpha\in \R \colon  \exists\, \varphi\in\overline{\Ac} u(t,\mathbf{x})\mbox{ such that }\varphi(s,\mathbf{y})=\alpha s,\ \forall\, (s,\mathbf{y})\in\Lambda \big \}.
\end{align*}
We have the following equivalence result.

\begin{Proposition}
\label{prop:jets}
$u\in\cpol(\Lambda)$ is a viscosity subsolution (resp.\ supersolution) of PPDE \eqref{eq:PPDE} if and only if
\begin{equation*}
  -\alpha - F(t,\mathbf x,u(t,\mathbf x))
\leq 0, \qquad  (\mbox{resp.}\ \ge 0),
\end{equation*}
for every $\alpha\in\underline{\Jc} u(t,\mathbf{x})$ (resp.\ $\alpha
\in\overline{\Jc} u(t,\mathbf{x})$).
\end{Proposition}
\begin{proof}
We focus on the \emph{`if'} part, since the other implication is clear. Fix $(t,\mathbf x)\in\Lambda$ and $\varphi\in\underline{\Ac} u(t,\mathbf{x})$ (the supersolution part has a similar proof).
From Definition \ref{def:smooth_functions} we know that there exists 
$\Lc\varphi\coloneqq \alpha\in\cpol(\Lambda)$ and $\beta\in \cpol(\Lambda;H)$  such that \eqref{eq:fct_Ito} holds, with $\varphi$ in place of $u$. Set
\[
\alpha_0  \coloneqq   \Lc\varphi(t,\mathbf x)  = \alpha(t,\mathbf{x})
\]
and, for every $\eps>0$, consider $\varphi_\eps(s,\mathbf{y})\coloneqq (\alpha_0+\eps) s$, for all $(s,\mathbf y)\in\Lambda$. Then $\varphi_\eps\in C^{1,2}_{X}(\Lambda)$. Since $\Lc\varphi$ is continuous, we can find $\delta_\varepsilon>0$ such that
\[
\big|\Lc\varphi(t',\mathbf x') - \alpha_0\big|  = \big|\Lc\varphi(t',\mathbf x') - \Lc\varphi(t,\mathbf x)\big|  \leq  \eps, \quad \mbox{if\ }\textbf{d}_\infty \big((t',\mathbf{x}'),(t,\mathbf{x})\big)  \leq  \delta_\varepsilon.
\]
Let $\H$ be the stopping time associated to $\varphi$ appearing in the definition of $\underline{\Ac} u(t,\mathbf{x})$ and define
\[
\H_\eps  \coloneqq   \H\wedge\big\{s\geq t\colon\textbf{d}_\infty \big((s,X^{t,\mathbf x}),(t,\mathbf{x})\big)  >  \delta_\eps\big\}.
\]
Note that $\H_\eps>0$. Then, for any $\tau\in\Tc$ with $\tau\geq t$, we have
\begin{equation}
\label{eq1}
\begin{split}
  (u-\varphi_\eps)(t,\mathbf x) - \E\big[ (u-&\varphi_\eps)(\tau\wedge\H_\eps,X^{t,\mathbf{x}}) \big]=\\
=& (u-\varphi)(t,\mathbf x) - \E\big[ (u-\varphi)(\tau\wedge\H_\eps,X^{t,\mathbf{x}}) \big] \\
&+ \E\big[ (\varphi_\eps-\varphi)(\tau\wedge\H_\eps,X^{t,\mathbf{x}})\big] - (\varphi_\eps-\varphi)(t,\mathbf x)   \\
\geq& \E\big[ (\varphi_\eps-\varphi)(\tau\wedge\H_\eps,X^{t,\mathbf{x}}) \big] - (\varphi_\eps-\varphi)(t,\mathbf x),
\end{split}
\end{equation}
where the last inequality follows from the fact that $\varphi\in\underline{\Ac} u(t,\mathbf{x})$. Since $\varphi$ and $\varphi_\eps$ belong to $C^{1,2}_{X}(\Lambda)$, we can write
\begin{equation}
\label{eq2}
\mathbb{E}\left[\varphi(\tau\wedge \H_\eps,X^{t,\mathbf{x}})\right]  =  \varphi(t,\mathbf{x})+\mathbb{E}\left[\int_t^{\tau\wedge \H_\eps}\mathcal{L}\varphi(s,X^{t,\mathbf{x}})ds
\right]
\end{equation}
and, clearly, we also have
\begin{equation}
\label{eq3}
\mathbb{E}\left[\varphi_\eps(\tau\wedge \H_\eps,X^{t,\mathbf{x}})\right]
=  \varphi_\eps(t,\mathbf{x})+\mathbb{E}\left[\int_t^{\tau\wedge \H_\eps} \big(\alpha_0+\eps\big) ds
\right].
\end{equation}
Plugging \eqref{eq2} and \eqref{eq3} into \eqref{eq1}, we obtain
\begin{multline*}
  (\varphi_\eps-u)(t,\mathbf x) - \E\big[ (\varphi_\eps-u)(\tau\wedge\H_\eps,X^{t,\mathbf{x}}) \big]  \leq\\
\leq  \mathbb{E}\left[\int_t^{\tau\wedge \H_\eps} \big( \mathcal{L}\varphi(s,X^{t,\mathbf{x}})-(\alpha_0+\eps)\big) ds\right] \leq  0,
\end{multline*}
where the last inequality follows by definition of $\H_\eps$.
It follows that $\varphi_\eps \in \underline{\mathcal{A}}(t,\mathbf{x})$,  hence that $\alpha_0+\eps\in\underline{\Jc} u(t,\mathbf{x})$, therefore
\[
-(\mathcal{L}\varphi(t,\mathbf{x})+\eps) - F(t,\mathbf x,u(t,\mathbf x))  =  -(\alpha_0+\varepsilon) - F(t,\mathbf x,u(t,\mathbf x))  \leq  0.
\]
By arbitrariness of $\eps$ we conclude.
\end{proof}

\begin{Remark}\label{rem:beta}
The map $\beta$ introduced in Definition \ref{def:smooth_functions} plays no role in the study of viscosity solutions of equation \eqref{eq:PPDE}. This can be seen, for instance, as a consequence of Proposition \ref{prop:jets}, since the definitions of sub/superjet $\underline{\Jc} u(t,\mathbf{x})$ and $\overline{\Jc} u(t,\mathbf{x})$ do not involve $\beta$. However,
$\beta$
 becomes relevant in the study of  nonlinear PPDEs such as those investigated in Section \ref{S:Extension} (see, notably, the definition of sub/superjet of Subsection \ref{SubS:Punctual} and the expression of $\mathcal L^a\varphi$ reported in \eqref{Lu_punctual}).
\end{Remark}

\subsection{Martingale characterization and stability}
\label{sec:stability}

 In the sequel, we shall consider the following conditions on $F$.
\begin{Assumption}\label{A:BSDE}
${}$
  \begin{enumerate}[(i)]
\setlength\itemsep{0.05em}
\item 
\label{2016-08-06:04}
\vskip-5pt
$F\colon\Lambda\times \mathbb{R}\to \mathbb{R}$  is continuous and satisfies the following growth condition: there exist $L>0, \ p\geq 0$ such that
\begin{equation}\label{A:BSDE_F}
|F(t,\mathbf{x},y)|\leq L(1+|  \mathbf{x}|_\infty^p+|y|),
\qquad  \forall\,(t,\mathbf x)\in \Lambda, \ \forall\,y\in \R.
\end{equation}
\item\label{2016-10-17:00}
$F$ is Lipschitz with respect to the third variable, uniformly in the other ones, i.e.\  there exists  $\hat L>0$ such that
\begin{equation}\label{F:lip}
|F(t,\mathbf{x},y)
-
F(t,\mathbf{x},y')|\leq \hat L |y-y'|,\qquad \forall\, (t,\mathbf{x})\in \Lambda, \ \forall\,y,y'\in \mathbb{R}.
\end{equation}
  \end{enumerate}
\end{Assumption}

We now state the main result of this section, the sub(super)martingale characterization for viscosity sub(super)solutions of   PPDE \eqref{eq:PPDE}.

\begin{Theorem}\label{lemma:pp1}
Let Assumptions \ref{A:SHDE}
 and \ref{A:BSDE}(\ref{2016-08-06:04}) hold and let $u\in  \cpol(\Lambda)$.
 The following facts are equivalent.
\begin{enumerate}[(i)]
\setlength\itemsep{0.05em}
\item 
\label{2016-06-21:02}
\vskip-5pt For every  $(t,\mathbf{x})\in\Lambda$, $s\in[t,T]$,
\begin{equation}\label{hjn}
 u(t,\mathbf{x}) \leq  \E\bigg[u(s,X^{t,\mathbf{x}}) + \int_t^{s} {F}(r,X^{t,\mathbf{x}}, u(r,X^{t,\mathbf{x}}))dr\bigg],
\end{equation}
(resp., $\geq$).
\item\label{2016-08-06:05} For every  $(t,\mathbf{x})\in\Lambda$, the process
\begin{equation}\label{eq:2014-05-08:ab}
\left(u(s,X^{t,\mathbf{x}})  +\int_t^s F(r,X^{t,\mathbf{x}}, u(r,X^{t,\mathbf{x}})) d r\right)_{s\in[t,T]}
\end{equation}
is a $(\mathcal{F}_s)_{s\in[t,T]}$-submartingale (resp., supermartingale).
\item\label{2016-06-21:03} $u$ is a viscosity subsolution (resp., supersolution) of  PPDE \eqref{eq:PPDE}.
\end{enumerate}
\end{Theorem}

To prove  Theorem \ref{lemma:pp1} we need some technical results from the optimal stopping theory. Let $\phi\in \cpol(\Lambda)$.
 Given $s\in[0,T]$, define
$$
\Lambda_s\defeq    \{(t,\mathbf{x})\in\Lambda\colon \ t\in[0,s]\}
$$
and  consider the  optimal stopping problems
\begin{equation}\label{OS}
\Psi_s(t,\mathbf{x})\coloneqq \sup_{\tau\in \mathcal{T}, \, \tau\geq t} \mathbb{E}\left[ \phi (\tau\wedge s,X^{t,\mathbf{x}})
\right], \qquad
\forall\, (t,\mathbf{x})\in\Lambda_s.
\end{equation}
Using the fact that $ \phi \in \cpol(\Lambda)$,
  we see, by  Corollary \ref{corr:2014-10-29:aa}, that the functional
$$\Lambda_s\rightarrow \R, \  (t,\mathbf{x})\mapsto \mathbb{E}\left[ \phi ((\tau\wedge s)\vee t,X^{t,\mathbf{x}})
\right]$$ is well-defined
 and continuous for every $\tau\in\mathcal{T}$. We deduce that
\begin{equation}
  \begin{split}
    \Psi_s(t,\mathbf{x}) =& \sup_{\tau\in \mathcal{T}, \, \tau\geq t} \mathbb{E}\left[
 \phi (\tau\wedge s,X^{t,\mathbf{x}})
\right]\\
=&
\sup_{\tau\in \mathcal{T}} \mathbb{E}\left[ \phi ((\tau\wedge s)\vee t,X^{t,\mathbf{x}})
\right], \quad  (t,\mathbf{x})\in\Lambda_s,
\end{split}
\end{equation}
is lower semicontinuous, as it is  supremum of continuous functions.
Define the continuation region
$$\mathcal{C}_s\defeq \{(t,\mathbf{x})\in\Lambda_s |  \  \Psi_s(t,\mathbf{x})> \phi (t,\mathbf{x})\}.$$
Due to the continuity of $ \phi $ and the lower semicontinuity of $\Psi_s$, it follows that $\Cc_s$ is an open subset of $\Lambda_s$.
From the general theory of optimal stopping we have the following result.

\begin{Theorem}\label{teo:OS}
Let Assumption \ref{A:SHDE} hold. Let $s\in [0,T]$,
$(t,\mathbf{x})\in\Lambda_s$ and define the random time
$\tau^*_{t,\mathbf{x}}  \defeq  \inf\big\{r\in[t,s]\colon(r,X^{t,\mathbf{x}})\notin \mathcal{C}_s\big\}
$, with the convention $\inf\emptyset = s$. Then
$\tau_{t,\mathbf{x}}^*$ is the first  optimal stopping time  for problem \eqref{OS}.
\end{Theorem}
\begin{proof}
 First of all, we notice that,
since $ \phi \in \cpol(\Lambda)$, by
(O\ref{2016-08-09:04})
we have, for every $ (t,\mathbf{x})\in\Lambda$,
\begin{equation}\label{est1}
\mathbb{E}\left[\sup_{r\in[t,T]}| \phi (r,X^{t,\mathbf{x}})|\right]<+\infty
\end{equation}
Now,
given $(t,\mathbf{x})\in \Lambda$, consider the \emph{window process}
$$[0,T]\times \Omega  \rightarrow  \mathbb W, \ (r,\omega)\mapsto \mathbb{X}^{t,\mathbf{x}}_r(\omega),$$
where
$$
\mbox{for\ }r\in[0,T]\mbox{ and }  s\in[0,T],\qquad
\mathbb{X}^{t,\mathbf{x}}_r(\omega)(s)\defeq
\begin{dcases}
\mathbf{x}_0, & \mbox{if }s+r<T,\\
X^{t,\mathbf{x}}_{s+r-T}(\omega),
&\mbox{if }s+r\geq T.
\end{dcases}
$$
Clearly this process is Markovian and we can write the optimal stopping problem in terms of it. Then, the standard theory of optimal stopping of Markovian processes allows to conclude. More precisely, taking into account \eqref{est1}, we can apply Corollary 2.9, Ch.\ I.1, of \cite{PS}.
\end{proof}

\begin{Lemma}\label{lemma:pp}

Let Assumption \ref{A:SHDE} hold. Let $u,f\in  \cpol(\Lambda)$  and assume that there exist $s\in [0,T]$ and $(t,\mathbf{x})\in\Lambda_s$, with $t<s$, such that
\begin{equation}\label{ass:s}
u(t,\mathbf{x})  >  \mathbb{E}\left[u(s,X^{t,\mathbf{x}})  +\int_t^s f(r,X^{t,\mathbf{x}}) d r\right] \qquad \mbox{(resp.\ $<$)}.
\end{equation}
Then there exists $(a,\mathbf{y})\in \Lambda_s$ such that $\varphi$ defined as $\varphi(s,\mathbf{z})\defeq -\int_{0}^s f(r,\mathbf{z})dr$ belongs to
$\underline{\mathcal{A}}u(a,\mathbf{y})$ (resp.\  belongs to $\overline{\mathcal{A}}u(a,\mathbf{y})$).
\end{Lemma}
\begin{proof}
We prove the claim for the ``sub-part''. The proof of the ``super-part'' is completely symmetric.

First, we notice that  $\varphi\in C^{1,2}_{X}(\Lambda)$, as it satisfies
 \eqref{eq:fct_Ito} with $\alpha=-f$
  and $\beta\equiv 0$.
Let us now focus on the maximum property. Consider the optimal stopping problem \eqref{OS}
with $ \phi (s,\mathbf{y})\coloneqq  u(s,\mathbf{y})+\int_0^sf(r,\mathbf{y})dr$, for $(s,\mathbf{y})\in\Lambda$,
 and let $\tau^*_{t,\mathbf{x}}$ be the stopping time of Theorem~\ref{teo:OS}. Due to  \eqref{ass:s} we have
$\mathbb{P}\{\tau^*_{t,\mathbf{x}}<s\}>0$. This implies that there exists $(a,\mathbf{y})\in\Lambda_s \setminus \mathcal{C}_s$. Hence
\begin{equation*}
  \begin{split}
    -u(a,\mathbf{y})
-\int_0^af(r,\mathbf{y})dr
&=
-  \phi (a,\mathbf{y})
=-\Psi_s(a,\mathbf{y})\\
&=\min_{\tau\in \mathcal{T}, \,\tau\geq a}\mathbb{E}\left[
-  u(\tau\wedge s,X^{a,\mathbf{y}})-\int_0^{\tau\wedge s} f(r,X^{a,\mathbf{y}}) dr \right],
\end{split}
\end{equation*}
and
the claim
is proved (\footnote{The role of the localizing stopping time  $\H$ in the definition of test functions is here played by $s$.}).
\end{proof}
\smallskip
\begin{proof}[\textbf{Proof of Theorem \ref{lemma:pp1}.}]
We prove the claim for the case of the subsolution/submartingale. The other claim can be proved in a completely symmetric way.

\emph{(\ref{2016-06-21:02})$\Rightarrow$ (\ref{2016-08-06:05})}.
We need to prove that, for every pair of times $(s_1,s_2)$ with $t\leq s_1\leq s_2\leq T$,
\begin{equation}\label{ass:s01}
u(s_1,X^{t,\mathbf{x}}) \leq    \mathbb{E}\left[u(s_2,X^{t,\mathbf{x}})  +\int_{s_1}^{s_2} F(r,X^{t,\mathbf{x}}, u(r,X^{t,\mathbf{x}})) d r \big| \mathcal{F}_{s_1}\right].
\end{equation}
 Using  \eqref{flow} and the equality $X^{s_1,X^{t,\mathbf{x}}}=X^{s_1,X^{t,\mathbf{x}}_{\cdot\wedge s_1}}$,  we have (\footnote{The flow property of $X^{t,\mathbf{x}}$ used here plays the role of the method based on regular conditional probability used in \cite{EkrenKellerTouziZhang, EkrenTouziZhang1, EkrenTouziZhang2}.}) 
  \begin{multline*}
   \mathbb{E}\left[u(s_2,X^{t,\mathbf{x}}) + \int_{s_1}^{s_2} F(r,X^{t,\mathbf{x}},u(r,X^{t,\mathbf{x}})) dr \big| \Fc_{s_1}\right] =\\
=   \mathbb{E}\bigg[u(s_2,X^{s_1,X^{t,\mathbf{x}}_{\cdot\wedge s_1}}) + \int_{s_1}^{s_2} F(r,X^{{s_1},X^{t,\mathbf{x}}_{\cdot\wedge s_1}},u(r,X^{{s_1},X^{t,\mathbf{x}}_{\cdot\wedge s_1}})) dr \big| \Fc_{s_1}\bigg].
\end{multline*}
Note  that $X^{s_1,\mathbf{x}'}$ is independent of $\mathcal{F}_{s_1}$ for each $\mathbf{x}'$ and  $X^{t,\mathbf{x}}_{\cdot\wedge {s_1}}$ is  $\mathcal{F}_{s_1}$-measurable.
Hence, using  \cite[Lemma 3.9,\ p.\ 55]{Baldi2000},
\begin{multline*}
   \mathbb{E}\bigg[u(s_2,X^{s_1,X^{t,\mathbf{x}}_{\cdot\wedge s_1}}) + \int_{s_1}^{s_2} F(r,X^{{s_1},X^{t,\mathbf{x}}_{\cdot\wedge s_1}},u(r,X^{{s_1},X^{t,\mathbf{x}}_{\cdot\wedge s_1}})) dr \big| \Fc_{s_1}\bigg]=
\\
=  \mathbb{E}\bigg[u(s_2,X^{s_1,\mathbf{x}'}) + \int_{s_1}^{s_2} F(r,X^{s_1,\mathbf x'},u(r,X^{s_1,\mathbf x'})) dr\bigg]\bigg|_{\mathbf{x}'=X^{t,\mathbf{x}}}
\end{multline*}
Now we conclude,  as \emph{(\ref{2016-06-21:02})} holds.

\emph{(\ref{2016-08-06:05})$\Rightarrow $(\ref{2016-06-21:03})}.
Let $\varphi\in\underline{\mathcal{A}}(t,\mathbf{x})$. Then, by definition of test function, there exists $\H\in \mathcal{T}$, with $\H>t$, such that
\begin{equation}\label{tre}
(\varphi-u)(t,\mathbf{x}) \leq  \mathbb{E}\left[(\varphi-u)\left(\tau\wedge \H,X^{t,\mathbf{x}}\right)\right],\qquad \forall\, \tau\in\mathcal{T},\ \forall \,t\in [0, \tau].
\end{equation}
As $\varphi\in C^{1,2}_{X}(\Lambda)$, we can write
\begin{equation}\label{trew}
\mathbb{E}\left[\varphi(\tau\wedge \H,X^{t,\mathbf{x}})\right] =
 \varphi(t,\mathbf{x})+\mathbb{E}\left[\int_t^{\tau\wedge \H}\mathcal{L}\varphi(s,X^{t,\mathbf{x}})ds
\right]
\end{equation}
Combining \eqref{tre}-\eqref{trew}, we get
$$
-\mathbb{E}\left[\int_t^{\tau\wedge \H}\mathcal{L}\varphi(s,X^{t,\mathbf{x}})ds
\right]\
\leq\
u(t,\mathbf{x})-\mathbb{E}\left[u(\tau\wedge \H,X^{t,\mathbf{x}})\right]
$$
or, equivalently,
\begin{multline}\label{eq:2014-05-08:aa}
-\mathbb{E}\left[\int_t^{\tau\wedge \H}\left(\mathcal{L}\varphi(s,X^{t,\mathbf{x}})
+F(s,X^{t,\mathbf{x}},u(s,X^{t,\mathbf{x}}))
\right)ds
\right]\leq
\\\leq
u(t,\mathbf{x})-\mathbb{E}\left[u(\tau\wedge \H,X^{t,\mathbf{x}})
+\int_t^{\tau\wedge \H}F(s,X^{t,\mathbf{x}},u(s,X^{t,\mathbf{x}}))ds
\right].
\end{multline}
Now observe that the submartingale assumption \eqref{eq:2014-05-08:ab} implies that the right-hand side  of \eqref{eq:2014-05-08:aa} is smaller than 0. Hence,  we can conclude by considering in \eqref{eq:2014-05-08:aa} stopping times of the form $\tau=t+\varepsilon$, with $\varepsilon>0$,  dividing by $\varepsilon$ and letting $\varepsilon\to 0^+$.

\emph{(\ref{2016-06-21:03})$ \Rightarrow $ (\ref{2016-06-21:02})}.
Let $\varepsilon>0$ and consider the function $u_\varepsilon(r,\mathbf{z})\defeq u(r,\mathbf{z})+\varepsilon r$. Assume that there exist $\varepsilon>0$, $(t,\mathbf{x})\in\Lambda$ and $t< s\leq T$ such that
\begin{equation}\label{ass:s1}
u_\varepsilon(t,\mathbf{x}) > \mathbb{E}\left[u_\varepsilon(s,X^{t,\mathbf{x}})+\int_t^s F(r,X^{t,\mathbf{x}},u(r,X^{t,\mathbf{x}})) dr\right].
\end{equation}
By applying Lemma \ref{lemma:pp}, we get that $\varphi^\varepsilon$ defined as    $\varphi^\varepsilon(r,\mathbf{z})\defeq  \varphi(r,\mathbf{z})-\varepsilon r$, where $\varphi$ is defined as in Lemma \ref{lemma:pp} taking $f(r,\cdot)\defeq  F(r,\cdot,u(r,\cdot))$, belongs to $\underline{\mathcal{A}}u (a,\mathbf{y})$ for some $(a,\mathbf{y})$. By the viscosity subsolution property of $u$, we then obtain the contradiction $\varepsilon\leq 0$. Hence we deduce that
\begin{equation}\label{ass:s22}
u_\varepsilon(t,\mathbf{x}) \leq \mathbb{E}\left[u_\varepsilon(s,X^{t,\mathbf{x}}))+\int_t^s F(r,X^{t,\mathbf{x}},u(r,X^{t,\mathbf{x}})) dr\right].
\end{equation}
As $\varepsilon$ is arbitrary in the argument above, we can take $\varepsilon\downarrow 0$ in \eqref{ass:s22}, getting  \eqref{hjn}.
\end{proof}

Theorem \ref{lemma:pp1} has several important consequences that we will investigate in the rest of the section.

\subsection{Comparison principle}
\label{sec:comparison}

In this section we provide a comparison result for viscosity sub and supersolutions of  \eqref{eq:PPDE}, which, through the use of a technical lemma provided here, turns out to be a  corollary of the characterization of Theorem \ref{lemma:pp1}.
\begin{Lemma}\label{lemma:sub}
Let $Z\in \mathcal{H}^1_\mathcal{P}$ and $g\colon[0,T]\times\Omega\times  \mathbb{R}\rightarrow \mathbb{R}$ be such that $g(\cdot,\cdot,z)\in L^1_\mathcal{P}$, for all $z\in\R$, and, for some constant $C_g>0$,
 \begin{equation}\label{ass:g}
 g(\cdot,\cdot,z)  \leq  C_g|z|, \qquad\forall\,z\in\mathbb{R}.
 \end{equation}
Assume that the process
\begin{equation}\label{PRZ}
\left(Z_s+\int_t^s g(r,\cdot,Z_r) dr\right)_{s\in[t,T]}
\end{equation}
is an $(\mathcal{F}_s)_{s\in[t,T]}$-submartingale. Then $Z_T\leq0$, $\P$-a.s., implies  $Z_t\leq0$, $\P$-a.s..
\end{Lemma}
\begin{proof}
Let $Z_T\leq0$ and
define
$$
\tau^*  \defeq  \inf\left\{s\geq t \colon Z_s\leq
0\right\}.
$$
Clearly $t\leq\tau^*\leq T$ and, since $Z$ has continuous trajectories,
\begin{equation}\label{iiuy}
Z_{\tau^*}  \leq  0.
\end{equation}
Using the submartingale property, we obtain
\begin{equation}\label{eq:2014-04-10:aa}
Z_s  \leq  \mathbb{E}\left[Z_{\tau^*\vee s}+\int_s^{\tau^*\vee s} g(r,\cdot,Z_r)dr  \bigg|  \mathcal{F}_s\right], \qquad
\forall\,s\in[t,T].
\end{equation}
Multiplying \eqref{eq:2014-04-10:aa} by the $\mathcal{F}_s$ -measurable random variable $\mathbf{1}_{\{s\leq \tau^*\}}$, and recalling \eqref{iiuy}, we find
\begin{equation}\label{eq:2014-04-10:aa0}
  \begin{split}
    \mathbf{1}_{\{s\leq \tau^*\}} Z_s &\leq  \mathbb{E}\left[\mathbf{1}_{\{s\leq \tau^*\}}\left(Z_{\tau^*}+\int_s^{\tau^*} g(r,\cdot,Z_r)dr\right)  \bigg|  \mathcal{F}_s\right] \\
&\leq
\mathbb{E}\left[\mathbf{1}_{\{s\leq\tau^*\}} \int_s^{\tau^*} g(r,\cdot,Z_r)dr  \bigg|  \mathcal{F}_s\right]\\
&=
\mathbb{E}\left[\int_s^{T} \mathbf{1}_{\{r \leq \tau^*\}} g(r,\cdot,Z_r)dr  \bigg|  \mathcal{F}_s\right],
 \qquad \forall\, s\in[t,T].
\end{split}
\end{equation}
Now from \eqref{ass:g} and the definition of $\tau^*$, we have
$$
 \mathbf{1}_{\{r \leq \tau^*\}} g(r,\cdot,Z_r)  \leq   \mathbf{1}_{\{r \leq\tau^*\}} C_g |Z_r| =  \mathbf{1}_{\{r\leq \tau^*\}} C_g Z_r, \qquad\forall\,r\in[t,T].
$$
Plugging the latter inequality into \eqref{eq:2014-04-10:aa0} and taking the conditional expectations with respect to $\mathcal{F}_t$, we obtain
\begin{equation}\label{oool}
\mathbb{E}\left[\mathbf{1}_{\{s\leq \tau^*\}} Z_s|\mathcal{F}_t\right] \leq  C_g\int_s^{T}  \mathbb{E}[\mathbf{1}_{\{r \leq \tau^*\}}Z_r|\mathcal{F}_t]dr,
 \qquad \forall\,s\in[t,T].
\end{equation}
Now, setting $h(s)\defeq  \mathbb{E}[\mathbf{1}_{\{s\leq \tau^*\}} Z_s|\mathcal{F}_t]$, \eqref{oool} becomes
\begin{equation}\label{pplk1}
h(s)  \leq  C_g\int_s^T h(r) dr, \qquad \forall\,s\in[t,T].
\end{equation}
Gronwall's Lemma yields $h(s)\leq0$, for all $s\in[t,T]$.
In particular, for $s=T$, we obtain, $\P$-a.s., $
Z_t  =  \mathbb{E}[Z_t|\mathcal{F}_t]  =  \mathbb{E}[\mathbf{1}_{\{t \leq \tau^*\}}Z_t|\mathcal{F}_t]  =  h(t)  \leq  0.$
\end{proof}

\begin{Corollary}[Comparison principle]
\label{C:CompPrinciple}
Let Assumptions \ref{A:SHDE} and \ref{A:BSDE} hold.
Let  $u^{(1)}\in \cpol(\Lambda)$ (resp.\ $u^{(2)}\in \cpol(\Lambda)$) be a viscosity subsolution (resp.\ supersolution) to PPDE \eqref{eq:PPDE}.
If $u^{(1)}(T,\cdot)\leq u^{(2)}(T,\cdot)$ on $\mathbb W$, then $u^{(1)}\leq u^{(2)}$ on $\Lambda$.
\end{Corollary}
\begin{proof}
Let $(t,\mathbf{x})\in \Lambda$.
Set
$$
g(r,\omega, z)  \defeq  F(r, X^{t,\mathbf{x}}(\omega), z+u^{(2)}(r,  X^{t,\mathbf{x}}(\omega)))-  F(r, X^{t,\mathbf{x}}(\omega), u^{(2)}(r,  X^{t,\mathbf{x}}(\omega)))
$$
and
$$
Z_r(\omega)  \defeq  
(u^{(1)}
- u^{(2)})
(r,  X^{t,\mathbf{x}}(\omega)).
$$
Due to Assumption \ref{A:BSDE}, the map $g$ satisfies the assumptions of Lemma \ref{lemma:sub}. Moreover, by using the implication 
\emph{(\ref{2016-06-21:03})}$\Rightarrow$\emph{(\ref{2016-08-06:05})}
of  Theorem \ref{lemma:pp1} and the inequality $u^1(T,\cdot)-u^2(T,\cdot)\leq0$, we see that $Z$ satisfies the assumption of  Lemma \ref{lemma:sub}. Then the claim follows as, $\P$-a.e.\ $\omega\in\Omega$,
$$
 (u^{(1)}
- u^{(2)})
(t,  X^{t,\mathbf{x}}(\omega)) =   
(u^{(1)}
- u^{(2)})
(t,  \mathbf{x}).\eqno\qed
$$
\let\qed\relax
\end{proof}

\subsection{Equivalence with mild solutions}\label{sub:mild-visc}

The concept of mild solution has been mainly introduced and used to study infinite-dimensional or functional PDEs.
Roughly speaking, mild solutions  are solutions to integral equations treating the nonlinearity of the PDE as a perturbation of the linear PDE. This concept turns out to be very suitable in the infinite-dimensional framework, as it allows to bypass the complications arising from the unboundedness of the linear term $\langle Ax, \partial_x v\rangle$ in the equation (for details we refer to \cite[Ch.\ 13]{DPZ02} or to \cite[Ch.\ 4]{FGS14} and references therein;  see also \cite{FedericoGozzi16} in the elliptic case).
Moreover such concept is also suitable in the non-Markovian framework, as it allows to bypass the difficulties related to the characterization of the infinitesimal generator of the process (see \cite{FMT}).

The next definition adapts to our context the concept of mild solutions for functional PDEs introduced in \cite{FMT}.

\begin{Definition}
\label{df:mildnew}
A function $u\in  \cpol(\Lambda)$ is  a \emph{mild solution} to \eqref{eq:PPDE} if
\begin{equation}\label{eq:milddefMarkov}
u(t,\mathbf{x}) =  P_{t,T}[u(T,\cdot)](\mathbf{x}) +
\int_t^{T} P_{t,s}\left[{F}(s,\cdot, u(s,\cdot))\right](\mathbf{x})ds, \qquad  \forall\,(t,\mathbf{x})\in \Lambda,
\end{equation}
where
$$
P_{t,s}[\phi](\mathbf{x})\coloneqq \E\left[\phi(X_{\cdot\wedge s}^{t,\mathbf{x}})\right], \ \ \ \forall\,(t,\mathbf{x})\in \Lambda,\ \forall\, s\in[t,T], \ \forall\,\phi\in C(\mathbb{W};\R).$$
\end{Definition}
Notice that we recover the standard definition for the Markovian case when $\phi(\mathbf{x})=\tilde \phi(\mathbf{x}_T)$ for some $\tilde \varphi\in C(H)$.

By definition of $P_{t,s}$ and non-anticipativity of $u,F$, we see that $u\in\cpol(\Lambda)$ is a mild solution to \eqref{eq:PPDE} if and only if
   \begin{equation}\label{eq:milddefbis}
 u(t,\mathbf{x}) =  \E\bigg[u(T,X^{t,\mathbf{x}}) + \int_t^{T} {F}(r,X^{t,\mathbf{x}}, u(r,X^{t,\mathbf{x}}))dr\bigg], \ \ \ \forall\,(t,\mathbf{x})\in \Lambda.
\end{equation}
\begin{Proposition}\label{mildmild}
 $u\in\cpol(\Lambda)$ is a mild solution to \eqref{eq:PPDE} if {and only if} 
\begin{equation}\label{eq:milddef}
 u(t,\mathbf{x}) =  \E\bigg[u(s,X^{t,\mathbf{x}}) + \int_t^{s} {F}(r,X^{t,\mathbf{x}}, u(r,X^{t,\mathbf{x}}))dr\bigg],
\end{equation}
for all $(t,\mathbf{x})\in\Lambda$,  $s\in[t,T]$.
\end{Proposition}
\begin{proof}
 ($\Leftarrow$)  This is immediate by definition of $P_{t,\cdot}$ by taking $s=T$ in \eqref{eq:milddef}.

   ($\Rightarrow$)  Let $u$ be a mild solution to \eqref{eq:PPDE}.
Then
\eqref{eq:milddefbis} holds.
Using \eqref{flow} and  \cite[Lemma 3.9, p.\ 55]{Baldi2000} we get for every $s\in[t,T]$
\begin{equation*}
  \begin{split}
      {u}(s,X^{t,\mathbf{x}})
&  =\mathbb{E}\left[\xi(X^{s,\mathbf{y}})+\int_s^{T}F(r,X^{s,\mathbf{y}},{u}(r,X^{s,\mathbf{y}}))dr\right]_{|\mathbf{y}=X^{t,\mathbf{x}}}\\
&  =\mathbb{E}\left[\left.\zeta(X^{s,X^{t,\mathbf{x}}})+\int_s^T F(r, X^{s,X^{t,\mathbf{x}}}
      ,{u}(r, X^{s,X^{t,\mathbf{x}}}
      ))dr \right| \mathcal{F}_s\right]
\\
 & =\mathbb{E}\left[\left.\xi(X^{t,\mathbf{x}})+\int_s^T F(r, X^{t,\mathbf{x}}
      ,{u}(r, X^{t,\mathbf{x}}
      ))dr \right| \mathcal{F}_s\right].
\end{split}
\end{equation*}
Hence
\begin{equation}\label{mmmnn}
\mathbb{E}\left[  {u}(s,X^{t,\mathbf{x}})\right]\
=\
\mathbb{E}\left[\xi(X^{t,\mathbf{x}})+\int_s^T F(r, X^{t,\mathbf{x}}
      ,{u}(r, X^{t,\mathbf{x}}
      ))dr\right].
\end{equation}
We conclude plugging \eqref{mmmnn} into  \eqref{eq:milddefbis}.
\end{proof}

\begin{Corollary}\label{visc-mild}
$u\in \cpol(\Lambda)$ is a viscosity solution to \eqref{eq:PPDE} 
if and only if it is a mild solution to \eqref{eq:PPDE}.
\end{Corollary}
\begin{proof} It follows from the equivalence
\emph{(\ref{2016-06-21:02})$\Leftrightarrow$(\ref{2016-06-21:03})}
in Theorem \ref{lemma:pp1} and from Proposition \ref{mildmild}. \end{proof}

Although the two concepts of solution (\emph{viscosity} and \emph{mild}) turn out to coincide in the case of PPDE~\eqref{eq:PPDE}\ (\,\footnote{Actually, even the  proof of existence of viscosity solutions that we provide in Theorem \ref{teo:main} is based on a fixed point argument, typical when dealing with mild solutions.})  and, possibly, in some other cases involving gradient nonlinearities with \emph{structure condition} (see  Subsection \ref{SubS:Punctual}), we emphasize that our concept of viscosity solution is, to many extents, genuinely different from the concept of mild solution. The latter has a global  nature,  the former has a local nature. The local nature of our notion turns out to be fundamental to address general possibly degenerate and fully nonlinear equations of Hamilton-Jacobi-Bellman type by standard viscosity methods based on Dynamic Programming. This is the way we proceed in the next Section  \ref{S:Extension} to address this type of equations and prove existence by ``local'' arguments.

\subsection{Existence, uniqueness and stability of viscosity solutions}
\label{sec:exun}

In this section we provide
existence and uniqueness of viscosity solutions (hence, by Subsection \ref{sub:mild-visc}, also of mild solutions).

\smallskip
For $p\geq 0$, we denote by $C_p(\mathbb{W})$ the vector space of continuous functions $\xi\colon \mathbb{W}\rightarrow \mathbb{R}$ such that
\begin{equation}
\label{A:BSDE_xi}
  |\xi|_{C_p(\mathbb{W})} 
\coloneqq
\sup_{\mathbf{x}\in \mathbb{W}}
\frac{|\xi(\mathbf{w})|}
{
1+|\mathbf{x}|_\infty^p
}
<\infty.
\end{equation}

\begin{Proposition}\label{propp:2014-05-09:aa}
 Let Assumption \ref{A:SHDE} hold,
let $\xi\in C_p(\mathbb{W})$, for some $p\geq 0$, and let
 Assumption~\ref{A:BSDE}
hold with
 the same  $p$.
Then there exists a unique $\hat u\in C_p(\Lambda)$ viscosity solution to \eqref{eq:milddef} with terminal condition \eqref{terminal}.
\end{Proposition}

\begin{proof}
 \emph{Step I.}
Fix a function $\zeta\in C_p(\Lambda)$, and let $0\leq a\leq b\leq T$.
Consider the nonlinear operator $\Gamma\colon C_p(\Lambda)\to C_p(\Lambda), \ u\mapsto \Gamma (u)$,
defined by
\begin{equation}
  \label{eq:2014-05-15:aa}
\Gamma(u)(t,\mathbf{x})\defeq
\mathbb{E}\left[\zeta(X^{t,\mathbf{x}})+\mathbf{1}_{[a,b]}(t)\int_{t}^{b} F(s,X^{t,\mathbf{x}},
  u(s,X^{t,\mathbf{x}}))ds\right], \quad\forall\,(t,\mathbf{x})\in\Lambda.
\end{equation}
First we note that actually $\Gamma$ is well defined and maps $C_p(\Lambda)$ into itself:  it follows from Assumption \ref{A:BSDE} and Corollary \ref{corr:2014-10-29:aa}.

We now show that there exists $\varepsilon>0$ such that, if $b-a<\varepsilon$, then $\Gamma$ is a contraction on $C_p(\Lambda)$, hence admits  a unique fixed point.
Let $u,v\in C_p(\Lambda)$. Using Assumption \ref{A:BSDE}\emph{(\ref{2016-10-17:00})}
\begin{equation*}
  \begin{split}
      |\Gamma(u)(t,&\mathbf{x})-\Gamma(v)(t,\mathbf{x})|\leq\\
&\leq
  \mathbb{E}\left[\mathbf{1}_{[a,b]}(t)\int_{t}^b\left | F(s,X^{t,\mathbf{x}},
      u(s,X^{t,\mathbf{x}})) -F(s,X^{t,\mathbf{x}},
      v(s,X^{t,\mathbf{x}})\right|
    ds\right]\\
  &\leq
  \hat L\mathbb{E}\left[
\mathbf{1}_{[a,b]}(t)\int_{t}^b
|u(s,X^{t,\mathbf{x}})-v(s,X^{t,\mathbf{x}})|ds\right]\\
 & \leq  \hat L| u-v|_{C_p(\Lambda)}\mathbb{E}\left[
\mathbf{1}_{[a,b]}(t)\int_{t}^b
\left(1+| X^{t,\mathbf{x}}|_{\infty}^p\right)d s\right]\\
 & \leq  \hat L| u-v|_{C_p(\Lambda)}
\mathbf{1}_{[a,b]}(t)\int_{t}^b
\left(1+M(1+|
      \mathbf{x}|_{\infty}^p)\right)d s\\
      &\leq \varepsilon \hat L
  \left(1+M)(1+| \mathbf{x}|_{\infty}^p\right)| u-v|_{C_p(\Lambda)}
\end{split}
\end{equation*}
which yields
\begin{equation}\label{eq:2014-11-27:aa}
| \Gamma(u)-\Gamma(v)|_{C_p(\Lambda)} \leq  \varepsilon \hat L(1+M)| u-v|_{C_p(\Lambda)}.
\end{equation}
Thus, $\Gamma$ is a contraction whenever $\varepsilon<(\hat L(1+M))^{-1}$. For such $\varepsilon$, it admits a unique fixed point $\hat u$:
\begin{equation}
  \label{eq:2014-05-15:aa1}
\hat{u}(t,\mathbf{x}) =\
\mathbb{E}\left[\zeta(X^{t,\mathbf{x}})+\mathbf{1}_{[a,b]}(t)\int_{t}^{b} F(s,X^{t,\mathbf{x}},
  \hat{u}(s,X^{t,\mathbf{x}}))ds\right],  \quad \forall\,(t,\mathbf{x})\in\Lambda.
\end{equation}
\vskip0.2cm\noindent \emph{Step II.} We prove that, if a function $\hat{u}$ satisfies  \eqref{eq:2014-05-15:aa1} for $(t,\mathbf{x})\in \Lambda$, $a\leq t\leq b$, then it  also satisfies,  for every $(t,\mathbf{x})\in\Lambda$ and every  $(s,\mathbf{x})\in\Lambda$ with  $a\leq t\leq s\leq b$, the equality
\begin{equation}
  \label{eq:2014-05-15:ab}
\hat{u}(t,\mathbf{x})=
\mathbb{E}\left[\hat{u}(s,X^{t,\mathbf{x}})+\int_{t}^{s} F(r,X^{t,\mathbf{x}},
  \hat{u}(r,X^{t,\mathbf{x}}))dr\right].
\end{equation}
Indeed, using \eqref{flow} and  \cite[Lemma 3.9, p.\ 55]{Baldi2000}
\begin{equation*}
  \begin{split}
      \hat{u}(s,X^{t,\mathbf{x}})
&  =\mathbb{E}\left[\zeta(X^{s,\mathbf{y}})+\int_s^{b}F(r,X^{s,\mathbf{y}},\hat{u}(r,X^{s,\mathbf{y}}))dr\right]_{|\mathbf{y}=X^{t,\mathbf{x}}}\\
&  =\mathbb{E}\left[\left.\zeta(X^{s,X^{t,\mathbf{x}}})+\int_s^b F(r, X^{s,X^{t,\mathbf{x}}}
      ,\hat{u}(r, X^{s,X^{t,\mathbf{x}}}
      ))dr \right| \mathcal{F}_s\right]
\\
 & =\mathbb{E}\left[\left.\zeta(X^{t,\mathbf{x}})+\int_s^b F(r, X^{t,\mathbf{x}}
      ,\hat{u}(r, X^{t,\mathbf{x}}
      ))dr \right| \mathcal{F}_s\right].
\end{split}
\end{equation*}
Hence
$$
\mathbb{E}\left[  \hat{u}(s,X^{t,\mathbf{x}})\right]\
=\
\mathbb{E}\left[\zeta(X^{t,\mathbf{x}})+\int_s^b F(r, X^{t,\mathbf{x}}
      ,\hat{u}(r, X^{t,\mathbf{x}}
      ))dr\right]
$$
and we conclude by  \eqref{eq:2014-05-15:aa1}.
\vskip0.2cm\noindent\emph{Step III.}
In this step we conclude the proof. Let $a,b$ as in \emph{Step I} and  let us assume, without loss of generality, that $T/(b-a)=n\in\mathbb{N}$. By \emph{Step I}, there exists a unique $\hat u_n\in C_p(\Lambda)$ satisfying
\begin{equation*}
\hat u_n(t,\mathbf{x}) \defeq
\mathbb{E}\left[\xi(X^{t,\mathbf{x}})+\mathbf{1}_{[T-(b-a),T]}(t)\int_{t}^{T} F(s,X^{t,\mathbf{x}},
  \hat u_n(s,X^{t,\mathbf{x}}))ds\right],
\end{equation*}
for all $(t,\mathbf{x})\in\Lambda$.
With a backward recursion argument,  using \emph{Step I}, we can find (uniquely determined) functions $\hat u_i\in C_p(\Lambda)$, $i=1,\ldots, n$, such that
\begin{equation*}
  \begin{split}
    \hat u_{i-1}(t,\mathbf{x})\defeq
\mathbb{E}&\left[\hat{u}_{i}(i(b-a),X^{t,\mathbf{x}})\phantom{\int_t^{i(b-a)}}\right.\\
&\left.\phantom{\int}+\mathbf{1}_{[(i-1)(b-a),i(b-a)]}(t)\int_{t}^{i(b-a)} F(s,X^{t,\mathbf{x}},
  \hat u_i(s,X^{t,\mathbf{x}}))ds\right],
\end{split}
\end{equation*}
for all $(t,\mathbf{x})\in\Lambda$. Now define $\hat u(t,\cdot)=\sum_{1\leq i\leq n}
\mathbf{1}_{[(i-1)(b-a),i(b-a))}(t)
\hat u_i(t,\cdot)+\mathbf{1}_{\{T\}}(t)\xi(\cdot)$.
To conclude the existence, we use recursively \emph{Step II} to prove that $\hat u$ satisfies \eqref{eq:milddef} with terminal condition \eqref{terminal}.

Uniqueness follows from local uniqueness. Indeed, let $\hat u, \hat v$ be two solutions in $C_p(\Lambda)$ of \eqref{eq:milddef}-\eqref{terminal} and  define
$$
T^*\defeq\sup\left\{t\in[0,T]\colon \sup_{\mathbf{x}\in \mathbb W} |\hat u(t,\mathbf{x})-\hat v(t,\mathbf{x})|>0\right\},
$$
with the convention $\sup\emptyset =0$. By  continuity of $\hat u, \hat v$, and since $\hat u(T,\cdot)=\hat v(T,\cdot)$, we have $\hat u(t,\cdot)\equiv \hat v(t,\cdot)$ for every $t\in[T^*,T]$. 
In order to prove that
 $T^*=0$,
we assume to the contrary that $T^*>0$.
As done in \emph{Step II}, one can prove that both $\hat u $ and $\hat v$ satisfy \eqref{eq:2014-05-15:ab}.
In particular, if we consider the definition \eqref{eq:2014-05-15:aa}  with $\zeta(\cdot)=\hat{u}(T^*,\cdot)=\hat{v}(T^*,\cdot)$, $a=0\vee(T^*-\varepsilon)$, $b=T^*$, where $\varepsilon<(\hat L(1+M))^{-1}$, we have
$$
\Gamma(\hat u)(t,\mathbf{x})=\hat u(t,\mathbf{x}) \ \
\mbox{and}\ \
\Gamma(\hat v)(t,\mathbf{x})=\hat
v(t,\mathbf{x}), \quad \forall\,(t,\mathbf{x})\in\Lambda,\ \forall\, t\in [T^*-\varepsilon
T^*].
$$
Then,  recalling  \eqref{eq:2014-11-27:aa}, we get a contradiction and conclude.
\end{proof}
\begin{Theorem}\label{teo:main}
 Let Assumption \ref{A:SHDE} hold,
let $\xi\in C_p(\mathbb{W})$, for some $p\geq 0$, and let
 Assumption~\ref{A:BSDE}
hold.
  Then  PPDE~\eqref{eq:PPDE} has a unique viscosity solution in the space $\cpol(\Lambda)$  satisfying the terminal condition \eqref{terminal}.
Moreover such solution belongs to the space $C_p(\Lambda)$, where $p$ is such that both \eqref{A:BSDE_F} and \eqref{A:BSDE_xi} hold.
\end{Theorem}
\begin{proof}
Uniqueness is consequence of the comparison principle (Corollary \ref{C:CompPrinciple}).
Existence (and uniqueness) in $C_p(\Lambda)$ is consequence of Theorem \ref{lemma:pp1} and Proposition \ref{propp:2014-05-09:aa}.
\end{proof}

\begin{Remark} \label{Urem:UC}
If there exists a modulus of continuity $w_F$ such that
$$
|F(t,\mathbf{x},y)-F(t',\mathbf{x}',y')|  \leq  w_F(\mathbf{d}_\infty((t,\mathbf{x}), (t',\mathbf{x}')))+\hat L |y-y'|,
$$
then $\Gamma$ defined in \eqref{eq:2014-05-15:aa} maps $UC(\Lambda)$ into itself. Hence, if $\xi$ is uniformly continuous  and the condition above on $F$ holds, then the solution $\hat{u}$ belongs to $UC(\Lambda)$.
\end{Remark}

\begin{Remark}[Nonlinear Feynman-Kac formula]\label{rem:back}
Existence and uniqueness of solutions to the functional equation \eqref{eq:milddef} could be deduced from the theory of backward stochastic differential equations in Hilbert spaces. Indeed, another way to solve the functional equation \eqref{eq:milddef} is to consider the following backward stochastic differential equation
\begin{equation}
\label{BSDE}
Y_s  =  \xi(X^{t,\mathbf{x}}) + \int_s^T F(r,X^{t,\mathbf{x}},Y_r) dr - \int_s^T Z_r dW_r, \qquad s\in[t,T].
\end{equation}
Then, it follows from \cite[Prop.\ 4.3]{fuhrmantess02} that, under Assumptions \ref{A:SHDE}, \ref{A:BSDE}, and
if $\xi\in C_p(\mathbb{W})$, $p\geq 2$,
then
for any $(t,\mathbf{x})\in\Lambda$ there exists a unique solution $(Y_s^{t,\mathbf{x}},Z_s^{t,\mathbf{x}})_{s\in[0,T]}\in \mathcal{H}^2_\mathcal{P}(\mathbb{R})\times L_{\cal{P}}^2(H^*)$  to \eqref{BSDE} and it  belongs to the space $\in \mathcal{H}^p_\mathcal{P}(\mathbb{R})\times L_{\cal{P}}^p(H^*)$. Such solution  can be viewed as a Sobolev solution to PPDE \eqref{eq:PPDE}
(see e.g.\ \cite{BarlesLesigne}). Moreover $Y_t^{t,\mathbf{x}}$ is constant, so we may define
\begin{equation}
\label{hatu}
\hat u(t,\mathbf{x})  \defeq   Y_t^{t,\mathbf{x}}  = \E\bigg[\xi(X^{t,\mathbf{x}}) + \int_t^T F(s,X^{t,\mathbf{x}},Y_s^{t,\mathbf{x}}) ds\bigg], \ \ \ (t,\mathbf{x})\in\Lambda.
\end{equation}
 It can be shown, using the flow property of $X^{t,\mathbf x}$ and the uniqueness of the backward equation \eqref{BSDE}, that $Y_s^{t,\mathbf x}=\hat u(s,X^{t,\mathbf x})$ for all $s\in[t,T]$, $\P$-almost surely. Moreover, using the backward equation \eqref{BSDE}, the regularity of $\xi$ and $F$, and the flow property of $X^{t,\mathbf{x}}$ with respect to $(t,\mathbf{x})$, we can prove that $\hat u\in \cpol(\Lambda)$. This implies that $\hat u$ solves the functional equation \eqref{eq:milddef}   and it coincides with
the function of Proposition \ref{propp:2014-05-09:aa}. Vice versa, we can also prove an existence and uniqueness result for the backward equation \eqref{BSDE} if we know that there exists a unique solution $\hat u\in \cpol(\Lambda)$ to the functional equation \eqref{eq:milddef}. In conclusion, $\hat u$ admits a nonlinear Feynman-Kac representation formula through a non-Markovian forward-backward stochastic differential equation given by:
\[
\begin{cases}
X_s  =  e^{(s-t)A}\mathbf x_t + \int_t^s e^{(s-r)A}b(r,X)dr + \int_t^s e^{(s-r)A} \sigma(r,X) dW_r,    & s\in[t,T], \\
X_s = \mathbf x_s, & s\in[0,t], \\
Y_s  =  \xi(X) + \int_s^T F(r,X,Y_r) dr - \int_s^T Z_r dW_r, &s\in[0,T].
\end{cases}
\]
\end{Remark}

As a direct consequence of the martingale characterization in Theorem \ref{lemma:pp1}, we also get the following stability result.

\begin{Proposition}\label{propp:2014-10-boh}
Let the assumptions of Proposition \ref{prop:stabSDE} hold. Let Assumption~\ref{A:BSDE}(\ref{2016-08-06:04}) hold and assume that it also holds, for each $n\in\N$, for analogous objects $F_n$ with the same constants $L,p$.
 Let  $\{u_n, n\in \mathbb{N}\}$ be a bounded subset of $C_p(\Lambda)$, for some $p\geq 0$,  and let $u\in C_p(\Lambda)$.
Assume that the following convergences hold:
\begin{enumerate}[(i)]
\setlength\itemsep{0.05em}
\item \vskip-5pt
 $F_n(s,\cdot,y)\rightarrow F(s,\cdot,y)$ uniformly on compact subsets of $\mathbb W$ for each $(s,y)\in [0,T]\times \R$.
\item $u_n(s,\cdot)\rightarrow u(s,\cdot)$ uniformly on compact subsets of $\mathbb W$ for each $s\in [0,T]$.
\end{enumerate}
Finally, assume that, for each $n\in\N$,  the function $u_n$ is viscosity subsolution (resp., supersolution) to PPDE \eqref{eq:PPDE} associated to the coefficients $A_n,b_n,\sigma_n, F_n$.
Then u is a viscosity subsolution (resp., supersolution) to \eqref{eq:PPDE} associated to the coefficients $A,b,\sigma, F$.

\end{Proposition}
\begin{proof}
For any $n>0$ and $(t,\mathbf x)\in\Lambda$, it follows from Proposition \ref{Exist} that there exists a unique mild solution $X^{(n),t,\mathbf x}$ to   SDE \eqref{SHDE} with coefficients $A_n$, $b_n$, $\sigma_n$.
By Proposition \ref{prop:stabSDE}
\begin{equation}
\label{StabMild}
\lim_{n\rightarrow \infty}X^{(n),t,\mathbf x}
=
 X^{t,\mathbf x} \ \mbox{in}\ {\mathcal{H}_\mathcal{P}^p(H)}, \ \forall\,(t,\mathbf x)\in\Lambda.
\end{equation}
Since $u_n$ is a viscosity subsolution (the supersolution case can be proved in a similar way) to PPDE \eqref{eq:PPDE}, from 
 Theorem~\ref{lemma:pp1}\emph{(\ref{2016-06-21:02})} we have, for every $(t,\mathbf{x})\in\Lambda$ with $t<T$,
\begin{equation}
\label{Ineq_u_eps}
u_n(t,\mathbf{x})  \leq  \E\bigg[u_n(s,X^{(n),t,\mathbf{x}}) + \int_t^{s} F_n(r,X^{(n),t,\mathbf{x}}, u_n(r,X^{(n),t,\mathbf{x}}))dr\bigg],
\end{equation}
for all $s\in[t,T]$.
In view of the same theorem, to conclude the proof we just need to prove,  letting $n\rightarrow \infty$, that the same inequality holds true when $u_n,\ F_n $ and $X^{(n),t,\mathbf{x}}$ are replaced by $u,\ F $ and $X^{t,\mathbf{x}}$, respectively.

Clearly the left-hand side of the above inequality tends to $u(t,\mathbf{x})$ as $n\rightarrow \infty$. Let us consider the right-hand side. From \eqref{StabMild}, up to extracting a subsequence, we have for $\P$-a.e.\ $\omega$, the convergence $X^{(n),t,\mathbf{x}}(\omega)\to X^{t,\mathbf{x}}(\omega)$ in $\mathbb W$.
Fix such an $\omega$. Then
$$
\mathcal{S}(\omega)\coloneqq  \left\{X^{(n),t,\mathbf{x}}(\omega)\right\}_{n\in\mathbb{N}}\bigcup  \left\{X^{t,\mathbf{x}}(\omega)\right\}
$$
is a compact subset of $\mathbb W$. Then, for each $s\in [t,T]$,
\begin{multline*}
 |u_n(s,X^{(n),t,\mathbf{x}}(\omega))-
    u(s,X^{t,\mathbf{x}}(\omega))|\leq\\
\leq
    \sup_{\mathbf{z}\in \mathcal{S}(\omega)}|u_n(s,\mathbf{z})-
    u(s,\mathbf{z})|
    +
|u(s,X^{(n),t,\mathbf{x}}(\omega))-
    u(s,X^{t,\mathbf{x}}(\omega))|\stackrel{n\rightarrow\infty}{\longrightarrow} 0
\end{multline*}
because $u_n(s,\cdot)\to u(s,\cdot)$ on compact subsets of $\mathbb W$,  $u$ is continuous and $X^{(n),t,\mathbf{x}}(\omega)\to X^{t,\mathbf{x}}(\omega)$  in $\mathbb W$.
This shows that $u_n(s,X^{(n),t,\mathbf{x}}(\omega)) {\rightarrow}  u(s,X^{t,\mathbf{x}}(\omega))$ for every $s\in[t,T]$. Arguing analogously, we have for each $s\in [t,T]$
$$
F_n(s,X^{(n),t,\mathbf{x}}(\omega),u_n(s,X^{(n),t,\mathbf{x}}(\omega)))
\stackrel{n\rightarrow\infty}{\longrightarrow}
F(s,X^{t,\mathbf{x}}(\omega),u(s,X^{t,\mathbf{x}}(\omega))).
$$
Now we can conclude by applying Lemma \ref{GD}. Indeed, assuming without loss of generality $t<s$, the hypotheses are verified for $(\Sigma, \mu)=(\Omega\times[t,s],\mathbb{P}\otimes Leb)$ and
\begin{equation*}
\begin{split}
f_n(\omega,r)=& \frac{1}{s-t}u_n(s,X^{(n),t,\mathbf{x}}(\omega))+F_n(r,X^{(n),t,\mathbf{x}}(\omega),u_n(r,X^{(n),t,\mathbf{x}}(\omega))),\\
f(\omega,r)=& \frac{1}{s-t}u(s,X^{t,\mathbf{x}}(\omega))+F(r,X^{t,\mathbf{x}}(\omega),u(r,X^{t,\mathbf{x}}(\omega))),\\
g_n(\omega,r)=&g_n(\omega)=M'(1+| X^{(n),t,\mathbf{x}}(\omega)|^p_\infty), \\
g(\omega,r)=&g(\omega)=M'(1+| X^{t,\mathbf{x}}(\omega)|^p_\infty),
\end{split}
\end{equation*}
for a sufficiently large $M'>0$, as $\{u_n, n\in \mathbb{N}\}$ is a bounded subset of $C_p(\Lambda)$, and since $\int_\Sigma g_n d\mu\rightarrow\int_\Sigma  gd \mu$ by  \eqref{StabMild}.
\end{proof}

\section{The Markovian case}\label{sec:markov}

In the Markovian case, i.e., when all data depend only on the present, infinite-dimensional PDEs of  type  \eqref{eq:PPDE}--\eqref{terminal}  have been studied  from the point of view of viscosity solutions starting from  \cite{Lio1,Lio2,Lio3}. In this section we compare the results of the literature with the statement of  our main Theorem  \ref{lemma:pp1}  in this Markovian  framework.

Hence,  let us assume that the data
 $b$, $\sigma$, $F$, $\xi$ satisfy all the assumptions used in the previous sections
and, moreover, that they depend only on $x=\mathbf{x}_t$, instead of the whole path $\mathbf{x}$.
The SDE \eqref{SHDE} is no more
path-dependent and takes the following form:
\begin{equation}
\label{SHDE1}
\begin{cases}
dX_s= AX_s ds + b(s,X_s)ds + \sigma(s,X_s)dW_s, \qquad &s\in[t,T], \\
X_t = x\in H.
\end{cases}
\end{equation}
Accordingly, \eqref{PPDE-intro}
 becomes a non path-dependent (\footnote{In this section we drop, for simplicity, the final condition $\xi$. But it is important to notice that the PDE must be considered path-dependent even if only $\xi$ depends on the past, while $b,\sigma, F$ do not.}) second order
parabolic PDE in the Hilbert space $H$, which is formally written   for $(t,x)\in[0,T)\times \mathcal{D}(A)$ as (\footnote{Notice that the time derivative $\partial_t u(t,{x})$ here appearing can denote equivalently the Dupire time-derivative of Definition \ref{D:SpaceDer} or  the standard partial right time-derivative,
as in this Markovian case they coincide each other on $[0,T)$.
})
\begin{multline}
\label{eq:PPDE1}
-\partial_t u(t,x) - \frac{1}{2}\text{Tr}\left[\sigma(t,x)\sigma^*(t,x)D^2 u(t,x)\right]
- \langle Ax,Du(t,x)\rangle-
\\
 - \langle b(t,x),Du(t,x)\rangle
-{F}(t,x,u(t,x)) =0.
\end{multline}
In such Markovian framework, the results of Section~\ref{sec:ppde} 
 still hold.
Indeed, defining viscosity solutions of \eqref{eq:PPDE1} as in Definition \ref{def:visc_sol_PPDE}, with $x$ in place of $\mathbf{x}$, we know from Theorem \ref{teo:main} that there exists a unique viscosity solution $\hat{u}$ to \eqref{eq:PPDE1}
and that it admits the probabilistic representation formula \eqref{hatu} of Remark \ref{rem:back}, with $x$ in place of $\mathbf{x}$.

On the other hand,  equations like \eqref{eq:PPDE1} have been studied in the literature, by means of
what we call here the ``standard'' viscosity solution approach. This is performed, in the spirit of the finite-dimensional case,  by computing the
terms of \eqref{eq:PPDE1} on smooth test functions suitably defined
and using the method of doubling variables to prove the comparison.
Such ``standard'' approach in infinite dimension has been first introduced in \cite{Lio1,Lio2,Lio3}
and then developed in various papers
(see e.g.\ \cite{GRS,GSZakai,GozSw,Kel,S94} and \cite[Ch.\ 3]{FGS14} for a survey).

To compare our results with those obtained in the literature quoted  above,
we first introduce a concept of classical solution of \eqref{eq:PPDE1}.

First of all, observe that \eqref{eq:PPDE1} is well defined only in $[0,T)\times \mathcal{D}(A)$.
In order to {give a meaning to \eqref{eq:PPDE1} in $[0,T)\times H$}
we consider the operator $A^*$, adjoint of $A$, defined on $\mathcal{D}(A^*)\subset H$, and express the term containing $Ax$ in \eqref{eq:PPDE1} by writing
$$\langle Ax,Du(t,x)\rangle=\langle x,A^*Du(t,x)\rangle,$$ which is well defined in
$[0,T)\times H$ provided that $Du\in \mathcal{D}(A^*)$.
Hence, to define classical solutions of such equation, we define the operator $\Lc_1$ as follows: the domain is\ (\footnote{ $UC^{1,2}([0,T]\times H)$ denotes the space of maps $\psi\colon[0,T]\times H\rightarrow\R$ which are uniformly continuous together with their first time Fr\'echet derivative and their first and second spatial Fr\'echet derivatives})
\begin{equation*}
  \begin{split}
      \mathcal{D}(\Lc_1)=&\Big\{
        \psi \in UC^{1,2}([0,T]\times H)\colon \mbox{the maps } (t,x) \mapsto \langle x,A^*D\psi(t,x)\rangle,\\
&\phantom{\ \ \ }     (t,x) \mapsto   A^*D\psi(t,x),\ 
(t,x) \mapsto \frac{1}{2}\text{Tr}\left[\sigma(t,x)\sigma^*(t,x)D^2\psi(t,x)\right],\\
&    \mbox{\hskip6cm belong to $UC([0,T]\times H)$}
      \Big\},
\end{split}
\end{equation*}
and, {for $\psi \in \mathcal{D}(\Lc_1)$,}
\begin{multline*}
  \Lc_1 \psi (t,x) =\partial_t \psi(t,x)
+\frac{1}{2}\text{Tr}\left[\sigma(t,x)\sigma^*(t,x)D^2 \psi(t,x)\right]\\
+ \langle x,A^*D\psi(t,x)\rangle+\langle b(t,x),D\psi(t,x)\rangle.
\end{multline*}
Then we say that $u$ is a classical solution
of \eqref{eq:PPDE1} if $u \in D(\Lc_1)$ and satisfies
\begin{equation}
\label{eq:PPDE1bis}
-\Lc_1 u(t,x) -F(t,x,u(t,x)) =0,\qquad
\forall\, (t,x)\in[0,T)\times H.
\end{equation}

The standard definition of viscosity subsolution (supersolution) for \eqref{eq:PPDE1}
says roughly that, at any given $(t,x) \in [0,T)\times H$, the equation must be satisfied
with $\le$ ($\ge$), when we substitute to the derivatives of $u(t,x)$ the derivatives of $\varphi(t,x)$, where $\varphi$ is a suitably chosen test function.

Clearly, in this context test functions should be chosen in such a way that  all terms of \eqref{eq:PPDE1} have classical sense. Hence, their regularity must be substantially the one
required for classical solutions, i.e., roughly, $\varphi \in \mathcal{D}(\Lc_1)$.
This regularity is very demanding, much more than the one required in the finite-dimensional case:  requiring that $D\varphi\in \mathcal{D}(A^*)$ and the finite trace condition in the second order term strongly restricts the set of test functions. In this way the proof of the existence has not a greater structural difficulty with respect to the finite-dimensional case,
 but the uniqueness, which is based on a delicate
construction of suitable test functions, becomes much harder.

To be more explicit, let us first give  a definition of
 ``naive''  viscosity solution to \eqref{eq:PPDE1}.

\begin{Definition}\label{def:viscst1naive}
${}$
  \begin{enumerate}[(i)]
\setlength\itemsep{0.05em}
\item \vskip-5pt
An  upper semicontinuous function $u\colon[0,T]\times H\rightarrow \R$ is called  a \textbf{naive viscosity subsolution} of\eqref{eq:PPDE1} if
\begin{equation*}
  - \Lc_1 \varphi(t,{x}) -{F}(t,{x},u(t,{x})) \le 0,
\end{equation*}
for any $(t,x) \in [0,T) \x H$ and any function $\varphi\in D(\Lc_1)$ such that $\varphi-u$ has a local minimum at $(t,x)$.
\item
 A lower semicontinuous function $u\colon[0,T]\times H\rightarrow \R$ is called a \textbf{naive viscosity supersolution} of  \eqref{eq:PPDE1} if
\begin{equation*}
  - \Lc_1 \varphi(t,{x}) -{F}(t,{x},u(t,{x})) \ge 0,
\end{equation*}
for any $(t,x) \in [0,T) \x H$ and any function $\varphi\in D(\Lc_1)$ such that $\varphi-u$ has a local maximum at $(t,x)$.
\item  A continuous function $u\colon[0,T]\times H\rightarrow \R$ is called a \textbf{naive viscosity solution} of  \eqref{eq:PPDE1}
if it is both a viscosity subsolution and a viscosity supersolution.
\end{enumerate}
\end{Definition}

If we adopt this definition, it is clear that the set of test functions
used is strictly included in the one used in our Definition \ref{def:visc_sol_PPDE}.
Hence,  if a function is a viscosity solution according to Definition \ref{def:visc_sol_PPDE},
it must also  be a viscosity solution according to Definition \ref{def:viscst1naive}, while
the opposite is, a priori, not true.
Hence, if one { was} able to prove a uniqueness result for viscosity solution according to Definition \ref{def:viscst1naive}, such a  result would be more powerful than our
existence and uniqueness Theorem \ref{teo:main}.
However, the technique used to prove uniqueness in finite dimension does not work with { Definition \ref{def:viscst1naive}} and  there are no general uniqueness results with this definition.

In the literature concerning  ``standard'' viscosity solutions in infinite dimension this problem has been overcome by introducing suitable
restrictions on the family of equations and adding an ad hoc radial term $g$ to each test function $\varphi$.
We explain more in detail what is needed to apply such techniques to our equation  \eqref{eq:PPDE1};
then we give a result obtained with such technique and compare it with our previous results.

To start, it is useful to rewrite equation \eqref{eq:PPDE1} as follows:

\begin{equation} \label{eq:PPDE2}
-\partial_t u(t,{x}) - \langle x,A^*Du(t,x)\rangle - Lu(t,{x})
 -{F}(t,{x},u(t,{x}))  =  0, \\
\end{equation}
for $(t,x)\in[0,T)\times H$,
with, for any $u\in C^{1,2}([0,T]\times H)$ in the sense of Fr\'echet,
\[
Lu(t,x)  =  \langle b(t,x),Du(t,x)\rangle
+ \frac{1}{2}\text{Tr}\big[\sigma(t,x)\sigma^*(t,x)D^2u(t,x)\big].
\]
To account for the ``difficult'' term $\langle x,A^*Du(t,x)\rangle$ we impose the following assumption.

\begin{Assumption}
\label{A:A}
$A$ is a maximal dissipative operator in $H$.
\end{Assumption}

Notice that Assumption  \ref{A:A} implies that $A$
generates a $C_0$-semigroup of contractions on $H$.
Moreover, from
Assumption~\ref{A:A} and 
 \cite{Ren}, it follows   that there exists a symmetric, strictly positive, and bounded operator $B$ on $H$ such that $A^*B$ is a bounded operator  on $H$
and
\begin{equation}\label{WeakBCondition}
-A^*B+c_0B  \geq  0,
\end{equation}
for some $c_0>0$.

\begin{Definition}
  Let $\{x_n\}_{n\in \mathbb{N}}\subset H$ be a sequence and let $x\in H$. We say that the sequence $\{x_n\}_{n\in\mathbb{N}}$ is \emph{$B$-convergent} to $x$, if  $\{x_n\}_{n\in \mathbb{N}}$ converges weakly to $x$ and 
$\{Bx_n\}_{n\in \mathbb{N}}$ converges strongly to  $Bx$ in $H$.

A function $u\colon[0,T]\times H\to \mathbb{R}$ is said to be \textbf{$B$-upper semicontinuous} (resp.\ \textbf{$B$-lower semicontinuous}) if for any $\{t_n\}_{n\in\mathbb{N}}\subset [0,T]$ convergent to $t\in[0,T]$, and  for any $\{x_n\}_{n\in\mathbb{N}}\subset H$ $B$-convergent to $x\in H$, we have
$$
\limsup_{n\to \infty}u(t_n,x_n)\leq u(t,x)\qquad
\mbox{(resp.\ }
\liminf_{n\to \infty}u(t_n,x_n)\geq u(t,x)
 \mbox{).}
$$
Finally, $u$ is \textbf{$B$-continuous} if it is $B$-upper and $B$-lower semicontinuous.
\end{Definition}

We consider two classes of smooth (test) functions:
\begin{enumerate}[(C1)]
\setlength\itemsep{0.05em}
\item
\label{2016-10-17:03}
 \vskip-5pt
 (the ``smooth'' part) $\varphi\in C^{1,2}([0,T]\times H)$, $D\varphi$ is $\mathcal{D}(A^*)$-valued, $\partial_t\varphi$, $A^*D\varphi$, and $D^2\varphi$ are uniformly continuous on $[0,T]\times H$, and $\varphi$ is $B$-lower semiconinuous.
\item 
\label{2016-10-17:04}
(the ``radial'' part) $g\in C^{1,2}([0,T]\times\R)$ and, for every $t\in[0,T]$, the function $g(t,\cdot)$ is even on $\R$ and nondecreasing on $[0,\infty)$.
\end{enumerate}

\begin{Definition}\label{def:viscst1}
  \begin{enumerate}[(i)]
\setlength\itemsep{0.05em}
\item \vskip-5pt
 A  $B$-upper semicontinuous function $u\colon[0,T]\times H\rightarrow \R$, which is bounded on bounded sets,  is called  a \textbf{viscosity subsolution} of \eqref{eq:PPDE2} if
\begin{equation*}
  -\partial_t(\varphi+g)(t,{x}) - \langle x,A^*D\varphi(t,x)\rangle - L(\varphi+g)(t,{x})
 -{F}(t,{x},u(t,{x})) \le 0,
\end{equation*}
for any $(t,x) \in [0,T) \x H$ and any pair of functions $(\varphi, g)$ belonging, respectively, to the classes 
\emph{(C\ref{2016-10-17:03})--(C\ref{2016-10-17:04})}
 above and   such that $\varphi+g-u$ has a local minimum at $(t,x)$.
\item  A  $B$-lower semicontinuous function $u\colon[0,T]\times H\rightarrow \R$ , which is bounded on bounded sets,  is called a \textbf{viscosity supersolution} of  \eqref{eq:PPDE2} if
\begin{equation*}
  -\partial_t (\varphi-g)(t,{x}) - \langle x,A^*D\varphi(t,x)\rangle -L(\varphi-g)(t,{x}) - {F}(t,{x},u(t,{x})) \ge 0,
\end{equation*}
for any $(t,x) \in [0,T) \x H$ and any pair of functions $(\varphi, g)$ belonging, respectively, to the classes 
\emph{(C\ref{2016-10-17:03})--(C\ref{2016-10-17:04})}
above and   such that $\varphi-g-u$ has a local maximum at $(t,x)$.
\item  A function $u\colon[0,T]\times H\rightarrow \R$ is called a \textbf{viscosity solution} of  \eqref{eq:PPDE2} if it is both a viscosity subsolution and a viscosity supersolution.
\end{enumerate}
\end{Definition}

\begin{Remark}\label{rem:radial}
The radial function $g$ belonging to the class 
\emph{(C\ref{2016-10-17:04})}
 introduced in Definition \ref{def:viscst1}
plays the role of cut-off function and is needed to produce, together with the $B$-continuity property, local/global minima and maxima  of $\varphi+g-u $ and $\varphi-g-u$, respectively. However, the introduction of the radial function  forces to impose Assumption \ref{A:A} to get rid of the term $\langle Ax,Dg(t,x)\rangle$ which would come out from the gradient of $g$.

Radial test functions could also be included in our Definition \ref{def:visc_sol_PPDE} when $A$ is a maximal monotone operator without compromising the existence result (but note that it would be redundant  including them in our definition, as  they are  not needed  to prove uniqueness in Theorem \ref{teo:main}). In this case, our Definition \ref{def:visc_sol_PPDE} would be stronger than Definition \ref{def:viscst1} in the sense that a viscosity subsolution (supersolution) in the sense of Definition \ref{def:visc_sol_PPDE} must be necessarily also a viscosity subsolution (supersolution) according to Definition \ref{def:viscst1}. Indeed, a test function in the sense of Definition \ref{def:viscst1} would be also a test function in the sense of Definition \ref{def:visc_sol_PPDE}.
\end{Remark}

We can now state a comparison theorem and an existence result for equation \eqref{eq:PPDE2}. Firstly, we need to introduce some notations. Let $H_{-1}$ be the completion of $H$ with respect to the norm
\[
|x|_{-1}^2  \coloneqq   \langle Bx,x\rangle.
\]
The norm $|\cdot|_{-1}$ in $H$ is weaker than $|\cdot|$.
The space $H_{-1}$ is a separable Hilbert space with the inner product
\[
\langle x,x'\rangle_{-1}  \coloneqq   \big\langle B^{1/2}x,B^{1/2}x'\big\rangle.
\]
Let now $\{e_1,e_2,\ldots\}$ be an orthonormal basis in $H_{-1}$ made of elements of $H$. For $N\geq 1$, 
 we denote $H_N=\text{span}\{e_1,\ldots,e_N\}$. Let $P_N\colon H_{-1}\rightarrow H_{-1}$ be the orthogonal projection onto $H_N$ and denote $P_N^\perp=I-P_N$.

\begin{Theorem}\label{teo:SW}
Let Assumption \ref{A:A} hold and
assume the following.
\begin{enumerate}[(i)]
\setlength\itemsep{0.05em}
\item\label{2016-10-13:04} \vskip-5pt
 The map $\mathbb{R} \mapsto \mathbb{R},\ y\mapsto F(t,x,y)$, is nonincreasing, for all $(t,x)\in[0,T]\times H$.
\item \label{2016-10-13:02} $F$ is uniformly continuous on bounded sets.
\item\label{2016-10-13:06}
 For all $R>0$, there exists a modulus of continuity $w_R$ such that
\begin{equation*}
    |F(t,x,y)-F(t',x',y)|\leq w_R(|t-t'|+|x-x'|_{-1}),
\end{equation*}
for all $t,t'\in (0,T)$, $y\in\mathbb R$, 
and for all $x,x'\in H$, with $ |x|\leq R$, $ |x'|\leq R$.
\item\label{2016-10-13:01} $b$ is {uniformly continuous on bounded sets} and
$$
|b(t,x)-b(t,x')|\leq
 L|x-x'|_{-1}\qquad \forall t\in[0,T],\ x,x'\in H.
$$
\item \label{2016-10-13:03}
{$\sigma(t,x)\in L_2(K;H)$ for every $(t,x)\in[0,T]\times H$, $\sigma\colon[0,T]\times H\rightarrow L_2(K;H)$ is uniformly continuous on bounded sets} and
$$
|\sigma(t,x)-\sigma(t,x')|_{L_2(K;H)}  \leq  L|x-x'|_{-1},\qquad
\forall t\in[0,T],\ x,x'\in H.
$$
\item \label{2016-10-15:00}
The following limit holds
\[
\lim_{N\rightarrow\infty}\textup{Tr}\big[\sigma(t,x)\sigma^*(t,x)BP_N^\perp\big]  =  0, \qquad \forall\,(t,x)\in[0,T]\times H.
\]
\end{enumerate}
Then, the following statements hold true.
\begin{enumerate}[(a)]
\setlength\itemsep{0.05em}
\item\label{2016-10-13:00} \vskip-5pt
 Let $u,v$ be  continuous
 viscosity subsolution  and  supersolution, respectively, 
\emph{in the sense of Definition \ref{def:viscst1}},  to \eqref{eq:PPDE2}.
Assume that, for every $R>0$ there exists a modulus of continuity $\tilde w_R$ such that,
for all $t,t'\in (0,T)$ and all  $x,y\in H$, with $  |x|\leq R$, $ |y|\leq R$,
\begin{equation}\label{2016-10-14:00}
\max\{  |u(t,x)-u(t',y)|,  |v(t,x)-v(t',y)|\}\leq \tilde w_R(|t-t'|+|x-y|_{-1})
\end{equation}
and that, for some $p\geq 0$, 
\begin{equation}\label{2016-10-14:01}
  \max\{u(t,x),-v(t,x)\}\leq p+|x|^p\qquad \forall (t,x)\in(0,T)\times H.
\end{equation}

 If $u(T,\cdot)\leq v(T,\cdot)$, then $u\leq v$ on $[0,T]\times H$.
\item\label{2016-10-14:02}
Let
$\xi\colon H\rightarrow \mathbb{R}$.
Assume in addition that
$F(t,x,y)=F(t,x)$ does not depend on $y$ and that, for some $p\geq 0$,
\begin{equation*}
  \max\{|F(t,x)|,  |\xi(x)|\}\leq p+|x|^p\qquad \forall t\in [0,t],\ \forall x\in H.
\end{equation*}
Finally assume
 that
for all $R>0$ 
 there exists a modulus of continuity $\hat w_R$ such that
 \begin{multline*}
\max\{|b(t,x)-b(t',x)|,
|\sigma(t,x)-\sigma(t',x)|_{L_2(K;H)}
,   |\xi(x)-\xi(x')|\}\leq\\
\leq \hat w_R(|t-t'|+|x-x'|_{-1})
 \end{multline*}
for all $t,t'\in[0,T]$ and all
$ x,x'\in H$, with $ |x|\leq R$, $|x'|\leq R$.

Then, there exists a unique viscosity solution $\hat u$ to \eqref{eq:PPDE2},
\emph{in the sense of Definition \ref{def:viscst1}},
among functions in the set
\begin{equation*}
  \begin{split}
    S\coloneqq
\Big\{ & u\colon [0,T]\times H\rightarrow \mathbb{R}\mbox{\ s.t.\ }\sup_{(t,x)\in [0,T]\times H}\frac{|u(t,x)|}{1+|x|^k}\mbox{\ for some\ }k\geq 0 \\
&
 \lim_{t\rightarrow T}|u(t,x)-\xi(x)|=0\ \mbox{uniformly on bounded subsets of $H$}
\Big\},
\end{split}
\end{equation*}
The solution $\hat u$  admits the probabilistic representation
 (\,\footnote{When $H$ is finite-dimensional, the probabilistic representation formula \eqref{hatu} provides the unique ``standard'' viscosity solution of \eqref{eq:PPDE2} also when $F$ depends on $y$, see \cite{pardouxpeng}.})
\[
\hat{u}(t,x)  = \E\left[\xi(X^{t,x}_T)+\int_t^TF(s,X^{t,x}_s)ds\right], \qquad (t,x)\in[0,T]\times H.
\]
\end{enumerate}
\end{Theorem}
\begin{proof}
The proof of Theorem~\ref{teo:SW} is an application of Theorems 3.50 and 3.66 in \cite{FGS14}, see the Appendix.
\end{proof}

We note the following facts.
\begin{enumerate}
\setlength\itemsep{0.05em}
\item
{The presence of the norm $|\cdot|_{-1}$ in assumptions 
\emph{(\ref{2016-10-13:06})},
\emph{(\ref{2016-10-13:01})},
\emph{(\ref{2016-10-13:03})}
of Theorem \ref{teo:SW}} is needed to exploit the $B$-continuity.
Indeed the requirement of $B$-continuity on the sub(super)solutions is needed to generate maxima and minima in the proof of comparison. In this way one is obliged to assume these stronger conditions on the coefficients to ensure the existence of solutions (see \cite{S94}).

\item Assumption \emph{(\ref{2016-10-15:00})}
  in Theorem \ref{teo:SW} is needed since, to prove uniqueness, one has to use the so-called Ishii's Lemma which allows to perform the procedure of doubling variables. {Ishii's Lemma has only a finite dimensional formulation}, so the proof is performed through finite-dimensional approximations: the condition \emph{(\ref{2016-10-15:00})} ensures the convergence of such approximations.
\end{enumerate}

{The assumptions we used in Section \ref{sec:ppde}, when  reduced to the Markovian case, are weaker than those of Theorem \ref{teo:SW}. We can conclude that, under the assumptions of Theorem \ref{teo:SW}, the two definitions of viscosity solution (Definition \ref{def:visc_sol_PPDE} and Definition \ref{def:viscst1}) select the same solution in the present Markovian case. However, as noticed, adopting our Definition \ref{def:visc_sol_PPDE} of viscosity solution requires weaker assumptions to prove that the function $\hat u$ in \eqref{hatu} is the unique solution in such sense.}
In particular:
\begin{enumerate}
\setlength\itemsep{0.05em}
\item \vskip-5pt
 The map $\sigma$ does not need to satisfy 
assumption
\emph{(\ref{2016-10-15:00})}
 of Theorem \ref{teo:SW} and  $\sigma(t,x)\in L_2(K;H)$ for all $(t,x)\in[0,T]\times H$ --- which, in the case of constant $\sigma$, would imply that $\sigma \sigma^*$ is a nuclear operator, hence reducing the applicability of the theory --- as the proof of uniqueness does not require the use of Ishii's lemma on the corresponding finite-dimensional approximations.
\item The coefficients $b$, $\sigma$, $F$, and $\xi$ do not need to be $B$-continuous with respect to $x$, as no local compactness is needed to produce local max/min in our sense.
\item The operator $A$ does not need to be maximal monotone, as radial test functions are not needed to produce local max/min in our sense.
\end{enumerate}
Roughly speaking, we can say that our definition allows to cover more general cases since the relation with the PDE is different in the following sense: the PDE is tested in analytical sense, but over test functions which satisfy the min/max condition only in a probabilistic sense and only when composed with the process $X^{t,x}$; indeed minimum (maximum) of  $\varphi-u$  is not pointwise in a neighborhood of $(t,x)$, but only in mean and when composed with the process $X^{t,x}$.

\section{On the extension to semilinear and fully nonlinear equations}
\label{S:Extension}
The notion of viscosity solution we introduced is designed for our PPDE \eqref{eq:PPDE} and needs to be suitably modified when  considering nonlinearities in the derivatives. In \cite{EkrenKellerTouziZhang}, this entails a substantial change in  the definition of viscosity solution by considering  optimal stopping problems under nonlinear expectation, i.e., under a family of probability measures. In our formalism, which separates the (fixed) probability space from the state space (see Remark \ref{rem:setting}), this corresponds to take a mixed control/stopping problem.

In the present section we investigate how and to which extent, up to now, some of the results can be extended to the case of semilinear and fully nonlinear PPDEs of Hamilton-Jacobi-Bellman type: \begin{align}
\sup_{a\in U}\big[\mathcal{L}^a u(t,\mathbf{x})
+ \ell(t,\mathbf{x},a)\big]  &=  0, \qquad \forall\, (t,\mathbf x)\in\Lambda, \ t\in[0,T), \label{eq:PPDE_fullynonlinear}
\end{align}
where $U$ is a Polish space, $\ell\colon\Lambda\times U\rightarrow\R$ is a measurable function and $\mathcal{L}^a u(t,\mathbf{x})$ will be defined in the spirit of \eqref{def:L}.

More precisely we provide:
\begin{enumerate}[-]
\setlength\itemsep{0.05em}
\item \vskip-5pt
 an existence result (Subsection \ref{SubS:Existence});
  \item a partial comparison result assuming existence for an associated stochastic optimal mixed stopping/control problem (Subsection \ref{SubS:PartialComp});
  \item the main steps of a possible path to prove a comparison result for a semilinear equation satisfying a suitable structure condition (see \eqref{structure}),  generalizing the argument used in finite dimension in [36] (Subsection \ref{SubS:Punctual}).
\end{enumerate}

PPDEs of Hamilton-Jacobi-Bellman type are naturally associated to optimal control problems. In our context the state process solves a controlled path-dependent stochastic differential equation.
We now introduce such a stochastic optimal control problem.
We define the set of admissible controls $\mathcal U$ as follows
$$\Uc\coloneqq  \{\mathbf{a}\colon[0,T]\times\Omega\rightarrow U \ \mbox{predictable}\}.$$
Let  $t\in[0,T]$,  $Z\in\mathcal{H}_\mathcal{P}^0(H)$, and $\mathbf{a}\in\Uc$. We consider the following \emph{controlled path-dependent} SDE:
\begin{equation}
\label{SHDE_control}
\begin{cases}
dX_s= AX_s ds + \bar b(s,X,\mathbf{a}_s)ds + \bar\sigma(s,X,\mathbf{a}_s)dW_s, \qquad &s\in[t,T], \\
X_{\cdot\wedge t} = Z_{\cdot\wedge t},
\end{cases}
\end{equation}
where  $A$ satisfies Assumption \ref{A:SHDE}\emph{(\ref{2016-07-13B:00})}, whereas $\bar b$ and $\bar\sigma$ satisfy the following conditions
(\footnote{Hereafter, if a function $f=f((t,\mathbf{x}),y_1,\ldots,y_j)$ depends on $(t,\mathbf{x})\in \Lambda$ and on some other variables $y_1,\ldots,y_j$, by an abuse of notation we
denote
$f((t,\mathbf{x}),y_1,\ldots,y_j)$
by
$f(t,\mathbf{x},y_1,\ldots,y_j)$.}).

\begin{Assumption}
\label{A:SHDE_control}
\quad
\begin{enumerate}[(i)]
\setlength\itemsep{0.05em}
\item\label{2016-10-17:01} \vskip-5pt
 $\bar b\colon\Lambda\times U \rightarrow H$ is measurable with respect to the Borel $\sigma$-algebra
and, for some constant $M>0$,
\begin{align}
|\bar b(t,\mathbf x,a)| & \leq  M(1 + |\mathbf x|_\infty),\\[5pt]
  |\bar b(t,\mathbf{x},a) - \bar b(t,\mathbf{x}',a)|  &\leq  M|\mathbf{x}-\mathbf{x}'|_{\infty},\label{2016-08-05:08}
\end{align}
for all $\mathbf x,\mathbf x'\in\mathbb W$, $t\in[0,T]$, $a\in U$.
\item $\bar\sigma\colon\Lambda\times U\rightarrow L(K;H)$  is such that $\bar\sigma(\cdot,\cdot,\cdot)v\colon\Lambda\times U \rightarrow H$ is measurable for all $v\in K$ and  $e^{sA}\bar\sigma(t,\mathbf x,a)\in L_2(K;H)$ for all $s>0$,  $(t,\mathbf{x})\in\Lambda$, $a\in U$. Moreover, there exist $\hat M>0$ and $\gamma\in[0,1/2)$ such that, for all $\mathbf{x},\mathbf{x}'\in \mathbb W$, $t\in[0,T]$, $s\in(0,T]$, $a\in U$,
\begin{align}
|e^{sA}\bar\sigma(t,\mathbf x,a)|_{L_2(K;H)}  &\leq  \hat M s^{-\gamma}(1 + |\mathbf{x}|_{\infty}), \\[5pt]
|e^{sA}\bar\sigma(t,\mathbf x,a) - e^{sA}\bar\sigma(t,\mathbf x',a)|_{L_2(K;H)}  &\leq  \hat M s^{-\gamma}|\mathbf{x}- \mathbf{x}'|_{\infty}.\label{2016-08-05:07}
\end{align}
\end{enumerate}
\end{Assumption}

\smallskip
Notice that,
by Remark~\ref{R:non-anticipative}\emph{(\ref{2016-06-21:01})}, we have,
for all $(t,\mathbf{x})\in \Lambda$, $a\in U$,
$$
  \bar b(t,\mathbf{x},a)=
 \bar b(t,\mathbf{x}_{\cdot\wedge t},a)\qquad
  \bar \sigma(t,\mathbf{x},a)=
 \bar \sigma(t,\mathbf{x}_{\cdot\wedge t},a).
$$

\smallskip

\begin{Definition}\label{def:2014-11-17:aa_control}
Let  $t\in[0,T]$,  $Z\in\mathcal{H}_\mathcal{P}^0(H)$, and $\mathbf{a}\in\Uc$.
We call \textbf{mild solution} of \eqref{SHDE_control} a process $X\in\mathcal{H}^0_\mathcal{P}(H)$ such that $X_{\cdot\wedge t}=Z_{\cdot\wedge t}$ and
\begin{equation}
  \label{eq:2016-08-05:06}
  X_s  =  e^{(s-t)A}Z_t + \int_t^s e^{(s-r)A}\bar b(r,X,\mathbf{a}_r)dr + \int_t^s e^{(s-r)A} \bar\sigma(r,X,\mathbf{a}_r) dW_r,
\end{equation}
for all $s\in [t,T]$.
\end{Definition}

\medskip

The proof of the following theorem is postponed in 
the Appendix.

\begin{Theorem}
\label{Exist_control}
Let Assumptions \ref{A:SHDE}(\ref{2016-07-13B:00})
 and \ref{A:SHDE_control} hold. For every $p>p^*\coloneqq \frac{2}{1-2\gamma}$, $t\in[0,T]$,  $Z\in\mathcal{H}_\mathcal{P}^p(H)$,
 and $\mathbf{a}\in\Uc$, there exists a unique mild solution $X^{t,Z,\mathbf{a}}$ to \eqref{SHDE_control}.
Moreover,    $X^{t,Z,\mathbf{a}}\in \mathcal{H}_\mathcal{P}^p(H)$ and
\begin{equation}
  \label{eq:2016-08-06:02}
  |X^{t,Z,\mathbf{a}}|_{\mathcal{H}^p_\mathcal{P}(H)}\leq \bar K_0(1 + |Z|_{\mathcal{H}^p_\mathcal{P}(H)}), \qquad
\forall \, (t,Z,\mathbf{a})\in[0,T]\times \mathcal{H}^p_\mathcal{P}(H)\times\Uc.
\end{equation}
Finally, for every $\mathbf{a}\in\Uc$, the map
\begin{equation}\label{eq:2014-10-29:aa_control}
[0,T]\times\mathcal{H}_\mathcal{P}^p(H)\to\mathcal{H}_\mathcal{P}^p(H), \ (t,Z)\mapsto X^{t,Z,\mathbf{a}}
\end{equation}
is Lipschitz continuous with respect to $Z$,  uniformly in $t\in[0,T]$ and in $\mathbf{a}\in \mathcal{U}$, and the family
\begin{equation}
  \label{eq:2016-08-05:04}
  \{X^{\cdot,\cdot,\mathbf{a}}\colon [0,T]\times \mathcal{H}^p_\mathcal{P}(H)\rightarrow \mathcal{H}^p_\mathcal{P}(H)\}_{\mathbf{a}\in \mathcal{U}}
\end{equation}
is equicontinuous.
\end{Theorem}

As for \eqref{flow}, we notice that
 uniqueness of mild solutions yields the flow property:
\begin{equation}\label{flow_control}
X^{t,\mathbf{x},\mathbf{a}}= X^{s,X^{t,\mathbf x,\mathbf{a}},\mathbf{a}}\mbox{ in } \mathcal{H}^p_\mathcal{P}(H), \
\forall\, (t,\mathbf{x})\in \Lambda,\ \forall\, \mathbf{a}\in \Uc,\
 \forall\,s\in[t,T].
\end{equation}

Given $(t,\mathbf x)\in\Lambda$, we consider the stochastic optimal control problem consisting in maximizing, over all admissible control processes $\mathbf{a}\in\Uc$, the following gain functional:
\begin{equation}
  \label{eq:2016-08-05:00}
  J(t,\mathbf x,\mathbf{a})  \coloneqq   \E\bigg[\int_t^T \ell(s,X^{t,\mathbf x,\mathbf{a}},\mathbf{a}_s)ds + \xi(X^{t,\mathbf x,\mathbf{a}})\bigg],
\end{equation}
where $\xi\colon\mathbb{W}\to\R$.
We define the value function $v\colon\Lambda\rightarrow\R$ of the stochastic optimal control problem:
\begin{equation}\label{value_function}
v(t,\mathbf x)  \coloneqq   \sup_{\mathbf{a}\in\Uc} J(t,\mathbf x,\mathbf{a}), \qquad \forall\,(t,\mathbf x)\in\Lambda.
\end{equation}

\subsection{Existence of viscosity solutions}
\label{SubS:Existence}

In order to prove existence of viscosity solutions to PPDE \eqref{eq:PPDE_fullynonlinear}, we introduce the following assumptions.
\begin{Assumption}\label{A:ell}
${}$
  \begin{enumerate}[(i)]
\setlength\itemsep{0.05em}
\item 
\label{2016-07-08:00}
\vskip-5pt
$\ell\colon\Lambda\times U\rightarrow\R$ is measurable;
\item \label{2016-07-08:01}  the family $\{\Lambda\rightarrow \mathbb{R},\ (t,\mathbf{x}) \mapsto \ell(t,\mathbf{x},a)\}_{a\in U}$ is equicontinuous;
\item \label{2016-07-08:02} there exists $N>0$, $p\geq 0$, such that
  \begin{equation}\label{ellep}
    |\ell(t,\mathbf{x},a)|\leq N(1+|  \mathbf{x}|_\infty^p),
\qquad  \forall\,(t,\mathbf x)\in \Lambda,\ \forall\,a\in U.
\end{equation}
\end{enumerate}
\end{Assumption}

\begin{Proposition}\label{prop:DPP}
Let Assumptions \ref{A:SHDE}(\ref{2016-07-13B:00}), 
\ref{A:SHDE_control}, \ref{A:ell} hold, 
let $\xi\in C(\mathbb{W})$,
 and let $p\geq 0$ be such that 
both 
\eqref{A:BSDE_xi} and
\eqref{ellep} hold true.
Then $v\in C_p(\Lambda)$ and satisfies the Dynamic Programming Principle:
\[
v(t,\mathbf x)  =  \sup_{\mathbf{a}\in\Uc} \E\bigg[\int_t^\tau \ell(s,X^{t,\mathbf x,\mathbf{a}},\mathbf{a}_s)ds + v(\tau,X^{t,\mathbf x,\mathbf{a}})\bigg],
\]
for all
$(t,\mathbf x)\in\Lambda$,
$\tau \in \Tc$ with $\tau\geq t$.
\end{Proposition}
\begin{proof}

Once one proves that the value function is continuous, the proof of the
Dynamic Programming Principle can be done following the same path of \cite[Sec.2.3]{FGS14} (\footnote{In \cite{FGS14} this is proved for deterministic times $\tau$. If $\tau$ is a stopping time, the proof can be obtained by an approximation procedure by discrete valued stopping times as usual.}).
So, we only prove that $v\in C_p(\Lambda)$.

The $p$-polynomial growth of $v$ is straightforward from
 Theorem~\ref{Exist_control},
Assumption~\ref{A:ell}\emph{(\ref{2016-07-08:02})},
and 
$\xi\in C_p(\mathbb{W})$.
Moreover, $v$ is clearly non-anticipative.
 By non-anticipativity, the continuity of $v$ with respect to $\textbf{d}_\infty$ is equivalent to the continuity  with respect to  $|\cdot|+|\cdot|_\infty$.
So,  letting $(t_n,\mathbf{x}_n)\to (t,\mathbf{x})$ with respect to $|\cdot|+|\cdot|_\infty$, we need to  prove that $v(t_n,\mathbf{x}_n)\to v(t,\mathbf{x})$.
Let $\epsilon>0$.
 For $n\in \mathbb{N}$, let $\mathbf{a}^{n,\epsilon}$ be an $\epsilon$-optimal control for $v(t_n,\mathbf{x}_n)$.
We can
write
\begin{equation}\label{2016-08-05:10}
  \begin{split}
 v(t_n,\mathbf{x}_n)-
   v(t,&\mathbf{x})
-\epsilon  \leq 
    J(t_n,\mathbf{x}_n,\mathbf{a}^{n,\epsilon})-
    J(t,\mathbf{x},\mathbf{a}^{n,\epsilon})\\
\leq&
   \E\bigg[
   \int_{t_n}^T \ell(s,X^{t_n,\mathbf x_n,\mathbf{a}^{n,\epsilon
}},\mathbf{a}^{n,\epsilon}_s)ds
-
 \int_{t}^T \ell(s,X^{t,\mathbf x,\mathbf{a}^{n,\epsilon}},\mathbf{a}^{n,\epsilon}_s)ds\bigg]\\
& +\mathbb{E}\bigg[ \xi(X^{t_n,\mathbf x_n,\mathbf{a}^{n,\epsilon}})
- \xi(X^{t,\mathbf x,\mathbf{a}^{n,\epsilon}})
\bigg].
  \end{split}
\end{equation}
By Theorem~\ref{Exist_control}, the family $\{X^{\cdot,\cdot,\mathbf{a}}\colon
[0,T]\times \mathbb{W}\rightarrow
 \mathcal{H}^p_\mathcal{P}(H)
\}_{\mathbf{a}\in \mathcal{U}}$ is equicontinuous.
Then,
passing to a subsequence
again denoted by $\{(t_n,\mathbf{x}_n,\mathbf{a}^{n,\epsilon})\}_{n\in \mathbb{N}}$
 if necessary,
\begin{equation*}
\lim_{n\rightarrow \infty}
|X^{t,\mathbf{x},\mathbf{a}^{n,\epsilon}}-
  X^{t_n,\mathbf{x}_n,\mathbf{a}^{n,\epsilon}}|_\infty
=0\qquad \mathbb{P}\mbox{-a.e.,}
\end{equation*}
hence, by
Assumption~\ref{A:ell}\emph{(\ref{2016-07-08:01})},
for all  $s\in[0,T]\setminus \{t\}$
\begin{equation}
  \label{eq:2016-08-05:09}
\lim_{n\rightarrow \infty}|\mathbf{1}_{[t_n,T]}(s)\ell(s,X^{t_n,\mathbf x_n,\mathbf{a}^{n,\epsilon
}},\mathbf{a}^{n,\epsilon}_s)
-
\mathbf{1}_{[t,T]}(s)\ell(s,X^{t,\mathbf x,\mathbf{a}^{n,\epsilon
}},\mathbf{a}^{n,\epsilon}_s)|
=0
\end{equation}
$\mathbb{P}\mbox{-a.e.}$,
and,
since $\xi\in C_p(\mathbb{W})$,
\begin{equation}
\label{axi}
\lim_{n\rightarrow \infty}|\xi(X^{t_n,\mathbf x_n,\mathbf{a}^{n,\epsilon
}})
-
\xi(X^{t,\mathbf x,\mathbf{a}^{n,\epsilon
}})|
=0\qquad \mathbb{P}\mbox{-a.e..}
\end{equation}
By Theorem~\ref{Exist_control}, the family $\{X^{t_n,\mathbf{x}_n,\mathbf{a}}\}_{\mathbf{a}\in \mathcal{U}}\subset \mathcal{H}^{p'}_\mathcal{P}(H)$ is bounded for any $p'>p$, hence it is uniformly integrable in $L^p((\Omega,\mathcal{F}_T,\mathbb{P}),\mathbb{W})$.
Then,
taking also into account
Assumption~\ref{A:ell}\emph{(\ref{2016-07-08:02})}
and  that $\xi\in C_p(\mathbb{W})$,
we can
pass to the limit in
\eqref{2016-08-05:10} and obtain
\begin{equation}
  \label{eq:2016-08-05:11}
  \limsup_{n\rightarrow \infty} v(t_n,\mathbf{x}_n)
 \leq   v(t,\mathbf{x})
+\epsilon.
\end{equation}
On the other hand, letting $\mathbf{a}^\epsilon$ be an $\epsilon$-optimal control for $v(t,\mathbf{x})$, we have
\begin{equation*}
 v(t_n,\mathbf{x}_n)-
   v(t,\mathbf{x})
+\epsilon  \geq
J(t_n,\mathbf{x}_n,\mathbf{a}^{\epsilon})-
    J(t,\mathbf{x},\mathbf{a}^{\epsilon})
\end{equation*}
and by arguing as above (here it is even simpler as $\mathbf{a}^\varepsilon$ is fixed), we obtain
\begin{equation}
  \label{eq:2016-08-05:12}
  \liminf_{n\rightarrow \infty} v(t_n,\mathbf{x}_n)\ge
   v(t,\mathbf{x})
 -\epsilon.
\end{equation}
Since $\epsilon>0$ was arbitrary,
\eqref{eq:2016-08-05:11}
and
\eqref{eq:2016-08-05:12} provide the continuity of $v$.
\end{proof}

\begin{Definition}\label{def:smooth_functions_control}
We say that $u\in C^{1,2}_{X}(\Lambda)$ if $u\in \cpol(\Lambda)$ and, for all $(t,\mathbf x)\in\Lambda$, $s\in[t,T]$, $\mathbf{a}\in\Uc$,
\begin{equation}
\label{eq:fct_Ito_control}
du(s,X^{t,\mathbf x,\mathbf{a}})  =  \bar\alpha(s,X^{t,\mathbf x,\mathbf{a}},\mathbf{a}_s)ds + \langle \bar\beta(s,X^{t,\mathbf x,\mathbf{a}},\mathbf{a}_s),dW_s\rangle_{K},
\end{equation}
for some
measurable functions
$\bar\alpha\colon\Lambda\times U\rightarrow\R$, $\bar\beta\colon\Lambda\times U\rightarrow K$, such that
$\{\bar \alpha (\cdot,\cdot,a)\}_{a\in U}$ is equicontinuous and
\[
|\bar\alpha(t,\mathbf x,a)| + |\bar\beta(t,\mathbf x,a)|_K  \leq  \bar M\big(1 + |\mathbf x|_\infty^p\big), \qquad \forall\,(t,\mathbf x)\in\Lambda,\
\forall\, a\in U,
\]
for some constants $\bar M\geq 0$, $p\geq 0$.
\end{Definition}

By identifying the finite variation part and the Brownian part in \eqref{eq:fct_Ito_control}, we see that $\bar\alpha$ and $\bar\beta$ in Definition \ref{def:smooth_functions_control} are uniquely determined.
 Given $u\in C^{1,2}_{X}(\Lambda)$, following \eqref{def:L} we denote
\begin{equation}\label{operatorL^a}
\Lc^a u(t,\mathbf x) \coloneqq   \bar\alpha(t,\mathbf x,a), \qquad \forall\,(t,\mathbf x)\in\Lambda,\ \forall\,a\in U.
\end{equation}

We now provide the definition of viscosity subsolution/supersolution of equation \eqref{eq:PPDE_fullynonlinear}. In order to do that, we redefine the two classes of test functions  $\underline{\Ac} u(t,\mathbf{x})$ and $\overline{\Ac} u(t,\mathbf{x})$ accordingly to the present controlled case. Given $u\in \cpol(\Lambda)$  we define, for $(t,\mathbf{x})\in \Lambda$, $t\in[0,T)$,
\begin{multline*}
\underline{\Ac} u(t,\mathbf{x}) \defeq  \Big\{ \varphi \in C^{1,2}_{X}(\Lambda) \colon \text{there exists }\H\in\mathcal{T},\ \H>t, \mbox{ such that,}\\
\mbox{for all }\mathbf{a}\in\Uc, 
(\varphi-u)(t,\mathbf{x}) = \min_{\tau \in \Tc,\,\tau\geq t}\E\big[ (\varphi-u)(\tau\wedge\H,X^{t,\mathbf{x},\mathbf{a}}) \big] \Big\},
\end{multline*}
\begin{multline*}
\overline{\Ac} u(t,\mathbf{x}) \defeq  \Big\{ \varphi \in C^{1,2}_{X}(\Lambda) \colon\text{there exists }\H\in\mathcal{T},\ \H>t, \mbox{ such that,}\\
\mbox{for all }\mathbf{a}\in\Uc, 
(\varphi-u)(t,\mathbf{x}) = \max_{\tau \in \Tc,\,\tau\geq t}\E\big[ (\varphi-u)(\tau\wedge\H,X^{t,\mathbf{x},\mathbf{a}}) \big] \Big\}.
\end{multline*}

\begin{Remark}
Notice that $\underline{\Ac} u(t,\mathbf{x})$ and $\overline{\Ac} u(t,\mathbf{x})$ can be  written in the following equivalent form:
\begin{multline*}
\underline{\Ac} u(t,\mathbf{x}) \defeq  \Big\{ \varphi \in C^{1,2}_{X}(\Lambda) \colon \text{there exists }\H\in\mathcal{T},\ \H>t, \mbox{ such that} \\
(\varphi-u)(t,\mathbf{x}) = \inf_{\mathbf{a}\in\Uc}\min_{\tau \in \Tc,\,\tau\geq t}\E\big[ (\varphi-u)(\tau\wedge\H,X^{t,\mathbf{x},\mathbf{a}}) \big] \Big\},
\end{multline*}
\begin{multline*}
\overline{\Ac} u(t,\mathbf{x}) \defeq  \Big\{ \varphi \in C^{1,2}_{X}(\Lambda) \colon\text{there exists }\H\in\mathcal{T},\ \H>t, \mbox{ such that} \\
(\varphi-u)(t,\mathbf{x}) = \sup_{\mathbf{a}\in\Uc}\max_{\tau \in \Tc,\,\tau\geq t}\E\big[ (\varphi-u)(\tau\wedge\H,X^{t,\mathbf{x},\mathbf{a}}) \big] \Big\}.
\end{multline*}
\end{Remark}

\begin{Definition} \label{def:visc_sol_PPDE_semi}
Let $u\in \cpol(\Lambda)$.
\begin{enumerate}[(i)]
\setlength\itemsep{0.05em}
\item 
\vskip-5pt
 We say that $u$ is a \textbf{viscosity subsolution} (resp.\ \textbf{supersolution}) of the path-dependent PDE~\eqref{eq:PPDE_fullynonlinear} if
\begin{equation*}
  - \sup_{a\in U}\big[\mathcal{L}^a \varphi(t,\mathbf{x})
+ \ell(t,\mathbf{x},a)\big]
\leq 0, \quad (\mbox{resp.}\ \ge 0)
\end{equation*}
for all $(t,\mathbf{x}) \in \Lambda$, $t<T$, and all $\varphi \in \Acb u(t, \mathbf{x})$ (resp.\ $\varphi \in \Acu u(t, \mathbf{x})$).
\item
 We say that $u$ is a \textbf{viscosity solution} of the path-dependent PDE~\eqref{eq:PPDE_fullynonlinear} if it is both a viscosity subsolution and a viscosity supersolution.
\end{enumerate}
\end{Definition}

\begin{Proposition}\label{T:Equivalence}
Let Assumptions \ref{A:SHDE}(\ref{2016-07-13B:00}), \ref{A:SHDE_control}, \ref{A:ell} hold, and let $\xi\in C_p(\mathbb{W})$, for some $p\geq 0$.
Then $v$ is a viscosity solution of PPDE \eqref{eq:PPDE_fullynonlinear}.
\end{Proposition}
\begin{proof}
\emph{Supersolution property.} Let $(t,\mathbf{x})\in\Lambda$ with $t<T$. By the Dynamic Programming Principle, we have, for every
 constant control ${\mathbf{a}}\equiv a\in U$ and $s\in(t,T]$,
\begin{equation}\label{eq:DDp}
v(t,\mathbf x)  \geq  \E\bigg[\int_t^s \ell(r,X^{t,\mathbf x,\mathbf{a}},a)dr + v(s,X^{t,\mathbf x,\mathbf{a}})\bigg].
\end{equation}
   Let $\varphi\in\overline{\mathcal{A}}v(t,\mathbf{x})$. Then, starting from \eqref{eq:DDp} and  arguing as in the proof of
Theorem \ref{lemma:pp1}, we get
 \begin{equation*}
  - \big[\mathcal{L}^a \varphi(t,\mathbf{x})
+ \ell(t,\mathbf{x},a)\big]
\ge 0, \ \ \ \forall\,a\in U.
\end{equation*}
Taking the infimum over $a\in U$, we get the claim.

\vspace{1mm}

\emph{Subsolution property.}
The following proof is inspired by the proof of the supersolution property in Theorem 3.66 of \cite{FGS14}. Let $(t,\mathbf{x})\in\Lambda$ with $t<T$ and $\varphi \in \Acb v(t, \mathbf{x})$. Let $\H\in\mathcal T$, $\H>t$, denote a stopping time associated to $\varphi$ as required in the definition of $\Acb v(t, \mathbf{x})$. Without loss of generality, we can suppose that $(\varphi-v)(t,\mathbf x)=0$. By  Dynamic Programming Principle, for every $\eps>0$ there exists a control $\mathbf a^\eps\in\mathcal U$ such that
\[
v(t,\mathbf{x})-\varepsilon^2  \leq  \E\bigg[\int_t^{(t+\eps)\wedge\H} \ell(r,X^{t,\mathbf{x},\mathbf{a}^\varepsilon},\mathbf{a}^\varepsilon_r)dr + v\big((t+\eps)\wedge\H,X^{t,\mathbf x,\mathbf{a}^\eps}\big)\bigg].
\]
This implies that
\[
\varphi(t,\mathbf{x})-\varepsilon^2  \leq  \E\bigg[\int_t^{(t+\eps)\wedge\H} \ell(r,X^{t,\mathbf{x},\mathbf{a}^\varepsilon},\mathbf{a}^\varepsilon_r)dr + \varphi\big((t+\eps)\wedge\H,X^{t,\mathbf x,\mathbf{a}^\eps}\big)\bigg].
\]
Since $\varphi \in C^{1,2}_{X}(\Lambda)$, by Definition \ref{def:smooth_functions_control} and \eqref{operatorL^a} we can write
\begin{equation}\label{proof_existence_eps^2}
-\varepsilon^2  \leq  \E\bigg[\int_t^{(t+\eps)\wedge\H} \big[
\mathcal{L}^{\mathbf a_r^\eps} \varphi(r,X^{t,\mathbf x,\mathbf{a}^\eps}) + \ell(r,X^{t,\mathbf{x},\mathbf{a}^\varepsilon},\mathbf{a}^\varepsilon_r)\big]dr\bigg].
\end{equation}
Notice that $\sup_{a\in U}\big[
\mathcal{L}^{a} \varphi(\cdot,\cdot) + \ell(\cdot,
\cdot,a)\big]$
is uniformly continuous, hence measurable, and then
\eqref{proof_existence_eps^2} implies
\begin{equation}\label{proof_existence_eps^2bis}
-\varepsilon^2  \leq  \E\bigg[\int_t^{(t+\eps)\wedge\H} \sup_{a\in U}\big[
\mathcal{L}^{a} \varphi(r,X^{t,\mathbf x,\mathbf{a}^\eps}) + \ell(r,X^{t,\mathbf{x},\mathbf{a}^\varepsilon},a)\big]dr\bigg].
\end{equation}
Now,
by equicontinuity of
$\{X^{\cdot,\cdot,\mathbf{a}}\colon [0,T]\times \mathbb{W}\rightarrow \mathcal{H}^p_\mathcal{P}(H)\}_{\mathbf{a}\in \mathcal{U}}$,
claimed in
 Theorem \ref{Exist_control},
we can write
\begin{equation}\label{limite}
  \begin{split}
          \lim_{\eps\rightarrow 0^+}\sup_{\mathbf a\in\mathcal U}\E&\bigg[\sup_{r\in[t,t+\eps]}\big|X_r^{t,\mathbf{x},\mathbf a} - \mathbf x_t\big|^p\bigg]
=\\
&=      \lim_{\eps\rightarrow 0^+}
\sup_{\mathbf a\in\mathcal U}\E\bigg[\sup_{r\in[t,t+\eps]}\big|X_r^{t,\mathbf{x},\mathbf a} - X_r^{t+\epsilon,\mathbf{x}_{t\wedge \cdot},\mathbf{a}}\big|^p\bigg]\\
&\leq
\lim_{\eps\rightarrow 0^+}
\sup_{\mathbf a\in\mathcal U}
\left|
X^{t,\mathbf{x}_{t\wedge \cdot},\mathbf a}
 -
 X^{t+\epsilon,\mathbf{x}_{t\wedge \cdot},\mathbf{a}}\right|_{\mathcal{H}^p_\mathcal{P}(H)}^p=0.
\end{split}
\end{equation}
Dividing \eqref{proof_existence_eps^2bis} by $\varepsilon$, letting $\varepsilon\to 0^+$,  using \eqref{limite},
and recalling the uniform continuity of
$\sup_{a\in U}\big[
\mathcal{L}^{a} \varphi(\cdot,\cdot) + \ell(\cdot,
\cdot,a)\big]$,
we conclude
$$
  \sup_{a\in U}
\{\mathcal{L}^a\varphi(t,\mathbf{x})
+\ell(t,\mathbf{x},a)
\}\geq 0.\eqno\qed
$$
\let\qed\relax
\end{proof}

\subsection{Partial comparison}
\label{SubS:PartialComp}

Hereafter, in this section, we assume that Assumptions~\ref{A:SHDE}\emph{(\ref{2016-07-13B:00})},
\ref{A:SHDE_control}, and
 \ref{A:ell}
hold. 

In order to prove a (partial) comparison result, we need a counterpart of Theorem \ref{teo:OS}. In this case, we would need to deal with a mixed optimal control/stopping problem. The problem of existence of solution for the latter is provided in \cite{RTZ14} passing through arguments strongly relying on the finite dimensionality of the problem, notably 
the local compactenss,
which fails in infinite dimension.  For this reason, here we leave out the treatment of such difficult issue and take such existence result as an assumption (Assumption \ref{2016-07-11:00} below; but it should be considered, rather, as a key middle step towards the comparison result).

\begin{Assumption}\label{2016-07-11:00}
Let $\phi\in \cpol(\Lambda)$ and let $s\in[0,T]$. Let $\Phi\colon \Lambda\rightarrow \mathbb{R}$ be  defined by
\begin{equation*}
  \Phi(t',\mathbf{x}')
\coloneqq
\sup_{\substack{
\tau\in \mathcal{T},\,\,\,  \tau \geq t'\\
\mathbf{a}\in \mathcal{U}
}
}  \mathbb{E} \left[ \phi(\tau,X^{t',\mathbf{x}',\mathbf{a}})
 \right]\qquad \forall\, (t',\mathbf{x}')\in\Lambda.
\end{equation*}
For every $(t,\mathbf{x})\in \Lambda$ there exist
$\tau^*\in \mathcal{T}$, with $ \tau^*\geq t$,  and $\mathbf{a}^*\in \mathcal{U}$, such that
\begin{equation}\label{2016-07-11:02}
  \begin{split}
    \Phi(t,\mathbf{x})&=
  \mathbb{E} \left[ \phi(\tau^*,X^{t,\mathbf{x},\mathbf{a}^*})
 \right]\\
\phi(\tau^*,X^{t,\mathbf{x},\mathbf{a}^*})&=
    \Phi(\tau^*,X^{t,\mathbf{x},\mathbf{a}^*})\qquad \mathbb{P}\mbox{-a.s.}
\end{split}
\end{equation}
\end{Assumption}

\begin{Proposition}[Partial comparison principle]\label{prop:PC}
Let Assumption
\ref{2016-07-11:00} hold.
 Let $u,v\in \cpol(\Lambda)$ be, respectively, a viscosity subsolution and a viscosity supersolution to
\eqref{eq:PPDE_fullynonlinear}.
If $u(T,\cdot)\leq v(T,\cdot)$  on $\mathbb{W}$ and either $u\in C^{1,2}_{X}(\Lambda)$ or $v\in C^{1,2}_{X}(\Lambda)$, then $u\leq v$ on $\Lambda$.
\end{Proposition}
\begin{proof}
Assume, by contradiction, that there exists $(\hat t,\mathbf{\hat x})\in \Lambda$  such that
  \begin{equation}\label{2016-07-13:01}
    q\coloneqq u(\hat t,\mathbf{\hat x})
    -
    v(\hat t,\mathbf{\hat x})>0.
  \end{equation}
By continuity, we can assume without loss of generality that $\hat t>0$.
Let $\epsilon>0$ be such that $q> 2\epsilon T$.
For $(t,\mathbf{x})\in \Lambda$, define
\begin{align*}
v_\epsilon(t,\mathbf{x})&\coloneqq
v(t,\mathbf{x})+\epsilon(T-\hat t)
\\
u_\epsilon(t,\mathbf{x})&\coloneqq
u(t,\mathbf{x})-\epsilon(T-\hat t)
\\
  \Phi^{(\epsilon)}(t,\mathbf{x})&\coloneqq
\sup_{\substack{
\tau\in \mathcal{T},\, t\leq \tau \leq T\\
\mathbf{a}\in \mathcal{U}
}
}  \mathbb{E} \left[
u(\tau,X^{t,\mathbf{x},\mathbf{a}})
-
v(\tau,X^{t,\mathbf{x},\mathbf{a}})
-\epsilon(T-\tau)
 \right].
\end{align*}
By Assumption~\ref{2016-07-11:00} applied to $\phi(t,\mathbf{x})\coloneqq u(t,x)-v(t,\mathbf{x})-\varepsilon(T-t)$,
there exist $\tau^*\in \mathcal{T}$, $t\leq \tau^*\leq T$, $\mathbf{a}^*\in \mathcal{U}$, such that
\begin{equation}\label{2016-07-13:00}
    \Phi^{(\epsilon)}(\hat t,\mathbf{\hat x})=
  \mathbb{E} \left[
 u(\tau^*,X^{\hat t,\mathbf{\hat x},\mathbf{a}^*})
-
 v(\tau^*,X^{\hat t,\mathbf{\hat x},\mathbf{a}^*})
-\epsilon (T-\tau^*)
 \right]
 \end{equation}
 and
 \begin{equation}\label{2016-07-13:00bis}
    \Phi^{(\epsilon)}(\tau^*,X^{t,\mathbf{x},\mathbf{a}^*})
\!=\!
u(\tau^*,X^{\hat t,\mathbf{\hat x},\mathbf{a}^*})
\!-\!
v(\tau^*,X^{\hat t,\mathbf{\hat x},\mathbf{a}^*})
\!-\!
\epsilon(T\!-\!
\tau^*)
\quad \mathbb{P}\mbox{-a.s.}
\end{equation}
By   \eqref{2016-07-13:01},
\eqref{2016-07-13:00},
and the fact that $u(T,	\cdot)-v(T,\cdot)\leq 0$ on $\mathbb{W}$,
 we have
$\mathbb{P}(\tau^* <T)>0$.
Combining with \eqref{2016-07-13:00}, we get the existence of  $\omega^*\in \Omega$ such that $\tau^*(\omega^*)<T$
and, setting $(t^*,\mathbf{x}^*)\coloneqq (\tau^*(\omega^*),X^{\hat t,\mathbf{\hat x},\mathbf{a}^*}(\omega^*))$,
\begin{multline}\label{2016-07-13:02}
u(t^*,\mathbf{x}^*)
-
v(t^*,\mathbf{x}^*)
-\epsilon(T-t^*)
=\Phi^\eps (t^*,x^*)\\
=
\sup_{\substack{
\tau\in \mathcal{T},\, t^*\leq \tau \leq T\\
\mathbf{a}\in \mathcal{U}
}
}  \mathbb{E} \left[
u(\tau,X^{t^*,\mathbf{x}^*,\mathbf{a}})
-
v(\tau,X^{t^*,\mathbf{x}^*,\mathbf{a}})
-\epsilon(T-\tau)
 \right].
\end{multline}

Now,
assume first that $v\in C^{1,2}_X(\Lambda)$.
In such a case,
\eqref{2016-07-13:02} entails $v_\epsilon\in
\underline{\Ac} u(t^*,\mathbf{x}^*)$.
By Definition \ref{def:visc_sol_PPDE_semi} of viscosity subsolution to \eqref{eq:PPDE_fullynonlinear}, we  have
\begin{multline}\label{2016-07-13:03}
 \epsilon t^* -\sup_{a\in U}
  [\mathcal{L}^a v(t^*,\mathbf{x}^*)+\ell(t^*,\mathbf{x}^*,a)]=\\
=
-  \sup_{a\in U}
  [\mathcal{L}^a v_\epsilon(t^*,\mathbf{x}^*)+\ell(t^*,\mathbf{x}^*,a)]\leq 0.
\end{multline}
Since
$v$ is a viscosity supersolution to \eqref{eq:PPDE_fullynonlinear}, we must have also
\begin{equation}\label{2016-07-13:04}
-  \sup_{a\in U}
  [\mathcal{L}^a v(t^*,\mathbf{x}^*)+\ell(t^*,\mathbf{x}^*,a)]\geq 0.
\end{equation}
Recalling that $t^*\geq \hat t>0$,
 \eqref{2016-07-13:03} and
 \eqref{2016-07-13:04} provides the  contradiction $\epsilon t^*\leq 0$.

Assume now that  $u\in C^{1,2}_X(\Lambda)$. Then
\eqref{2016-07-13:02} shows  that $u_\epsilon\in
\overline{\Ac} v(t^*,\mathbf{x}^*)$.
By definition of viscosity supersolution to \eqref{eq:PPDE_fullynonlinear}, we  have
\begin{multline}\label{2016-07-13:08}
    - \epsilon t^* -\sup_{a\in U}
  [\mathcal{L}^a u(t^*,\mathbf{x}^*)+\ell(t^*,\mathbf{x}^*,a)]
=\\
=
-  \sup_{a\in U}
  [\mathcal{L}^a u_\epsilon(t^*,\mathbf{x}^*)+\ell(t^*,\mathbf{x}^*,a)]\geq 0.
\end{multline}
Since
$u$ is a viscosity subsolution to \eqref{eq:PPDE_fullynonlinear}, we must have also
\begin{equation}\label{2016-07-13:09}
-  \sup_{a\in U}
  [\mathcal{L}^a u (t^*,\mathbf{x}^*)+\ell(t^*,\mathbf{x}^*,a)]\leq 0.
\end{equation}
We now conclude as in the previous case.
\end{proof}

\subsection{Further developments}
\label{SubS:Punctual}

When $H=\R^n$ and the path-dependent PDE \eqref{eq:PPDE_fullynonlinear} is semilinear --- corresponding to the case of the coefficient $\bar\sigma$ independent  of $a\in U$ --- and $\bar b$ satisfies the so-called \emph{structure condition}
\begin{equation}
\label{structure}
\bar b(t,\mathbf x,a)=\bar\sigma(t,\mathbf x)\bar{b}_0(t,\mathbf x,a)
\end{equation}
 then a proof of a comparison principle between viscosity sub/supersolutions is given in \cite{RTZ14}. This  proof  is inspired by the proof of the comparison principle \cite[Th.\,5.3]{CaffCabre}, which relies on the notion of punctual differentiability (see  \cite[Def.\,1.4.]{CaffCabre}) --- despite the usual proof of  the comparison principle in the framework of  viscosity solutions for second order PDEs, based on Ishii's lemma. This methodology seems to be implementable also in the present infinite-dimensional setting. We briefly recall and adapt to the present framework the main steps of the proof of \cite{RTZ14} leaving the argument at a descriptive level, as a  rigorous proof would go beyond the scopes of the present paper and is left for future research.

Let Assumptions
\ref{A:SHDE_control},  \ref{A:ell} and \ref{2016-07-11:00} hold.

\begin{enumerate}
\setlength\itemsep{0.05em}
\item 
\vskip-5pt
 In the proof of \cite{RTZ14}, the definition of viscosity solution in terms of jets is used. The PPDE is semilinear in this case and the definition of semijets needs to take into account also the term $\beta$ (see Remark \ref{rem:beta}). Precisely, following \cite{RTZ14}, given $u\in \cpol(\Lambda)$, we define the \emph{subjet} and \emph{superjet} of $u$ at $(t,\mathbf x)\in\Lambda$ as follows:
\begin{multline*}
\underline{\Jc} u(t,\mathbf{x})  \defeq  \big\{ (\alpha,\beta)\in \R\times H \colon  \exists\, \varphi\in\underline{\Ac} u(t,\mathbf{x}) \\
\mbox{such that }\varphi(s,\mathbf{y})=\alpha s + \langle\beta,\mathbf y_s\rangle,\ \forall\, (s,\mathbf{y})\in\Lambda \big \},
\end{multline*}
\vskip-1cm
\begin{multline*}
  \overline{\Jc} u(t,\mathbf{x})  \defeq  \big\{ (\alpha,\beta)\in \R\times H \colon  \exists\, \varphi\in\overline{\Ac} u(t,\mathbf{x}) \\
\mbox{such that }\varphi(s,\mathbf{y})=\alpha s + \langle\beta,\mathbf y_s\rangle,\ \forall\, (s,\mathbf{y})\in\Lambda \big \}.
\end{multline*}
Notice that, if $\varphi(s,\mathbf{y})=\alpha s + \langle\beta,\mathbf y_s\rangle$ for some $(\alpha,\beta)\in\R\times H$, then, using \eqref{structure}
,\begin{equation}\label{Lu_punctual}
\Lc^a\varphi(t,\mathbf x)  =  \alpha + \langle\beta,b(t,\mathbf x,a)\rangle  =  \alpha + \langle\bar\sigma^*(t,\mathbf x)\beta,\bar{b}_0(t,\mathbf x,a)\rangle_K,
\end{equation}
for all $(t,\mathbf x,a)\in\Lambda\times U$. Then, the main result of this first step would be the following equivalence (cf.\ \cite[Prop.\,3.8]{RTZ14}):
\begin{itemize}
\item[(R1)]
 \emph{$u\in C_p(\Lambda)$ is a viscosity subsolution (resp.\ supersolution) of the path-dependent PDE \eqref{eq:PPDE_fullynonlinear} if and only if:
  \begin{equation*}
    - \alpha - \sup_{a\in U}\big[\langle\bar\sigma^*(t,\mathbf x)\beta,\bar{b}_0(t,\mathbf x,a)\rangle_K + \ell(t,\mathbf x,a)\big]
\le 0, \quad (\mbox{resp.}\ \ge 0),
\end{equation*}
for every $(\alpha,\beta)\in\underline{\Jc} u(t,\mathbf{x})$ (resp.\ $(\alpha,\beta)
\in\overline{\Jc} u(t,\mathbf{x})$).}
\end{itemize}

\item The definition of viscosity solution in terms of jets can be  used in order to introduce the notion of \emph{punctual differentiability} (see  \cite[Def.\,3.10]{RTZ14}, inspired by  \cite[Def.\,1.4]{CaffCabre}). More precisely, \emph{given $u\in C_p(\Lambda)$, $p\geq 1$, and $(t,\mathbf x)\in\Lambda$, we say that $u$ is punctually $C^{1,2}_{X}(\Lambda)$ at $(t,\mathbf x)$ if
\[
\Jc u(t,\mathbf x)  \coloneqq   \textup{cl}\big(\underline{\Jc} u(t,\mathbf{x})\big) \cap \textup{cl}\big(\overline{\Jc} u(t,\mathbf{x})\big)  \neq  \emptyset,
\]
where $\textup{cl}(E)$ denotes the closure of the set $E\subset\R\times H$.}
\item \cite[Prop.\,4.17]{RTZ14} contains an important smoothness result, which in our context should be stated  as follows:
\begin{itemize}
\item[(R2)]\emph{Let $u\in C_p(\Lambda)$, $p\geq 1$, be a viscosity subsolution (or supersolution) of the path-dependent PDE \eqref{eq:PPDE_fullynonlinear}. Then, for every $(t,\mathbf x)\in\Lambda$ and $\mathbf{a}\in\Uc$, $u$ is punctually $C^{1,2}_{X}(\Lambda)$ at $dt\otimes\P^{t,\mathbf x,\mathbf{a}}(d\mathbf x)$-a.e.\ point $(s,\mathbf y)\in\Lambda$, where $dt$ denotes the Lebesgue measure on $[0,T]$, while the probability measure $\P^{t,\mathbf x,\mathbf{a}}$ on $\mathbb W$ denotes the law of the process $(X_s^{t,\mathbf x,\mathbf{a}})_{s\in[0,T]}$.}
\end{itemize} In this step it is strongly used the structure condition \eqref{structure}, which provides the equivalence of  all the probability measures $\{\P^{t,\mathbf x,\mathbf{a}}\}_ {\mathbf{a}\in\Uc}$.
\item
Using the notation of \cite{RTZ14}, we define:
\begin{equation}\label{F_control}
F(t,\mathbf x,z)  \coloneqq   \sup_{a\in U}\big[\langle z,\bar {b}(t,\mathbf x,a)\rangle_K + \ell(t,\mathbf x,a)\big], \  \forall\,(t,\mathbf x,z)\in\Lambda\times K.
\end{equation}
If $\overline b$ is bounded,
then
  $F$ is Lipschitz continuous with respect to the variable $z$, uniformly in  $(t,\mathbf{x})$.
Denote by $L_0\geq0$ the corresponding Lipschitz constant.
Then, the smoothness result stated at the previous point can be used  to prove the following  (cf.\ \cite[Prop.\,4.17]{RTZ14}):
\begin{itemize}
\item[(R3)] \emph{Let $p\geq 1$. If  $u^{(1)}\in C_p(\Lambda)$ (resp.\ $u^{(2)}\in C_p(\Lambda)$) is a viscosity subsolution (resp.\ supersolution) to PPDE \eqref{eq:PPDE_fullynonlinear}, then $w=u^{(1)}-u^{(2)}$ is a viscosity subsolution of the path-dependent PDE (cf.\ \cite[Eq.\,(4.10)]{RTZ14})
\begin{equation}\label{equation_w}
- \alpha - L_0|w(t,\mathbf x)| - L_0|\bar\sigma^*(t,\mathbf x)\beta|_K  \leq  0,
\end{equation}
for every $(\alpha,\beta)\in\underline\Jc w(t,\mathbf x)$, where $L_0$ is the Lipschitz constant of the function $F$ defined in \eqref{F_control}.}
\end{itemize}
\item Noticing that $w(T,\cdot)=u^{(1)}(T,\cdot)-u^{(2)}(T,\cdot)\leq0$ and that the identically null function is clearly a smooth supersolution to \eqref{equation_w}, we conclude, by the partial comparison principle (Theorem \ref{prop:PC}), that $u^{(1)}-u^{(2)}\leq0$ on $\Lambda$. This yields the \emph{comparison principle} for the path-dependent PDE \eqref{eq:PPDE_fullynonlinear}.
\end{enumerate}

\appendix
\renewcommand\thesection{}
\section{}
\renewcommand\thesection{\Alph{subsection}}
\renewcommand\thesubsection{\Alph{subsection}}

\newcounter{defcounter}
\setcounter{defcounter}{0}

\newenvironment{myequation}{\addtocounter{equation}{-1}
  \refstepcounter{defcounter}
  \renewcommand\theequation{A\thedefcounter}
  \begin{equation}} {\end{equation}}

\begin{proof}[\textbf{Proof of Theorem~\ref{teo:SW}}]
\emph{(\ref{2016-10-13:00})}
For the proof of the comparison principle 
 we refer to \cite{FGS14}, Theorem 3.50 (see also \cite[Theorem\ 3.2]{S94}).
We briefly
check that all the assumptions of Theorem 3.50 in \cite{FGS14} are satisfied.
The weak $B$-condition (3.2) in \cite{FGS14} corresponds to inequality \eqref{WeakBCondition}. Now, define the function $G\colon[0,T]\times H\times\mathbb R\times H\times S(H)\rightarrow\mathbb R$ as follows:
\[
G(t,x,r,p,X) = - \langle b(t,x),p\rangle
- \frac{1}{2}\text{Tr}\big[\sigma(t,x)\sigma^*(t,x)X\big] - F(t,x,r).
\]
Regarding Hypothesis 3.44 in \cite{FGS14}, 
assumptions \emph{(\ref{2016-10-13:02})}-\emph{(\ref{2016-10-13:01})} assure the uniform continuity of $\langle b(t,x),p\rangle$ and of $F(t,x,r)$ on bounded sets.
Moreover,
\begin{equation*}
  \begin{split}
    \big|\tr&[\sigma(t,x)\sigma^*(t,x)X] \!-\! \tr[\sigma(t',x')\sigma^*(t',x')X']\big| \notag \\
\leq & |(\sigma(t,x)\!-\!\sigma(t',x'))\sigma^*(t,x)X|_{L_1(H)} \!+\! |\sigma(t',x')(\sigma^*(t,x)\!-\!\sigma^*(t',x'))X|_{L_1(H)} \notag \\
&\! +\! |\sigma(t',x')\sigma^*(t',x')(X-X')|_{L_1(H)} \notag \\
\leq& |\sigma(t,x)\!-\!\sigma(t',x')|_{L_2(K;H)}
\big(|\sigma(t,x)|_{L_2(K;H)}
\!+\!
|\sigma(t',x')|_{L_2(K;H)}\big)
|X|_{L(H)} \\
& \!+\! |\sigma(t',x')|^2_{L_2(K;H)}|X\!-\!X'|_{L(H)}.
\end{split}
\end{equation*}
Then, using assumption
 \emph{(\ref{2016-10-13:03})}, we conclude that Hypothesis 3.44 in \cite{FGS14} is satisfied.

Hypothesis 3.45 in \cite{FGS14} with $\nu=0$ is guaranteed by \emph{(\ref{2016-10-13:04})}; Hypothesis 3.46 in \cite{FGS14} is straightforward to check;
Hypothesis 3.47 in \cite{FGS14} follows immediately from \emph{(\ref{2016-10-15:00})}.

Regarding Hypothesis 3.48 in \cite{FGS14}, 
let $R>0$,
$t\in[0,T]$,
 $x,x'\in X$, $|x|\leq R$, $|x'|\leq R$, $\epsilon>0$.
We have
\begin{myequation}\label{2016-10-13:07}
  \begin{split}
        \langle b(t,x')-b(t,x),B(x-x')\rangle
&\geq -|b(t,x)-b(t,x')|_{-1}|x-x'|_{-1}\\
&\geq -|B^{1/2}||b(t,x)-b(t,x')||x-x'|_{-1}\\
&\geq -2LR |b(\cdot,0)|_\infty |B||x-x'|_{-1},
\end{split}
\end{myequation}
where we used \emph{(\ref{2016-10-13:01})}.
Moreover, for $X,X'\in S(H)$ such that,
 for some $N\in \mathbb{N}\setminus \{0\}$,
 $X=P^*_NXP_N$, $X'=P^*_NX'P_N$, and
\[
-\frac{3}{\eps}\left( \begin{array}{cc}
BP_N & 0 \\
0 & BP_N\end{array} \right) \leq \left( \begin{array}{cc}
X & 0 \\
0 & -X'\end{array} \right) \leq \frac{3}{\eps}\left( \begin{array}{cc}
BP_N & -BP_N \\
-BP_N & BP_N\end{array} \right),
\]
we obtain, by standard computations,
\begin{myequation}\label{2016-10-13:09}
  \begin{split}
          \tr[\sigma(t,x')&\sigma^*(t,x')X'
-
 \sigma(t,x)\sigma^*(t,x)X]\geq \\
&\geq
 -\frac{3}{\epsilon}
|B||\sigma(t,x)-\sigma(t,x')|^2_{L_2(K;H)}
\geq 
 -\frac{3}{\epsilon}
L|B||x-x'|^2_{-1}.
\end{split}
\end{myequation}
Considering
\emph{(\ref{2016-10-13:06})},
\eqref{2016-10-13:07},
\eqref{2016-10-13:09},
 we see that 
Hypothesis 3.48 in \cite{FGS14} is satisfied.

Hypothesis 3.49 in \cite{FGS14} is satisfied with $\gamma=1$. Indeed,
for all $t\in[0,T]$, $x\in H$, $r\in \mathbb{R}$, $p,p'\in H$, $X,X'\in S(H)$, we can write
  \begin{equation*}
    \begin{split}
          G(t,x,r,p+p',X+X')
-&    G(t,x,r,p,X)\leq\\
&\leq
|b(t,x)||p'|+
\frac{1}{2}
|\tr [\sigma(t,x)\sigma^*(t,x)X']|\\
&\leq
 (M +L|x|)|p'|
+\frac{1}{2}|X'|_{L(H)}|\sigma(t,x)|^2_{L_2(K;H)}\\
&\leq
 (M +L|x|)|p'|
+\frac{1}{2}|X'|_{L(H)}
(M+|x|_{-1})
^2,
\end{split}
\end{equation*}
where
$M= \sup_{t\in[0,T]}\{|b(t,0)|+|\sigma(t,0)|_{L_2(K;H)}\}$.

Condition (3.75) in \cite{FGS14} can be easily checked by using \eqref{2016-10-14:00}.
Finally,
\eqref{2016-10-14:01} implies 
(3.77) in \cite{FGS14}.

\smallskip
\emph{(\ref{2016-10-14:02})}
For the proof of existence we 
 refer to Theorem 3.66 in \cite{FGS14},
whose assumptions are easy to verify in our case.
\end{proof}

\begin{proof}[\textbf{Proof of Theorem~\ref{Exist_control}}]
For the existence and uniqueness of the solution, and for the Lipschitz continuity of \eqref{eq:2014-10-29:aa_control} in $Z$, uniform in $(t,\mathbf{a})$,
see
 \cite[Th.\,3.6]{Roses2016a}.
For the
 continuity
in $(t,Z)$, for fixed $\mathbf{a}\in \mathcal{U}$,
see
 \cite[Th.\,3.14]{Roses2016a}.

We now show the equicontinuity of the family \eqref{eq:2016-08-05:04}.
  Due to the Lipschitz continuity in $Z$, uniform in $(t,\mathbf{a})$,
to prove the latter
\eqref{eq:2016-08-05:04} it is sufficient to prove the equicontinuity of the family
\begin{myequation}
  \label{eq:2016-08-05:05}
  \{X^{\cdot,Z,\mathbf{a}}\colon [0,T]
\rightarrow \mathcal{H}^p_\mathcal{P}(H)\}_{\mathbf{a}\in \mathcal{U}},
\end{myequation}
for every fixed $Z\in \mathcal{H}^p_\mathcal{P}(H)$.
Let $Z\in \mathcal{H}^p_\mathcal{P}(H)$, $0\leq t\leq t'\leq T$, $\Delta X^{t,t',Z,\mathbf{a}}_s\coloneqq X^{t,Z,\mathbf{a}}_s-X^{t',Z,\mathbf{a}}_s$, for $s\in[0,T]$ and $\mathbf{a}\in \mathcal{U}$.
First notice that, if $s\in [0,t]$, then $\Delta X^{t,t',Z,\mathbf{a}}_s=0$.
Moreover,
by using
the definition of mild solution
and by applying standard estimates
to the integrals
appearing in equality \eqref{eq:2016-08-05:06}, by means of the factorization formula (\cite[Lemma 1.114]{FGS14})
and of the Burkholder-Davis-Gundy inequality for the stochastic integral (\cite[Theorem 1.111]{FGS14}), we obtain
\begin{myequation}\label{2016-08-06:01}
\begin{split}
 \mathbb{E} \left[ \sup_{s\in[t,t']}
    |\Delta X^{t,t',Z,\mathbf{a}}_s|^p \right]&\leq
  C
   \left(
     \mathbb{E} \left[
       \sup_{s\in[t,t']}
       |Z_s-e^{(s-t)A}Z_t|^p \right]\right.\\
&\left.
\phantom{  \sup_{]}     }
     +
      \left(1+
      w(t'-t)
      \int_0^T |X_{r\wedge \cdot}^{t,Z,\mathbf{a}}|^p_{\mathcal{H}^p_\mathcal{P}(H)}dr\right)
   \right),
   \end{split}
\end{myequation}
where $C$ is a constant depending only on $T,p,M,\hat M,\gamma,\sup_{r\in[0,T]}|e^{rA}|_{L(H)}$, and $w$ is a modulus of continuity depending only on $p,\gamma$.

Moreover, by writing $X^{t,Z,\mathbf{a}}_s=X^{t',X_{t'\wedge \cdot}^{t,Z,\mathbf{a}},\mathbf{a}}_s$, for $s\in[t',T]$, and by recalling the uniform Lipschitz continuity of
\eqref{eq:2014-10-29:aa_control} with respect to $Z$, we have
\begin{myequation}
  \label{eq:2016-08-06:00}
    \mathbb{E} \left[ \sup_{s\in[t',T]}
    |\Delta X^{t,t',Z,\mathbf{a}}_s|^p \right]
  \leq \hat C |X^{t,Z,\mathbf{a}}_{t'\wedge \cdot}-Z_{t'\wedge \cdot}|^p_{\mathcal{H}^p_\mathcal{P}(H)},
\end{myequation}
where $\hat C$ is independent of $Z,\mathbf{a},t,t'$.
Notice that the right hand side  of
\eqref{eq:2016-08-06:00} can be estimated
through
\eqref{2016-08-06:01}.
We then finally obtain
{\begin{myequation}\label{2016-08-06:03}
\begin{split}
  |X^{t,Z,\mathbf{a}}-X^{t',Z,\mathbf{a}}|^p_{\mathcal{H}^p_\mathcal{P}(H)}
&\leq
  C(1+\hat C)
   \left(
     \mathbb{E} \left[
       \sup_{s\in[t,t']}
       |Z_s-e^{(s-t)A}Z_t|^p \right]\right.\\
&\left.\phantom{  \sup_{]}     }
     +
      \left(1+
      w(t'-t)
      \int_0^T |X_{r\wedge \cdot}^{t,Z,\mathbf{a}}|^p_{\mathcal{H}^p_\mathcal{P}(H)}dr\right)
   \right).
   \end{split}
\end{myequation}}
Estimate \eqref{2016-08-06:03} provides  the following information.
\begin{enumerate}[(a)]
\setlength\itemsep{0.05em}
\item \vskip-5pt
 Choosing $t'=T$ and applying Gronwall's inequality, we see that the set $\{X^{t,Z,\mathbf{a}}\}_{\substack{t\in[0,T],\mathbf{a}\in \mathcal{U},Z\in \mathbf{B}}}$ is
 bounded in $\mathcal{H}^p_\mathcal{P}(H)$ for each $\mathbf{B}\subset \mathcal{H}^p_\mathcal{P}(H)$ bounded.
\item Using the uniform boundedness just observed, we see from \eqref{2016-08-06:03} that, for fixed $Z$,
$$
\lim_{t'\rightarrow t^+}\sup_{\mathbf{a}\in \mathcal{U}}
  |X^{t,Z,\mathbf{a}}-X^{t',Z,\mathbf{a}}|_{\mathcal{H}^p_\mathcal{P}(H)}=0,
$$
which provides the desired
equicontinuity of
\eqref{eq:2016-08-05:05} and concludes the proof.\qedhere
\end{enumerate}
\end{proof}

\section*{Acknowledgements}
The authors thank the anonymous Referee for his/her useful remarks that helped to  improve  the paper.



\begin{thebibliography}{9}

\bibitem{Baldi2000} P.\ Baldi. \emph{Equazioni differenziali stocastiche e applicazioni}. Quaderni dell'Unione Matematica Italiana, Pitagora, 2000.

\bibitem{BarlesLesigne} G.\ Barles and L.\ Lesigne. 
SDE, BSDE and PDE.
In: El Karoui N., Mazliak L., (Eds.),
\emph{Backward Stochastic differential Equatons}.
Pitman Research Notes in Mathematics Series,
\textbf{364}, 47--80, 1997.

\bibitem{BS} E.\ Bayraktar and M.\ Sirbu. Stochastic Perron's method and verification without smoothness using viscosity comparison: the linear case.
 \emph{Proceedings of the American Mathematical Society},
\textbf{140}, pp.\ 3645--3654, 2012.

\bibitem{CaffCabre} L.A.\ Caffarelli and X.\ Cabr\'e.
\emph{Fully nonlinear elliptic equations}. American Mathematical Society Colloquium Publications, \textbf{43}. American Mathematical Society, Providence, RI, 1995.

\bibitem{Chojnowska78} A.\ Chojnowska-Michalik. 
Representation theorem for general stochastic delay equations.
 \emph{Bull.\ Acad.\ Polon.\ Sci.\ S\'er.\ Sci.\ Math.\ Astronom.\ Phys.}, \textbf{26}, 635--642, 1978.

\bibitem{CF10a} 
R.\ Cont and D.-A.\ Fourni\'e.
A functional extension of the It\^o formula.
\emph{C.\ R.\ Math.\ Acad.\ Sci.\ Paris}, \textbf{348}, 57--61, 2010.

\bibitem{CF10b} R.\ Cont and D.-A.\ Fourni\'e.
Change of variable formulas for non-anticipative functionals on path space.
 \emph{Journal of Functional  Analysis}, \textbf{259}, 1043--1072, 2010.

\bibitem{CF13} R.\ Cont and D.-A.\ Fourni\'e.
Functional It\^o calculus and stochastic integral representation of martingales.
 \emph{Annals of Probability}, \textbf{41}, 109--133, 2013.


\bibitem{CF16} R.\ Cont and D.-A.\ Fourni\'e.
Functional Kolmogorov equations.  In: V. Bally, L. Caramellino, R.  Cont, \emph{Stochastic integration by parts and Functional Ito calculus}, 183-207, Springer, 2016. 




\bibitem{CR14} A.\ Cosso and F.\ Russo.
Functional It\^o versus Banach space stochastic calculus and strict solutions of semilinear path-dependent equations.
\emph{Infin. Dimens. Anal. Quantum Probab. Relat. Top.}, \textbf{19}, 4, 2016.



\bibitem{DPZ92} G.\ Da Prato and J.\  Zabczyk.
\emph{Stochastic equations in infinite dimensions}.
 Encyclopedia of Mathematics and its Applications, \textbf{152}, Cambridge University Press, 2014.

\bibitem{DPZ02} G.\ Da Prato and J.\ Zabczyk.
\emph{Second order partial differential equations in Hilbert spaces}. Cambridge University Press, 2002.

\bibitem{DGR1}  C.\ Di Girolami and F.\  Russo.
Generalized covariation for Banach space valued processes and It\^o formula.
    \emph{Osaka Journal of Mathematics}, \textbf{51}, 2014.

   \bibitem{DGR2} C.\ Di Girolami and F.\ Russo.
Generalized covariation and extended Fukushima decompositions for Banach space valued processes. Application to windows of Dirichlet processes.
    \emph{Infinite Dimensional Analysis, Quantum Probability and Related Topics (IDA-QP)}, \textbf{15}, 2012.

\bibitem{Dupire}
B.\ Dupire.
Functional It\^o calculus.
\emph{Portfolio Research Paper}, Bloomberg, 2009.

\bibitem{EkrenKellerTouziZhang}
I.\ Ekren, C.\ Keller, N.\ Touzi, and J.\ Zhang.
On viscosity solutions of path-dependent PDEs.
 \emph{The Annals of Probability}, \textbf{42}, 1, 204--236, 2014.

\bibitem{EkrenTouziZhang1}
I.\ Ekren, N.\ Touzi, and J.\ Zhang.
Viscosity solutions of fully nonlinear parabolic path-dependent PDEs: Part I.
\emph{The Annals of Probability}, \textbf{44}, 2, 1212--1253, 2016.

\bibitem{EkrenTouziZhang2} 
I.\ Ekren, N.\ Touzi, and J.\ Zhang.
Viscosity Solutions of Fully Nonlinear Path-dependent PDEs: Part II. \emph{The Annals of Probability}, \textbf{44}, 4, 2507--2553, 2016.

\bibitem{FGS14} 
G.\ 
Fabbri, F.\  Gozzi, and A.\ \'{S}wi\c{e}ch.
\emph{Stochastic Optimal Control in Infinite Dimensions: Dynamic Programming and HJB Equations}.
Forthcoming in Springer series on ``Probability Theory and Stochastic Modelling'', Vol 82, 2017.


\bibitem{Ffinsto}
S.\ Federico.
A stochastic control problem with delay arising in a pension fund model. 
\emph{Finance \& Stochastics}, \textbf{15}, 3,  pp.\ 421--459, 2011.

\bibitem{FedericoGozzi16}
S.\ Federico and F.\ Gozzi.
Mild solution of semilinear elliptic equations in {H}ilbert spaces.  \emph{Journal of Differential Equations}, \textbf{262}, 5, 3343--3389, 2017.

\bibitem{FZ} 
F.\ Flandoli and G.\ Zanco.
An infinite-dimensional approach to path-dependent Kolmogorov's equations. \emph{The Annals of Probability}, \textbf{44}, 4, 2643--2693, 2016.
	

\bibitem{FMT} 
M.\ Fuhrman, F.\ Masiero, and G.\ Tessitore.
Stochastic Equations with Delay: Optimal Control via BSDEs and Regular Solutions of Hamilton-Jacobi-Bellman Equations.
\emph{SIAM Journal on Control and Optimization}, 
\textbf{48}, 7, pp.\ 4624--4651, 2010.

\bibitem{fuhrmantess02}
M.\ Fuhrman and G.\ Tessitore.
Nonlinear Kolmogorov Equations in Infinite Dimensional Spaces: The Backward Stochastic Differential Equations Approach and Applications to Optimal Control.
 \emph{The Annals of Probability}, \textbf{30}, pp.\ 1397--1465, 2002.

\bibitem{GM10} 
L.\ Gawarecki and V.\ Mandrekar.
 \emph{Stochastic Differential Equations in Infinite Dimensions}. Springer, 2010.

\bibitem{GRS}
F.\ Gozzi, E.\  Rouy, and A.\ Swiech.
Second order
Hamilton-Jacobi equations in Hilbert spaces and Stochastic Boundary Control.
\emph{SIAM Journal on Control and Optimization},
{\bf 38}, 2, pp.\ 400--430, 2000.

\bibitem{GSZakai}
F.\ Gozzi and A.\ Swiech.
Hamilton-Jacobi-Bellman equations for the optimal control of
the Duncan-Mortensen-Zakai equations.
 \emph{Journal of Functional Analysis}, \textbf{172}, pp.\ 466--510, 2000.

\bibitem{GozSw} 
F.\ Gozzi,
S.S.\ Sritharan,
and A.\ Swiech.
Bellman equations associated to the optimal feedback control of stochastic Navier-Stokes equations.
 \emph{Communications on Pure
and Applied Mathematics}, \textbf{58}, pp.\ 671--700, 2005.

\bibitem{Kel} 
D.\ Kelome and A.\ Swiech.
Viscosity solution of an infinite-dimensional Black-Scholes-Barenblatt equation.
\emph{Applied Mathematics and Optimization}, \textbf{47}, 3, pp.\ 253--278, 2003.

\bibitem{Lio1}
P.-L.\ Lions.
Viscosity solutions of fully nonlinear second-order equations and optimal
stochastic control in infinite dimensions. Part I: The case of bounded stochastic evolution.
\emph{Acta Mathematica}, \textbf{161}, pp.\ 243--278, 1988.

\bibitem{Lio2}
P.-L.\ Lions.
Viscosity solutions of fully nonlinear second-order equations and optimal
stochastic control in infinite dimensions.
Part II: Optimal Control of Zakai's equation.
\emph{Lecture Notes
in Mathematics}, \textbf{1390}, eds.\ G.\ Da Prato \& L.\ Tubaro, Springer-Verlag, Berlin, pp.\ 147--170, 1989.

\bibitem{Lio3}
P.-L.\ Lions.
Viscosity solutions of fully nonlinear second-order equations and optimal
stochastic control in infinite dimensions.
 Part III: Uniqueness of viscosity solutions for general second order equations.
\emph{Journal of  Functional Analysis}, \textbf{86}, pp.\ 1--18, 1989.

\bibitem{pardouxpeng} 
E.\ Pardoux and S.\ Peng.
Backward stochastic differential equations and quasilinear parabolic partial differential equations.
In 
\emph{Stochastic partial differential equations and their applications}, Springer Berlin Heidelberg, 1992.

\bibitem{M95}
P.\ Malliavin.
\emph{Integration and Probability}, Springer, 1995.


\bibitem{MasieroEJP07}
F.\ Masiero.
Regularizing properties of transition semigroups and semilinear parabolic equations in Banach spaces.
\emph{Electronic Journal of Probability},
\textbf{12},
 pp.\ 387--419, 2007.


\bibitem{mohammed84} 
S.E.A.\ 
Mohammed.
 \emph{Stochastic functional differential equations}.
Research Notes in Mathematics, \textbf{99}, Pitman (Advanced Publishing Program), Boston, 1984.

\bibitem{PS} 
G.\ Peskir  and A.\ Shirayev.
 \emph{Optimal stopping and free-boundary problems}.
  Birkh\"auser, Basel, 2006.

\bibitem{RTZ14} 
Z.-J. Ren,
N.\ Touzi, and J.\ Zhang.
Comparison of viscosity solutions of semi-linear path-dependent PDEs.
Preprint arXiv:1511.05910v1, 2015.

\bibitem{Ren} 
M.\ Renardy.
Polar decomposition of positive operators and a problem of Crandall and Lions.
\emph{Applicable Analysis}, \textbf{57}, 383--385, 1995.

\bibitem{Roses2016a} 
M.\ Rosestolato.
Path-dependent {SDE}s in {H}ilbert spaces.
 Preprint arXiv:1606.06321, 2016.

\bibitem{Roses2016b} 
M.\ Rosestolato.
Functional It\=o calculus in Hilbert spaces and application to path-dependent Kolmogorov equations.
 Preprint arXiv: arXiv:1606.06326, 2016.

\bibitem{S94} 
A.\ \'{S}wi\c{e}ch.
``Unbounded'' Second Order Partial Differential Equations in Infinite Dimensional Hilbert Spaces.
\emph{Communications in Partial Differential Equations}, \textbf{19}, pp.\ 1999--2036, 1994.

\bibitem{TZ15} 
S.\ Tang and F.\ Zhang.
Path-dependent optimal stochastic control and viscosity solution of associated Bellman equations.
\emph{Discrete and Continuous Dynamical Systems},
\textbf{35}, 11, pp.\,5521--5553, 2015.

\end{thebibliography}
\end{document}